\documentclass[%
  a4paper,
  onecolumn,
]{mypreprint}

\usepackage[english]{babel}

\usepackage{amsmath}
\usepackage{amssymb}
\usepackage{amsthm}

\usepackage{dsfont}

\usepackage{enumitem}

\usepackage{tikz}
  \usetikzlibrary{positioning}

\usepackage{pgfplots}
  \pgfplotsset{compat = 1.13,
    colormap name = viridis,
    unbounded coords = jump
  }
  \usepgfplotslibrary{external}
  \tikzset{external/system call = {%
    pdflatex \tikzexternalcheckshellescape
      -halt-on-error
      -interaction=batchmode
      -jobname "\image" "\texsource"}}
  \tikzexternaldisable

\usepackage{caption}
\usepackage{subcaption}
  
\definecolor{matlabBlue}{HTML}{0072BD}
\definecolor{matlabPurple}{HTML}{7E2F8E}
\definecolor{matlabGreen}{HTML}{77AC30}
\definecolor{matlabBrown}{HTML}{A2142F}

\pgfplotscreateplotcyclelist{plotlist}{
  {black, solid, line width = 2.5pt},
  {matlabBlue, solid, line width = 1.75pt},
  {matlabPurple, dashed, line width = 1.75pt},
  {matlabGreen, dash dot, line width = 1.75pt},
  {matlabBrown, dotted, line width = 1.75pt}
}

\usepackage{todonotes}

\renewcommand{\rm}[1]{\ensuremath{\mathrm{#1}}}
\renewcommand{\sf}[1]{\ensuremath{\mathsf{#1}}}

\newcommand{\trans}{\ensuremath{\mkern-1.5mu\mathsf{T}}}
\newcommand{\herm}{\ensuremath{\mathsf{H}}}

\newcommand{\R}{\ensuremath{\mathbb{R}}}
\newcommand{\C}{\ensuremath{\mathbb{C}}}

\renewcommand{\i}{\ensuremath{\mathfrak{i}}}

\newcommand{\cB}{\ensuremath{\mathcal{B}}}
\newcommand{\cC}{\ensuremath{\mathcal{C}}}

\newcommand{\cK}{\ensuremath{\mathcal{K}}}
\newcommand{\cN}{\ensuremath{\mathcal{N}}}

\newcommand{\hcB}{\ensuremath{\widehat{\mathcal{B}}}}
\newcommand{\hcC}{\ensuremath{\widehat{\mathcal{C}}}}
\newcommand{\hcK}{\ensuremath{\widehat{\mathcal{K}}}}
\newcommand{\hcN}{\ensuremath{\widehat{\mathcal{N}}}}
\newcommand{\hG}{\ensuremath{\widehat{G}}}
\newcommand{\hy}{\ensuremath{\hat{y}}}
\newcommand{\hb}{\ensuremath{\hat{b}}}

\newcommand{\tcN}{\ensuremath{\widetilde{\mathcal{N}}}}
\newcommand{\tG}{\ensuremath{\widetilde{G}}}
\newcommand{\tu}{\ensuremath{\tilde{u}}}
\newcommand{\ty}{\ensuremath{\tilde{y}}}
\newcommand{\htG}{\ensuremath{\widehat{\tG}}}

\newcommand{\margs}[2]{\ensuremath{(#1\,|\,#2)}}
\newcommand{\mGs}{\ensuremath{\boldsymbol{\mathsf{G}}}}
\newcommand{\hmGs}{\ensuremath{\boldsymbol{\mathsf{\hG}}}}

\newcommand{\mG}[3]{\ensuremath{\mGs_{#1}\margs{#2}{#3}}}
\newcommand{\hmG}[3]{\ensuremath{\hmGs_{#1}\margs{#2}{#3}}}
\newcommand{\mGf}[1]{\ensuremath{\mGs_{1}(#1)}}
\newcommand{\hmGf}[1]{\ensuremath{\hmGs_{1}(#1)}}
\newcommand{\mN}[2]{\ensuremath{\boldsymbol{\sf{N}}\margs{#1}{#2}}}
\newcommand{\hmN}[2]{\ensuremath{\boldsymbol{\sf{\widehat{N}}}\margs{#1}{#2}}}

\DeclareMathOperator{\mspan}{span}

\DeclareMathOperator{\err}{err}

\newcommand{\mtxint}{\sf{MtxInt}}
\newcommand{\bwtint}{\sf{BwtInt}}
\newcommand{\sftint}{\sf{SftInt}}
\newcommand{\sttint}{\sf{SttInt}}

\theoremstyle{plain}\newtheorem{theorem}{Theorem}
\theoremstyle{plain}\newtheorem{corollary}{Corollary}
\theoremstyle{definition}\newtheorem{remark}{Remark}
\theoremstyle{definition}\newtheorem{problem}{Problem}

\Crefname{problem}{Problem}{Problems}

\begin{document}

%%%%%%%%%%%%%%%%%%%%%%%%%%%%%%%%%%%%%%%%%%%%%%%%%%%%%%%%%%%%%%%%%%%%%%%%%%%%%%%%
% PAPER INFORMATION.                                                           %
%%%%%%%%%%%%%%%%%%%%%%%%%%%%%%%%%%%%%%%%%%%%%%%%%%%%%%%%%%%%%%%%%%%%%%%%%%%%%%%%

\title{A unifying framework for tangential interpolation of structured bilinear
  control systems}
  
\author[$\ast$]{Peter Benner}
\affil[$\ast$]{Max Planck Institute for Dynamics of Complex Technical
  Systems, Sandtorstra{\ss}e 1, 39106 Magdeburg, Germany.
  \email{benner@mpi-magdeburg.mpg.de}, \orcid{0000-0003-3362-4103},
  \authorcr \itshape
  Otto von Guericke University Magdeburg, Faculty of Mathematics,
  Universit{\"a}tsplatz 2, 39106 Magdeburg, Germany.
  \email{peter.benner@ovgu.de}
}

\author[$\ddagger$]{Serkan Gugercin}
\affil[$\ddagger$]{%
  Department of Mathematics and Division of Computational Modeling and Data
  Analytics, Academy of Data Science, Virginia Tech,
  Blacksburg, VA 24061, USA.\authorcr
  \email{gugercin@vt.edu}, \orcid{0000-0003-4564-5999}
}

\author[$\dagger$]{Steffen W. R. Werner}
\affil[$\dagger$]{Courant Institute of Mathematical Sciences, New York
  University, New York, NY 10012, USA.\authorcr
  \email{steffen.werner@nyu.edu}, \orcid{0000-0003-1667-4862}
}

\shorttitle{Structured bilinear tangential interpolation}
\shortauthor{Benner, Gugercin, Werner}
\shortdate{2022-06-03}
\shortinstitute{}
  
\keywords{%
  model order reduction,
  bilinear systems,
  tangential interpolation,
  structure-pre\-ser\-ving approximation
}

\msc{%
  30E05, % Moment problems and interpolation problems in the complex plane
  34K17, % Transformation and reduction of functional-differential equations
  65D05, % Numerical interpolation
  93A15, % Large-scale systems
  93C10  % Nonlinear systems in control theory
}
  
\abstract{%
  In this paper, we consider the structure-preserving model order reduction
  problem for multi-input/multi-output bilinear control systems by tangential 
  interpolation.
  We propose a new type of tangential interpolation problem for structured
  bilinear systems, for which we develop a new structure-preserving
  interpolation framework.
  This new framework extends and generalizes different formulations of
  tangential interpolation for bilinear systems from the literature and also
  provides a unifying framework.
  We then derive explicit conditions on the projection spaces to
  enforce tangential interpolation in different settings, including
  conditions for tangential Hermite interpolation.
  The analysis is illustrated by means of three numerical examples.
}

\novelty{%
  We propose a new formulation of the tangential interpolation problem for
  structured multi-input/multi-output bilinear control systems.
  We formulate conditions on the projection spaces to enforce
  structure-preserving tangential interpolation in this new framework, which
  also allow for Hermite interpolation, and generalize established formulations
  of tangential interpolation from the literature.
}

\maketitle

%%%%%%%%%%%%%%%%%%%%%%%%%%%%%%%%%%%%%%%%%%%%%%%%%%%%%%%%%%%%%%%%%%%%%%%%%%%%%%%%
% INTRODUCTION.                                                                %
%%%%%%%%%%%%%%%%%%%%%%%%%%%%%%%%%%%%%%%%%%%%%%%%%%%%%%%%%%%%%%%%%%%%%%%%%%%%%%%%

\pagebreak

\section{Introduction}%
\label{sec:intro}

Modeling of various real-world applications, e.g., biological, electrical or
population dynamics,  results in bilinear control
systems~\cite{Moh70, Moh73, AlbFB93, SapSH19, QiaZ14}.
Those bilinear systems usually inherit special structures based on their
underlying physical meaning.
For example, in the case of bilinear mechanical models, one obtains a
second-order bilinear control system of the form
\begin{align} \label{eqn:bsosys}
  \begin{aligned}
    M \ddot{q}(t) + D \dot{q}(t) + K q(t) & = \sum\limits_{j = 1}^{m}
      N_{\rm{p},j}q(t)u_{j}(t) + \sum\limits_{j = 1}^{m}
      N_{\rm{v},j}\dot{q}(t)u_{j}(t) + B_{\rm{u}}u(t),\\
    y(t) & = C_{\rm{p}}q(t) + C_{\rm{v}}\dot{q}(t),
  \end{aligned}
\end{align}
where $q(t) \in \R^{n} $ are the (internal) degrees of freedom;
$u(t) \in \R^{n}$ and  $y(t) \in \R^{p}$ are, respectively, the inputs and
outputs of the system;
$M, D, K, N_{\rm{p},j}, N_{\rm{v},j} \in \R^{n \times n}$ for all
$j = 1, \ldots, m$, $B_{\rm{u}} \in \R^{n \times m}$ and
$C_{\rm{p}}, C_{\rm{v}} \in \R^{p \times n}$.
Due to the usual request for high-fidelity modeling, the number
of differential equations, $n$, describing the dynamics of systems
as in~\cref{eqn:bsosys}, quickly increases.
This often results in a high demand for computational resources such as time
and memory.
One remedy is model order reduction: a new, \emph{reduced}, system is created, 
consisting of a significantly smaller number of differential equations than the 
original one while still accurately approximating the input-to-output behavior.
Then one can use this lower order approximation as a surrogate model for faster 
simulations or within algorithms for design optimization and controller
synthesis.

In the case of \emph{unstructured} bilinear systems with the state-space form
\begin{align} \label{eqn:bsys}
  \begin{aligned}
    E \dot{x}(t) & = A x(t) + \sum\limits_{j = 1}^{m}N_{j}x(t)u_{j}(t)
      + B u(t),\\
    y(t) & = C x(t),
  \end{aligned}
\end{align}
where $E, A, N_{j} \in \R^{n \times n}$ for all $j = 1, \ldots, m$,
$B \in \R^{n \times m}$ and $C \in \R^{p \times n}$,
there already exist different model reduction methodologies, 
e.g., bilinear balanced truncation~\cite{HsuDC83, AlbFB93, BenD11},
different types of interpolation approaches for the underlying
multi-variate transfer functions in the frequency
domain~\cite{BaiS06, ConI07, FenB07, BreD10, AntBG20}, 
complete Volterra series interpolation~\cite{ZhaL02, BenB12, FlaG15},
and the bilinear Loewner framework~\cite{AntGI16, GosPBetal19}.
For \emph{structured} bilinear control systems as in~\cref{eqn:bsosys},
recently~\cite{BenGW21a} developed the structure-preserving interpolation
framework where interpolation for multi-input/multi-output (MIMO) systems was
enforced as for single-input/single-output (SISO) systems, i.e., using full
matrix interpolation.
One of our major contributions in this paper is to devise a proper interpolation
framework for structured MIMO systems.

Reduction of MIMO bilinear systems is an intricate problem and only a few of
the aforementioned approaches provide suitable extensions  for model reduction
of MIMO structured bilinear systems.
The lack of a proper extension  is especially persistent for subsystem
interpolation since enforcing matrix interpolation results in quickly increasing
reduced-order dimension.
For MIMO linear dynamical systems, the concept of tangential interpolation 
resolves this issue~\cite{GalVV04} by interpolating the matrix-valued
transfer function along selected direction vectors. 
It is important to note that the optimal approximation in a specific, namely
the $\mathcal{H}_{2}$-norm, satisfies tangential interpolation, not matrix
interpolation,~\cite{AntBG20}.
For interpolatory model reduction of bilinear systems, it is not clear so far
what the  proper extension of tangential interpolation would be.
For the unstructured case, \cite{BenBD11, RodGB18}~provide one potential
extension where only certain blocks of the subsystem transfer functions are
employed in tangential interpolation.
In this paper, we will introduce a new unifying framework for tangential
interpolation of structured bilinear systems, inspired by the original ideas of
tangential interpolation for matrix-valued functions~\cite{BalGR90}.
This new framework will cover different extensions of tangential interpolation
to bilinear systems under one umbrella.
Especially, it will allow us to formulate a direct extension of the ideas
from~\cite{BenBD11, RodGB18} to the structured system case.

Parts of the theoretical results presented here were derived in the course of
writing the dissertation of the corresponding author~\cite{Wer21}.

The rest of the paper is organized as follows:
In \Cref{sec:basics}, we briefly recall the theory of bilinear systems and
Volterra series, introduce the structured transfer functions considered in this
paper and revisit the tangential interpolation problem for linear dynamical
systems.
In \Cref{sec:tang}, we will motivate our new unifying tangential interpolation
framework and provide conditions on underlying projection spaces to satisfy
interpolation conditions in this framework.
Three benchmark examples are presented in \Cref{sec:examples} that illustrate
the established theory by comparing different interpolatory model reduction 
approaches for (structured) MIMO bilinear systems, followed by the conclusions
in \Cref{sec:conclusions}.

%%%%%%%%%%%%%%%%%%%%%%%%%%%%%%%%%%%%%%%%%%%%%%%%%%%%%%%%%%%%%%%%%%%%%%%%%%%%%%%%
% SETUP.                                                                       %
%%%%%%%%%%%%%%%%%%%%%%%%%%%%%%%%%%%%%%%%%%%%%%%%%%%%%%%%%%%%%%%%%%%%%%%%%%%%%%%%

\section{Mathematical preliminaries}%
\label{sec:basics}

In this section, we briefly review various system-theoretic concepts for
bilinear systems and the idea of tangential interpolation for linear systems.

%%%%%%%%%%%%%%%%%%%%%%%%%%%%%%%%%%%%%%%%%%%%%%%%%%%%%%%%%%%%%%%%%%%%%%%%%%%%%%%%

\subsection{Frequency-domain representation of structured bilinear systems}

For the unstructured bilinear system~\cref{eqn:bsys}, define
$N = \begin{bmatrix} N_{1} & \ldots & N_{m} \end{bmatrix}$.
Assume $E$  to be invertible and zero initial conditions, i.e., $x(0) = 0$.
Then, the output of~\cref{eqn:bsys} can be expressed, under some mild
assumptions, in terms of a Volterra series~\cite{Rug81}, i.e.,
\begin{align*}
  y(t) & = \sum\limits_{k = 1}^{\infty} \int\limits_{0}^{t}
    \int\limits_{0}^{t_{1}} \ldots \int\limits_{0}^{t_{k-1}}
    g_{k}(t_{1}, \ldots, t_{k}) 
    \left( u(t - \sum\limits_{i = 1}^{j}t_{i}) \otimes \cdots \otimes
    u(t - t_{1}) \right) \mathrm{d}t_{k} \cdots \mathrm{d}t_{1},
\end{align*}
where $g_{k}$, for $k \geq 1$, is the $k$-th regular Volterra kernel given by
\begin{align} \label{eqn:voltime}
  \begin{aligned}
    g_{k}(t_{1}, \ldots, t_{k}) & = Ce^{E^{-1}At_{k}} \left(
      \prod\limits_{j = 1}^{k-1} (I_{m^{j-1}} \otimes E^{-1}N)
      (I_{m^{j}} \otimes e^{E^{-1}At_{k-j}}) \right)\\
    & \quad{}\times{}
      (I_{m^{k-1}} \otimes E^{-1}B),
  \end{aligned}
\end{align}
where $I_{m^{j}}$ denotes the identity matrix of size $m^{j}$ and $\otimes$ is
the Kronecker product.
Using the multivariate Laplace transform~\cite{Rug81}, the regular Volterra
kernels~\cref{eqn:voltime} yield a representation of~\cref{eqn:bsys}
in the frequency domain by the so-called multivariate regular transfer functions
\begin{align} \label{eqn:btf}
  \begin{aligned}
    G_{k}(s_{1},\ldots,s_{k}) & = C(s_{k}E - A)^{-1} \left(
      \prod\limits_{j = 1}^{k-1} (I_{m^{j-1}} \otimes N)
      (I_{m^{j}} \otimes (s_{k-j}E - A)^{-1}) \right)\\
    & \quad{}\times{}
      (I_{m^{k-1}} \otimes B),
  \end{aligned}
\end{align}
with $s_{1}, \ldots, s_{k} \in \C$.

Motivated by the structured linear case~\cite{BeaG09} and structured
bilinear systems as in~\cref{eqn:bsosys}, \cite{BenGW21a}~introduced the
frequency domain representation of structured bilinear systems in terms of
the \emph{structured regular subsystem transfer functions} of the form
\begin{align} \label{eqn:tf}
  \begin{aligned}
    G_{k}(s_{1}, \ldots, s_{k}) & = \cC(s_{k})\cK(s_{k})^{-1} \left(
      \prod\limits_{j = 1}^{k-1} \big( I_{m^{j-1}} \otimes \cN(s_{k-j}) \big)
      \big( I_{m^{j}} \otimes \cK(s_{k-j})^{-1} \big) \right)\\
    & \quad{}\times{}
      \big( I_{m^{k-1}} \otimes \cB(s_{1}) \big),
  \end{aligned}
\end{align}
for $k \geq 1$, where
$\cC(s)\colon \C \rightarrow \C^{p \times n}$,
$\cK(s)\colon \C \rightarrow \C^{n \times n}$,
$\cB(s)\colon \C \rightarrow \C^{n \times m}$, and
$\cN_{j}\colon \C \rightarrow \C^{n \times n}$
for $j = 1, \ldots, m$  are matrix-valued  functions, and
 $\cN(s) = \begin{bmatrix} \cN_{1}(s) & \ldots &
\cN_{m}(s) \end{bmatrix}$.
This general frameworks contains the unstructured bilinear
systems~\cref{eqn:bsys} as its special case where
\begin{align*}
  \begin{aligned}
    \cC(s) & = C, &
      \cK(s) & = sE - A, &
      \cB(s) & = B, &
      \cN(s) & = \begin{bmatrix} N_{1} & \ldots & N_{m} \end{bmatrix},
  \end{aligned}
\end{align*}
Also, it recovers the bilinear second-order system~\cref{eqn:bsosys} by
choosing
\begin{align*}
  \begin{aligned}
    \cC(s) & = C_{\rm{p}} + s C_{\rm{v}}, &
      \cK(s) & = s^{2} M + s D + K, &
      \cB(s) & = B_{\rm{u}}, &
      \cN(s) & = N_{\rm{p}}  + s N_{\rm{v}},
  \end{aligned}
\end{align*}
where $N_{\rm{p}} = \begin{bmatrix} N_{\rm{p}, 1} & \ldots & N_{\rm{p}, m} 
\end{bmatrix}$ and $N_{\rm{v}} = \begin{bmatrix} N_{\rm{v}, 1} & \ldots & 
N_{\rm{v}, m} \end{bmatrix}$.
We refer the reader to~\cite{BenGW21a} for a more detailed derivation of 
structured multivariate transfer functions and other structured examples.

For the full-order structured bilinear control system with subsystem transfer
functions~\cref{eqn:tf}, we will construct structure-preserving reduced
bilinear systems using Petrov-Galerkin projection:
Given two model reduction basis matrices $W, V \in \C^{n \times r}$ for the test
and trial spaces, respectively, with $r \ll n$, the reduced-order quantities
are given by
\begin{align} \label{eqn:proj}
  \begin{aligned}
    \hcC(s) & = \cC(s) V,
      & \hcK(s) & = W^{\herm} \cK(s) V,
      & \hcB(s) & = W^{\herm} \cB(s),~\text{and}
      & \hcN_{j}(s) & = W^{\herm} \cN_{j}(s) V
  \end{aligned}
\end{align}
for $j = 1, \ldots, m$, where $(\cdot)^{\herm}$ denotes the conjugate transpose.
The corresponding reduced-order system $\hG$ is then given by the underlying
reduced-order matrices from~\cref{eqn:proj} and the corresponding multivariate
transfer functions
\begin{align} \label{eqn:tfrom}
  \begin{aligned}
    \hG_{k}(s_{1}, \ldots, s_{k}) & = \hcC(s_{k})\hcK(s_{k})^{-1} \left(
      \prod\limits_{j = 1}^{k-1} \big( I_{m^{j-1}} \otimes \hcN(s_{k-j}) \big)
      \big( I_{m^{j}} \otimes \hcK(s_{k-j})^{-1} \big) \right)\\
    & \quad{}\times{}
      (I_{m^{k-1}} \otimes \hcB(s_{1})),
  \end{aligned}
\end{align}
for $k \geq 1$.
For example, for the mechanical bilinear system in~\cref{eqn:bsosys}, the
reduced-order model will have the form
\begin{align*}
  \begin{aligned}
    \widehat{M} \ddot{\hat{q}}(t) + \widehat{D} \dot{\hat{q}}(t) + 
      \widehat{K} \hat{q}(t) & = \sum\limits_{j = 1}^{m}
      \widehat{N}_{\rm{p},j} \hat{q}(t) u_{j}(t) + \sum\limits_{j = 1}^{m}
      \widehat{N}_{\rm{v},j} \dot{\hat{q}} (t)u_{j}(t) +
      \widehat{B}_{\rm{u}} u(t),\\
    \hat{y}(t) & = \widehat{C}_{\rm{p}}\hat{q}(t) + 
      \widehat{C}_{\rm{v}} \dot{\hat{q}}(t),
  \end{aligned}
\end{align*}
where $\widehat{M}, \widehat{D}, \widehat{K}, \widehat{N}_{\rm{p},j},
\widehat{N}_{\rm{v},j} \in \R^{r \times r}$ for $j = 1, \ldots, m$,
$\widehat{B}_{\rm{u}} \in \R^{r \times m}$, and
$\widehat{C}_{\rm{p}}, \widehat{C}_{\rm{v}} \in \R^{p \times r}$
are given by
\begin{align*}
  \begin{aligned}
    \widehat{M} & = W^{\herm} M V, &
      \widehat{D} & = W^{\herm} D V, &
      \widehat{K} & = W^{\herm} K V, &
      \widehat{B}_{\rm{u}} & = W^{\herm} B_{\rm{u}}, \\
    \widehat{N}_{\rm{p},j} & = W^{\herm} {N}_{\rm{p},j} V, &
      \widehat{N}_{\rm{v},j} & = W^{\herm} {N}_{\rm{v},j} V, &
      \widehat{C}_{\rm{p}} & = {C}_{\rm{p}} V, ~\text{and} &
      \widehat{C}_{\rm{v}} & = {C}_{\rm{v}} V.
  \end{aligned}
\end{align*}
We will construct the model reduction bases $W$ and $V$ such that the
reduced-order subsystem transfer functions $\hG_{k}$ in~\cref{eqn:tfrom} are
\emph{multivariate (tangential) interpolants} to the full-order ones $G_{k}$
at some selected frequencies
$\{\sigma_{1}, \sigma_{2}, \ldots, \sigma_{k}\} \subseteq \C$.
Below, we will make it precise what we mean by tangential interpolation in this
setting.
But first, it is worth noting the dimension of $G_{k}$.
For a MIMO bilinear system with $m$ inputs and $p$ outputs, $G_{k}$, evaluated
at given frequencies, is a $p \times m^{k}$ matrix, i.e., it has a polynomial
growth in the input dimension.
Then, full matrix interpolation of $G_{k}$ by $\hG_{k}$ imposes a rather large
number of interpolation conditions to satisfy, leading to rapid growth of the
reduced order.
We will resolve this issue via tangential interpolation.
It will help to recall the tangential interpolation problem for the linear case
first.

%%%%%%%%%%%%%%%%%%%%%%%%%%%%%%%%%%%%%%%%%%%%%%%%%%%%%%%%%%%%%%%%%%%%%%%%%%%%%%%%

\subsection{Tangential interpolation for linear dynamical systems}

The tangential interpolation replaces the full matrix interpolation of a
matrix-valued function with interpolation along selected directions and can be
interpreted as adding constraints to the matrix interpolation
problem~\cite{BalGR90}.
For given interpolation points $\sigma_{1}, \ldots, \sigma_{k} \in \C$,
given function values $h_{1}, \ldots, h_{k} \in \C^{p}$ and right tangential
directions $b_{1}, \ldots, b_{k} \in \C^{m}$, the task of the right
tangential interpolation is to find an interpolating function
$H\colon \C \rightarrow \C^{p \times m}$ such that
\begin{align} \label{eqn:tang}
  \begin{aligned}
    H(\sigma_{j}) b_{j} & = h_{j} & \text{for } j = 1, \ldots, k.
  \end{aligned}
\end{align}
The left interpolation problem is defined similarly.

It was then proposed in~\cite{BalGR90} and utilized in~\cite{GalVV04} to
employ tangential interpolation for model reduction of linear 
unstructured multi-input/multi-output systems by restricting the
interpolant~\cref{eqn:tang} to a rational matrix-valued function and using the 
system's transfer function evaluations along certain directions as function 
values to interpolate.
In other words, given the original linear system's transfer function
$G(s) = C (sE - A)^{-1} B$, the goal is to construct a reduced-order system with
transfer function $\hG(s) = \widehat{C} (s \widehat{E} - \widehat{A})^{-1}
\widehat{B}$ such that for given interpolation points
$\sigma_{1}, \ldots, \sigma_{k} \in \C$ and
directions $b^{(1)}, \ldots, b^{(k)} \in \C^{m}$ as well as
$c^{(1)}, \ldots, c^{(k)} \in \C^{p}$, the right or left tangential
interpolation conditions
\begin{align} \label{eqn:tanglin}
  \begin{aligned}
    G(\sigma_{j}) b^{(j)} & = \hG(\sigma_{j}) b^{(j)}, && \text{or} &
      \big( c^{(j)} \big)^{\herm} G(\sigma_{j}) & =
      \big( c^{(j)} \big)^{\herm} \hG(\sigma_{j}),
  \end{aligned}
\end{align}
for $j = 1, \ldots, k$, hold.
It has been shown via numerous examples that tangential interpolation yields
accurate reduced-order models while allowing to choose the size of the 
reduced-order model independent of the input and output dimensions (unlike in
the matrix interpolation framework) and thus results in smaller reduced-order
models.
Indeed, tangential interpolation, not the matrix interpolation, forms the
necessary conditions for optimal model reduction of linear systems in the
$\mathcal{H}_{2}$ norm~\cite{AntBG20}. 
The tangential interpolation problem~\cref{eqn:tanglin} (and the
projection-based solution framework) was later extended to structure-preserving
Hermite interpolation in~\cite{BeaG09} of structured transfer functions of
the form $G(s) = \cC(s) \cK(s)^{-1} \cB(s)$.

%%%%%%%%%%%%%%%%%%%%%%%%%%%%%%%%%%%%%%%%%%%%%%%%%%%%%%%%%%%%%%%%%%%%%%%%%%%%%%%%

\subsection{Blockwise tangential interpolation for unstructured bilinear systems}%
\label{subsec:tangblockstd}

Extending tangential interpolation to unstructured bilinear systems of the
form~\cref{eqn:bsys} was first considered in~\cite{BenBD11, RodGB18},
using the observation that multiplying out the Kronecker products
in~\cref{eqn:btf} yields
\begin{align*}
  G_{k}(s_{1}, \ldots, s_{k}) & = \big[ C(s_{k} E - A)^{-1} N_{1} \cdots
    N_{1} (s_{1} E - A)^{-1} B,\\
  & \phantom{= \big[~}
    C(s_{k} E - A)^{-1} N_{1} \cdots N_{2} (s_{1} E - A)^{-1} B,\\
  & \phantom{= \big[~} \ldots,\\
  & \phantom{= \big[~}
    C(s_{k} E - A)^{-1} N_{m} \cdots N_{m} (s_{1} E - A)^{-1} B \big].
\end{align*}
Each block entry in this formula is then considered as separate transfer
function, which will be interpolated along the same chosen directions.
For example, with a right tangential direction $b \in \C^{m}$, the blockwise
evaluation of the transfer function along $b$ is given by
\begin{align*}
  G_{k}(s_{1}, \ldots, s_{k}) (I_{m} \otimes b) & = \big[ C(s_{k} E - A)^{-1}
    N_{1} \cdots N_{1} (s_{1} E - A)^{-1} Bb,\\
  & \phantom{= \big[~}
    C(s_{k} E - A)^{-1} N_{1} \cdots N_{2} (s_{1} E - A)^{-1} Bb,\\
  & \phantom{= \big[~} \ldots,\\
  & \phantom{= \big[~}
    C(s_{k} E - A)^{-1} N_{m} \cdots N_{m} (s_{1} E - A)^{-1} Bb \big],
\end{align*}
leading to the concept of the \emph{blockwise tangential interpolation problem}:
Given interpolation points $\sigma_{1}, \ldots, \sigma_{k} \in \C$ and
tangential directions $b \in \C^{m}$ and $c \in \C^{p}$, find
a reduced-order model such that
\begin{align} \label{eqn:tangblock1}
  G_{k}(\sigma_{1}, \ldots, \sigma_{k}) (I_{m} \otimes b) &
    = \hG_{k}(\sigma_{1}, \ldots, \sigma_{k}) (I_{m} \otimes b) \quad\text{or}\\
  \label{eqn:tangblock2}
  c^{\herm} G_{k}(\sigma_{1}, \ldots, \sigma_{k}) &
    = c^{\herm} \hG_{k}(\sigma_{1}, \ldots, \sigma_{k})
\end{align}
hold.
Also, the bi-tangential interpolation condition,
\begin{align} \label{eqn:tangblock3}
  c^{\herm} G_{k}(\sigma_{1}, \ldots, \sigma_{k}) (I_{m} \otimes b) &
    = c^{\herm} \hG_{k}(\sigma_{1}, \ldots, \sigma_{k}) (I_{m} \otimes b),
\end{align}
will be of high interest in the bilinear system case.
While, in principle, \cref{eqn:tangblock1,eqn:tangblock2}
imply~\cref{eqn:tangblock3}, we will see later that it is possible to match
subsystem transfer functions of higher level $k$ in the sense
of~\cref{eqn:tangblock3} by enforcing~\cref{eqn:tangblock1,eqn:tangblock2} on
lower level transfer functions.

The blockwise tangential interpolation problem can be viewed as
a mixture of tangential interpolation, as it is done for the linear system
case~\cref{eqn:tanglin}, combined with the blocks of multivariate transfer
functions.
While in the case of linear systems, the tangential interpolation restricts the
problem to vectors or scalars of fixed sizes to be interpolated, this is not 
true anymore for the blockwise approach in the bilinear system case.
As already observed in~\cite{BenBD11, RodGB18}, the blockwise approach
still leads to the interpolation of an exponentially increasing number of
vectors or matrices, making it only marginally better than the matrix
interpolation method for model reduction.

%%%%%%%%%%%%%%%%%%%%%%%%%%%%%%%%%%%%%%%%%%%%%%%%%%%%%%%%%%%%%%%%%%%%%%%%%%%%%%%%

\subsection{Notation}

To simplify notation in this work, we will use:
\begin{align} \label{eqn:partderiv}
  \partial_{s_{1}^{j_{1}} \cdots s_{k}^{j_{k}}} f(z_{1}, \ldots, z_{k}) & :=
    \frac{\partial^{j_{1} + \ldots + j_{k}} f}{\partial s_{1}^{j_{1}} \cdots
    \partial s_{k}^{j_{k}}} (z_{1}, \ldots, z_{k}),
\end{align}
to denote the differentiation of an analytic function $f\colon \C^{k}
\rightarrow \C^{\ell}$ with respect to the complex variables $s_{1}, \ldots,
s_{k}$ and evaluated at $z_{1}, \ldots, z_{k} \in \C$.
We denote for matrix-valued functions $\cK\colon \C \to \C^{n \times n}$, which
map complex scalars onto square matrices, the inverse of their evaluation
by $\cK^{-1} := \cK(.)^{-1}$.
This notation of the inverse of evaluated matrix-valued functions will occur
together with the notation of partial derivatives~\cref{eqn:partderiv}.
For example, given two matrix-valued functions
$\cB\colon \C \to \C^{n \times m}$ and $\cK\colon \C \to \C^{n \times n}$, we
denote the partial derivative of the product of $\cB$ with the inverse of $\cK$
evaluated in the points $z_{1}, z_{2} \in \C$ by
\begin{align*}
  \partial_{s_{1}^{j_{1}}} \partial_{s_{2}^{j_{2}}} (\cK^{-1} \cB)
    (z_{1}, z_{2}) & := \frac{\partial^{j_{1} + j_{2}} \cK(.)^{-1}
    \cB(.)}{\partial s_{1}^{j_{1}} \partial s_{2}^{j_{2}}} (z_{1}, z_{2})
\end{align*}
Also, we will use the notion of the Jacobi matrix given by
\begin{align} \label{eqn:jacobi}
  \nabla f & = \begin{bmatrix}
    \partial_{s_{1}} f & \ldots & \partial_{s_{k}} f \end{bmatrix},
\end{align}
denoting the concatenation of all partial derivatives of an analytic 
function $f\colon \C^{k} \rightarrow \C^{\ell}$ with respect to the complex 
variables $s_{1}, \ldots, s_{k}$.

For bilinear systems, we have already introduced the notation
\begin{align*}
  \cN(s) & = \begin{bmatrix} \cN_{1}(s) & \ldots & \cN_{m}(s) \end{bmatrix}
\end{align*}
to denote horizontally concatenated matrix functions corresponding to the
bilinear terms.
Additionally, we use
\begin{align*}
  \tcN(s) & = \begin{bmatrix} \cN_{1}(s) \\ \vdots \\ \cN_{m}(s) \end{bmatrix}
\end{align*}
for denoting the vertical concatenation of the matrix functions corresponding 
to the bilinear terms.
We denote the vector of ones of length $m$ by $\mathds{1}_{m}$.

%%%%%%%%%%%%%%%%%%%%%%%%%%%%%%%%%%%%%%%%%%%%%%%%%%%%%%%%%%%%%%%%%%%%%%%%%%%%%%%%
% UNIFYING FRAMEWORK.                                                          %
%%%%%%%%%%%%%%%%%%%%%%%%%%%%%%%%%%%%%%%%%%%%%%%%%%%%%%%%%%%%%%%%%%%%%%%%%%%%%%%%

\section{Generalized structured tangential interpolation framework}%
\label{sec:tang}

In this section, we will start with two different interpretations of 
tangential interpolation~\cref{eqn:tanglin} and their corresponding 
interpolation problems for bilinear systems.
Motivated by these formulations, we introduce a unifying framework for
tangential interpolation of structured bilinear systems and give subspace
conditions for structure-preserving model reduction of the corresponding
bilinear systems.
As a special case of the unifying framework, we derive the 
theory for structure-preserving blockwise tangential interpolation as reviewed
in \Cref{subsec:tangblockstd} and previously employed in the literature for
standard (unstructured) bilinear systems.

%%%%%%%%%%%%%%%%%%%%%%%%%%%%%%%%%%%%%%%%%%%%%%%%%%%%%%%%%%%%%%%%%%%%%%%%%%%%%%%%

\subsection{Tangential interpolation in the frequency domain}
\label{subsec:tangfreq}

Examining the original formulation of tangential interpolation~\cref{eqn:tang}
and the multivariate transfer functions~\cref{eqn:tf}, a first natural approach
to tangential interpolation for bilinear systems would be
to choose an appropriately sized vector $\hb \in \C^{m^{k}}$, where
\begin{align*}
  \hb & = \begin{bmatrix} \left( \hb^{(1 1 \ldots 1)} \right)^{\herm} &
    \left(\hb^{(2 1 \ldots 1)} \right)^{\herm}\ldots & 
    \left(\hb^{\left( m m \ldots m \right)} \right)^{\herm}
    \end{bmatrix}^{\herm}
\end{align*}
and $b^{(j_{1} \ldots j_{k})} \in \C^{m}$ for all
$1 \leq j_{1}, \ldots, j_{k} \leq m$,
as right tangential direction and to consider interpolating
\begin{align} \label{eqn:tangsum}
  \begin{aligned}
    G_{k}(s_{1}, \ldots, s_{k}) \hb  = \sum\limits_{j_{1} = 1}^{m} \ldots 
      \sum\limits_{j_{k-1} = 1}^{m} &\cC(s_{k}) \cK(s_{k})^{-1} 
      \cN_{j_{k-1}}(s_{k-1}) \cK(s_{k-1})^{-1} \\
    & \quad{}\times{} \ldots \times \cN_{j_{1}}(s_{1}) \cK(s_{1})^{-1} 
      \cB(s_{1}) \hb^{(j_{1} \ldots j_{k})}.
  \end{aligned}
\end{align}
This general approach comes along with a computational drawback.
For every new transfer function level $k$, a different part of $\hb$ is
multiplied with the input functional $\cB(s)$ in each term of the
sum~\cref{eqn:tangsum}. Then,
the corresponding  basis for  model reduction would grow 
according to the different block entries of $\hb$ and, thus, even faster 
than for the blockwise tangential interpolation problem
(\Cref{subsec:tangblockstd}).
A remedy to this problem is to restrict the full direction 
vector to the repetition of a single small direction $b \in \C^{m}$, i.e.,
\begin{align} \label{eqn:resdirect}
  \hb = \mathds{1}_{m^{k-1}} \otimes b =
    \begin{bmatrix} b \\ \vdots \\ b \end{bmatrix}.
\end{align}
With this particular choice of $b$ in~\cref{eqn:resdirect}, the right tangential
interpolation problem can be written as
\begin{align} \label{eqn:tangfreq1}
  G_{k}(\sigma_{1}, \ldots, \sigma_{k}) (\mathds{1}_{m^{k-1}} \otimes b) &
    = \hG_{k}(\sigma_{1}, \ldots, \sigma_{k}) (\mathds{1}_{m^{k-1}} \otimes b),
\end{align}
for given interpolation points $\sigma_{1}, \ldots, \sigma_{k} \in \C$.
This restricts the interpolation problem to a vector of constant length with
respect to the transfer function level and thus allows for an efficient
construction of the projection basis.

For the left tangential interpolation problem, a direct extension of the 
classical approach~\cref{eqn:tang} would lead to the same results as in the 
blockwise tangential interpolation case~\cref{eqn:tangblock2} since the first
dimension of the transfer function is constant for all transfer function levels.
To consider a dual formulation of~\cref{eqn:tangsum} for the left tangential
interpolation problem (one for which the basis dimension
does not grow exponentially), we choose 
\begin{align} \label{eqn:tangfreq2}
  c^{\herm} G_{k}(\sigma_{1}, \ldots, \sigma_{k})
    (\mathds{1}_{m^{k-1}} \otimes I_{m})
    & = c^{\herm} \hG_{k}(\sigma_{1}, \ldots, \sigma_{k})
    (\mathds{1}_{m^{k-1}} \otimes I_{m}),
\end{align}
for a given direction $c \in \C^{p}$ and interpolation points
$\sigma_{1}, \ldots, \sigma_{k} \in \C$.  Consequently, we consider 
\begin{align} \label{eqn:tangfreq3}
  c^{\herm} G_{k}(\sigma_{1}, \ldots, \sigma_{k})
    (\mathds{1}_{m^{k-1}} \otimes b) &
    = c^{\herm} \hG_{k}(\sigma_{1}, \ldots, \sigma_{k})
    (\mathds{1}_{m^{k-1}} \otimes b)
\end{align}
as the bi-tangential interpolation problem.

%%%%%%%%%%%%%%%%%%%%%%%%%%%%%%%%%%%%%%%%%%%%%%%%%%%%%%%%%%%%%%%%%%%%%%%%%%%%%%%%

\subsection{Time domain interpretation of tangential interpolation}
\label{subsec:tangtime}

A different way to look at tangential interpolation of transfer functions
is its interpretation in the time domain.
We start with the tangential interpolation problem for linear dynamical 
systems~\cref{eqn:tanglin}.
For simplicity, we consider only the case of linear unstructured first-order
systems as given in the time domain by
\begin{align} \label{eqn:sys}
  \begin{aligned}
    E \dot{x}(t) & = A x(t) + B u(t),\\
    y(t) & = C x(t),
  \end{aligned}
\end{align}
with $E, A \in \R^{n \times n}$, $B \in \R^{n \times m}$ and $C \in
\R^{p \times n}$, and in the frequency domain by the transfer function
\begin{align*}
  G(s) & = C (sE - A)^{-1} B.
\end{align*}
We note that the following derivations work for all structured
linear systems as well~\cite{BeaG09}.
The multiplication with tangential directions in the frequency domain can be
considered independent of the chosen interpolation points, which gives new 
systems in the frequency domain described by the transfer functions
\begin{align} \label{eqn:freqtangsys}
  \begin{aligned}
    \tG_{\rm{b}}(s) & = G(s) b  && \text{and} &
      \tG_{\rm{c}}(s) & = c^{\herm} G(s),
  \end{aligned}
\end{align}
with the tangential directions $b \in \C^{m}$ and $c \in \C^{p}$.
Those new systems~\cref{eqn:freqtangsys} allow now for re-interpretation in the
time domain.
In fact, the resulting tangential systems can be seen as embedding the 
original linear system $G$ into single-input or single-output systems.
We set the outer inputs and outputs as $u(t) = b \tu(t)$ and
$\ty(t) = c^{\herm} y(t)$, respectively, and obtain the new systems:
\begin{align} \label{eqn:timetangsys1}
  \tG_{\rm{b}}: \left\{
  \begin{aligned}
    E \dot{x}(t) & = A x(t) + B b \tu(t), \\
    y(t) & = C x(t),
  \end{aligned}
  \right.
\end{align}
for embedding the inputs, and
\begin{align} \label{eqn:timetangsys2}
  \tG_{\rm{c}}: \left\{
  \begin{aligned}
    E \dot{x}(t) & = A x(t) + B u(t), \\
    \ty(t) & = c^{\herm} C x(t),
  \end{aligned}
  \right.
\end{align}
for the outputs.
Thereby, in the setting of tangential interpolation, we are restricting
the system inputs to a single input signal that is spread along a given 
direction $b$ to be fed into the original system~\cref{eqn:sys} or we 
restrict the output to a linear combination of the observations of the original 
system~\cref{eqn:sys} using the direction $c$.

Now, we consider the bilinear unstructured systems~\cref{eqn:bsys} and 
make use of the time domain interpretation of tangential interpolation we 
have done for the linear
case~\cref{eqn:timetangsys1,eqn:timetangsys2}.
Using the same tangential directions as before and the embedding strategy
for the bilinear system~\cref{eqn:bsys}, with
$b = \begin{bmatrix} b_{1} & b_{2} & \ldots & b_{m}\end{bmatrix}^{\trans}$
we obtain 
\begin{align} \label{eqn:timetangbsys1}
  \tG_{\rm{b}}: \left\{
  \begin{aligned}
    E \dot{x}(t) & = A x(t) + \sum\limits_{j = 1}^{m}N_{j} x(t) b_{j} \tu(t)
      + B b \tu(t), \\
    y(t) & = C x(t),
  \end{aligned}
  \right.
\end{align}
for the embedded inputs,
\begin{align} \label{eqn:timetangbsys2}
  \tG_{\rm{c}}: \left\{
  \begin{aligned}
    E \dot{x}(t) & = A x(t) + \sum\limits_{j = 1}^{m}N_{j} x(t) u_{j}(t)
      + B u(t), \\
    \ty(t) & = c^{\herm} C x(t),
  \end{aligned}
  \right.
\end{align}
for embedding the outputs.
Additionally, we consider here the fully embedded system
\begin{align} \label{eqn:timetangbsys3}
  \tG_{\rm{cb}}: \left\{
  \begin{aligned}
    E \dot{x}(t) & = A x(t) + \sum\limits_{j = 1}^{m}N_{j} x(t) b_{j} \tu(t)
      + B b \tu(t), \\
    \ty(t) & = c^{\herm} C x(t),
  \end{aligned}
  \right.
\end{align}
as it relates to the bi-tangential interpolation problem.
These new bilinear 
systems~\cref{eqn:timetangbsys1,eqn:timetangbsys2,eqn:timetangbsys3}
will be used to derive a new concept of tangential interpolation for bilinear 
systems.
The corresponding regular transfer functions for the embedded systems are given
as follows:
\begin{align} \label{eqn:freqtangbsys1}
  \tG_{\rm{b}, k}(s_{1}, \ldots, s_{k}) & = C(s_{k} I_{n} - A)^{-1} \left( 
    \prod\limits_{j = 1}^{k-1} \left( \sum\limits_{i = 1}^{m} b_{i} N_{i}
    \right) (s_{k-j}I_{n} - A)^{-1} \right) B b, \\ \nonumber
  \tG_{\rm{c}, k}(s_{1},\ldots,s_{k}) & = c^{\herm} C(s_{k}E - A)^{-1} \left(
      \prod\limits_{j = 1}^{k-1} (I_{m^{j-1}} \otimes N)
      (I_{m^{j}} \otimes (s_{k-j}E - A)^{-1})
      \right) \\ \label{eqn:freqtangbsys2}
  & \quad{}\times{}
    (I_{m^{k-1}} \otimes B), \\ \label{eqn:freqtangbsys3}
  \tG_{\rm{cb},k}(s_{1}, \ldots, s_{k}) & = c^{\herm} C(s_{k} I_{n} - A)^{-1}
    \left( \prod\limits_{j = 1}^{k-1} \left( \sum\limits_{i = 1}^{m} b_{i} N_{i}
    \right) (s_{k-j}I_{n} - A)^{-1} \right) B b,
\end{align}
for $k \geq 1$.
These new transfer 
functions~\cref{eqn:freqtangbsys1,eqn:freqtangbsys2,eqn:freqtangbsys3}
can now be combined with our structured transfer function setting~\cref{eqn:tf}.
For a given direction vector $b \in \C^{m}$, we denote the scaled summation of
the structured multivariate transfer functions by
\begin{align} \label{eqn:scaledsummed}
  \begin{aligned}
    \tG_{k}(s_{1}, \ldots, s_{k}) & = \cC(s_{k}) \cK(s_{k})^{-1}
      \left( \sum\limits_{j = 1}^{m} b_{j} \cN_{j}(s_{k-1}) \right)
      \cK(s_{k-1})^{-1} \times \ldots\\
    & \quad{}\times{}
      \left( \sum\limits_{j = 1}^{m} b_{j} \cN_{j}(s_{1}) \right)
      \cK(s_{1})^{-1} \cB(s_{1}).
  \end{aligned}
\end{align}
The bilinear terms in~\cref{eqn:scaledsummed} collapsed from large concatenated
matrices in~\cref{eqn:tf} to simple $n$-di\-men\-sion\-al matrices.
Therefore, the Kronecker products become classical matrix multiplications such
that $\tG_{k}\colon \C^{k} \rightarrow \C^{p \times m}$.

Denoting the scaled and summed transfer function of the reduced-order model
by $\htG_{k}(s_{1}, \ldots, s_{k})$, the corresponding right tangential 
interpolation problem is given by
\begin{align} \label{eqn:tangtime1}
  \tG_{k}(\sigma_{1}, \ldots, \sigma_{k}) b &
    = \htG_{k}(\sigma_{1}, \ldots, \sigma_{k}) b,
\end{align}
for given interpolation points $\sigma_{1}, \ldots, \sigma_{k} \in \C$.
As before, motivated by duality, the left and bi-tangential interpolation 
problems are chosen to be
\begin{align} \label{eqn:tangtime2}
    c^{\herm} \tG_{k}(\sigma_{1}, \ldots, \sigma_{k}) &
      = c^{\herm} \htG_{k}(\sigma_{1}, \ldots, \sigma_{k})
      \quad \text{and}\\ \label{eqn:tangtime3}
    c^{\herm} \tG_{k}(\sigma_{1}, \ldots, \sigma_{k}) b &
      = c^{\herm} \htG_{k}(\sigma_{1}, \ldots, \sigma_{k}) b,
\end{align}
respectively.

\begin{remark}[Relation to other control systems]
  The idea of time domain interpretation of tangential interpolation can easily 
  be extended to other types of control systems, e.g., to systems 
  with polynomial nonlinearities.
  This might lead to new efficient tangential interpolation approaches for 
  nonlinear multi-input/multi-output control systems.
\end{remark}

%%%%%%%%%%%%%%%%%%%%%%%%%%%%%%%%%%%%%%%%%%%%%%%%%%%%%%%%%%%%%%%%%%%%%%%%%%%%%%%%

\subsection{Structured tangential interpolation framework}%
\label{subsec:tangmod}

Now, employing the scaled and summed transfer functions we introduced
in~\cref{eqn:scaledsummed}, we develop a generalized framework for tangential
interpolation of multivariate transfer functions that unifies the different
approaches to bilinear tangential interpolation discussed in
\Cref{subsec:tangblockstd,subsec:tangfreq,subsec:tangtime}.
The new framework will encompass all these different approaches under one
umbrella and thus give one formulation to cover all these different
interpretations of bilinear tangential interpolation, filling an important gap
in the interpolatory model reduction theory  of bilinear systems. Moreover, we
will develop this new framework  for the structured bilinear dynamical systems
for which tangential interpolation has not been studied yet.

We start by defining the modified multivariate transfer functions
\begin{align} \label{eqn:tfmod}
  \begin{aligned}
    & \mG{k}{s_{1}, \ldots, s_{k}}{d^{(1)}, \ldots, d^{(k-1)}}\\
    & := \cC(s_{k}) \cK(s_{k})^{-1} \left( \prod\limits_{j = 1}^{k-1}
      \mN{s_{k-j}}{d^{(k-j)}} \cK(s_{k-j})^{-1} \right) \cB(s_{1}),
  \end{aligned}
\end{align}
for $k \geq 1$, with frequency points $s_{1}, \ldots, s_{k} \in \C$
and scaling vectors $d^{(1)}, \ldots, d^{(k-1)} \in \C^{m}$,  where
\begin{align*}
  \mN{s_{j}}{d^{(j)}} & = \cN(s) (d^{(j)} \otimes I_{n})
    = \sum\limits_{i = 1}^{m} d^{(j)}_{i} \cN_{i}(s_{j})
\end{align*}
denotes the scaled sum of the bilinear terms.
Note that the first modified transfer function does not depend on a scaling 
vector and it holds that
\begin{align*}
  G_{1}(s_{1}) & = \mGf{s_{1}}.
\end{align*}
In this setting, $\hmG{k}{s_{1}, \ldots, s_{k}}{d^{(1)}, \ldots, d^{(k-1)}}$ 
denotes the modified transfer functions of the re\-duc\-ed-order model.
For the modified transfer functions, we define the following tangential
interpolation problem:

\begin{problem}[Tangential modified transfer function interpolation]%
  \label{prb:tangmod}
  For given interpolation points $\sigma_{1}, \ldots, \sigma_{k} \in \C$,
  scaling vectors $d^{(1)}, \ldots, d^{(k-1)} \in \C^{m}$, and tangential
  directions $b \in C^{m}$ and $c \in \C^{p}$, find a reduced-order model such
  that
  \begin{align} \label{eqn:tangmod1}
    \mG{k}{\sigma_{1}, \ldots, \sigma_{k}}{d^{(1)}, \ldots, d^{(k-1)}} b
      & = \hmG{k}{\sigma_{1}, \ldots, \sigma_{k}}
      {d^{(1)}, \ldots, d^{(k-1)}} b, \\ \label{eqn:tangmod2}
    c^{\herm} \mG{k}{\sigma_{1}, \ldots, \sigma_{k}}
      {d^{(1)}, \ldots, d^{(k-1)}}
      & = c^{\herm} \hmG{k}{\sigma_{1}, \ldots, \sigma_{k}}
      {d^{(1)}, \ldots, d^{(k-1)}}, \quad\text{or} \\ \label{eqn:tangmod3}
    c^{\herm} \mG{k}{\sigma_{1}, \ldots, \sigma_{k}}
      {d^{(1)}, \ldots, d^{(k-1)}} b
      & = c^{\herm} \hmG{k}{\sigma_{1}, \ldots, \sigma_{k}}
      {d^{(1)}, \ldots, d^{(k-1)}} b
  \end{align}
  hold.
\end{problem}

Before we present our results that show how to construct the reduced bilinear
systems to solve the structure-preserving tangential interpolation problem in
the new generalized framework, we formally state in the following corollary that
the earlier bilinear tangential interpolation frameworks  are special cases of
the proposed unifying framework.
Due to its significance in the literature and its more complex formulation in the
unifying framework, the case of blockwise tangential interpolation is treated
separately in \Cref{subsec:tangblock}.

\begin{corollary}[Choices of the scaling vectors]
  Consider the proposed tangential interpolation problem (\Cref{prb:tangmod})
  with the corresponding scaling vectors $d^{(j)}$ in~\cref{eqn:tfmod}.
  Then:
  \begin{enumerate}[label = (\alph*)]
    \item Choosing $d^{(1)} = \ldots = d^{(k-1)} = \mathds{1}_{m}$ yields the 
      extension of classical tangential interpolation to the multivariate 
      transfer functions of bilinear 
      systems~\cref{eqn:tangfreq1,eqn:tangfreq2,eqn:tangfreq3} from 
      \Cref{subsec:tangfreq}.
    \item Choosing $d^{(1)} = \ldots = d^{(k-1)} = b$, with $b \in \C^{m}$ as
      the right tangential direction, yields the re-interpretation of tangential
      interpolation in time
      domain~\cref{eqn:tangtime1,eqn:tangtime2,eqn:tangtime3} from
      \Cref{subsec:tangtime}.
  \end{enumerate}
\end{corollary}

The following theorem establishes the subspace conditions on the model reduction
bases $V$ and $W$ to construct the reduced-order model~\cref{eqn:proj} that
satisfies the tangential interpolation
conditions~\cref{eqn:tangmod1,eqn:tangmod2,eqn:tangmod3}.

\begin{theorem}[Modified structured tangential interpolation]%
  \label{thm:tangmod}
  Let $G$ be a bilinear system, associated with its modified transfer 
  functions $\mGs_{k}$ in~\cref{eqn:tfmod}, and $\hG$ the reduced-order bilinear 
  system, constructed as in~\cref{eqn:proj} with its modified transfer functions 
  $\hmGs_{k}$.
  Given sets of interpolation points $\sigma_{1}, \ldots, \sigma_{k} \in \C$ and
  $\varsigma_{1}, \ldots, \varsigma_{\kappa} \in \C$ such that the matrix 
  functions $\cC(s)$, $\cK(s)^{-1}$, $\cN(s)$, $\cB(s)$, $\hcK(s)^{-1}$ are 
  defined for $s \in \{ \sigma_{1}, \ldots, \sigma_{k}, \varsigma_{1}, \ldots,
  \varsigma_{\kappa} \}$, two tangential directions $b \in \C^{m}$ and 
  $c \in \C^{p}$, and two sets of scaling vectors $d^{(1)}, \ldots, 
  d^{(k-1)} \in \C^{m}$ and $\delta^{(1)}, \ldots, \delta^{(\kappa-1)} \in
  \C^{m}$, the following statements hold:
  \begin{enumerate}[label=(\alph*)]
    \item If $V$ is constructed as
      \begin{align*}
        v_{1} & = \cK(\sigma_{1})^{-1}\cB(\sigma_{1})b,\\
        v_{j} & = \cK(\sigma_{j})^{-1} \mN{\sigma_{j-1}}{d^{(j-1)}} v_{j-1},
          & 2 \leq j \leq k,\\
        \mspan(V) & \supseteq \mspan\left([v_{1}, \ldots,
        v_{k}]\right),
      \end{align*}
      then the following interpolation conditions hold true:
      \begin{align*}
        \mGf{\sigma_{1}}b & =  \hmGf{\sigma_{1}}b,\\
        \mG{2}{\sigma_{1}, \sigma_{2}}{d^{(1)}} b & =
          \hmG{2}{\sigma_{1}, \sigma_{2}}{d^{(1)}} b,\\
        & \,\,\, \vdots\\
        \mG{k}{\sigma_{1}, \ldots, \sigma_{k}} 
          {d^{(1)}, \ldots, d^{(k-1)}} b &
          = \hmG{k}{\sigma_{1}, \ldots, \sigma_{k}} 
          {d^{(1)}, \ldots, d^{(k-1)}} b.
      \end{align*}
    \item If $W$ is constructed as
      \begin{align*}
        w_{1} & = \cK(\varsigma_{\kappa})^{-\herm}\cC
          (\varsigma_{\kappa})^{\herm} c,\\
        w_{i} & = \cK(\varsigma_{\kappa-i+1})^{-\herm}
          \mN{\varsigma_{\kappa-i+1}}{\delta^{(\kappa-i+1)}}^{\herm} 
          w_{i-1}, & 2 \leq i \leq \kappa,\\
        \mathrm{span}(W) & \supseteq \mathrm{span}\left([w_{1}, \ldots,
          w_{\kappa}] \right),
      \end{align*}
      then the following interpolation conditions hold true:
      \begin{align*}
        c^{\herm} \mGf{\varsigma_{\kappa}}
          & = c^{\herm} \hmGf{\varsigma_{\kappa}},\\
        c^{\herm} \mG{1}{\varsigma_{\kappa-1}, \varsigma_{\kappa}}
          {\delta^{(\kappa-1)}} & 
          = c^{\herm} \hmG{1}{\varsigma_{\kappa-1}, \varsigma_{\kappa}}
          {\delta^{(\kappa-1)}}, \\
        & \,\,\, \vdots \\
        c^{\herm} \mG{\kappa}{\varsigma_{1}, \ldots, \varsigma_{\kappa}}
          {\delta^{(1)}, \ldots, \delta^{(\kappa-1)}} &
          = c^{\herm} \hmG{\kappa}{\varsigma_{1}, \ldots, \varsigma_{\kappa}}
          {\delta^{(1)}, \ldots, \delta^{(\kappa-1)}}.
      \end{align*}
    \item Let $V$ be constructed as in Part~(a) and $W$ as in Part~(b).
      Then, additionally to the results in~(a) and~(b), the following
      interpolation conditions hold:
      \begin{align*}
        & c^{\herm} \mG{q + \eta}{\sigma_{1}, \ldots, \sigma_{q},
          \varsigma_{\kappa-\eta+1}, \ldots, \varsigma_{\kappa}}
          {d^{(1)}, \ldots, d^{(q-1)}, z, \delta^{(\kappa-\eta+1)}, \ldots, 
          \delta^{(\kappa-1)}} b \\
        & = c^{\herm} \hmG{q + \eta}{\sigma_{1}, \ldots, \sigma_{q},
          \varsigma_{\kappa-\eta+1}, \ldots, \varsigma_{\kappa}}
          {d^{(1)}, \ldots, d^{(q-1)}, z, \delta^{(\kappa-\eta+1)}, \ldots, 
          \delta^{(\kappa-1)}} b,
      \end{align*}
      for $1 \leq q \leq k$,  $1 \leq \eta \leq \kappa$ and an additional 
      arbitrary scaling vector $z \in \C^{m}$.
  \end{enumerate}
\end{theorem}
\begin{proof} \allowdisplaybreaks
  For brevity of the presentation, we restrict ourselves to prove Part~(c) of
  the theorem.
  Parts~(a) and~(b) can be proven analogously using the same projectors
  constructed in the following.
  The modified transfer functions of the reduced-order model are given by
  \begin{align*}
     c^{\herm} & \hmG{q + \eta}{\sigma_{1}, \ldots, \sigma_{q},
      \varsigma_{\kappa-\eta+1}, \ldots, \varsigma_{\kappa}}
      {d^{(1)}, \ldots, d^{(q-1)}, z, \delta^{(\kappa-\eta+1)}, \ldots, 
      \delta^{(\kappa-1)}} b \\
    & = \underbrace{c^{\herm} \hcC(\varsigma_{\kappa})
      \hcK(\varsigma_{\kappa})^{-1}
      \left( \prod\limits_{i = 1}^{\eta - 1}
      \hmN{\varsigma_{\kappa-i}}{\delta^{(\kappa-i)}}
      \hcK(\varsigma_{\kappa-i})^{-1} \right)}_{%
      \phantom{\,\hat{w}_{\eta}^{\herm}}=:\,\hat{w}_{\eta}^{\herm}}
      \hmN{\sigma_{q}}{z} \\
    & \quad \quad {}\times{} \underbrace{\left( \prod\limits_{j = 0}^{q-2} 
      \hcK(\sigma_{q-j})^{-1} \hmN{\sigma_{q-j-1}}{d^{(q-j-1)}} \right) 
      \hcK(\sigma_{1})^{-1} \hcB(\sigma_{1}) b}_{%
      \phantom{\,\hat{v}_{\eta}}=:\,\hat{v}_{\eta}}\\
    & = \hat{w}_{\eta}^{\herm} \hmN{\sigma_{q}}{z} \hat{v}_{q}\\
    & = \hat{w}_{\eta}^{\herm} W^{\herm} \mN{\sigma_{q}}{z} V \hat{v}_{q},
  \end{align*}
  for $1 \leq q \leq k$,  $1 \leq \eta \leq \kappa$, and an arbitrary vector
  $z \in \C^{m}$.
  The right-most product of the right-hand side can then be rewritten  using the construction of $V$ such
  that
  \begin{align*}
    V \hat{v}_{q} & = V \left( \prod\limits_{j = 0}^{q-3} 
      \hcK(\sigma_{q-j})^{-1} \hmN{\sigma_{q-j-1}}{d^{(q-j-1)}} \right) 
      \hcK(\sigma_{2})^{-1} \hmN{\sigma_{1}}{d^{(1)}} 
      \hcK(\sigma_{1})^{-1} \hcB(\sigma_{1}) b\\
    & = V \left( \prod\limits_{j = 0}^{q-3} 
      \hcK(\sigma_{q-j})^{-1} \hmN{\sigma_{q-j-1}}{d^{(q-j-1)}} \right) 
      \hcK(\sigma_{2})^{-1} W^{\herm} \mN{\sigma_{1}}{d^{(1)}}\\
    & \quad{}\times{} \underbrace{V \hcK(\sigma_{1})^{-1} W^{\herm} 
      \cK(\sigma_{1})}_{\phantom{\,P_{\rm{v}_{1}}}=:\,P_{\rm{v}_{1}}}
      \underbrace{\cK(\sigma_{1})^{-1} \cB(\sigma_{1})
      b}_{\phantom{\,v_{1}}=\,v_{1}}\\
    & = V \left( \prod\limits_{j = 0}^{q-3} 
      \hcK(\sigma_{q-j})^{-1} \hmN{\sigma_{q-j-1}}{d^{(q-j-1)}} \right) 
      \hcK(\sigma_{2})^{-1} W^{\herm} \mN{\sigma_{1}}{d^{(1)}} v_{1}\\
    & = \ldots\\
    & = V \hcK(\sigma_{q})^{-1} W^{\herm} \mN{\sigma_{q-1}}{d^{(q-1)}} v_{q-1}\\
    & = \underbrace{V \hcK(\sigma_{q})^{-1} W^{\herm} 
      \cK(\sigma_{q})}_{\phantom{\,P_{\rm{v}_{q}}}=:\,P_{\rm{v}_{q}}}  
      \underbrace{\cK(\sigma_{q})^{-1} \mN{\sigma_{q-1}}{d^{(q-1)}} 
      v_{q-1}}_{\phantom{\,v_{q}}=\,v_{q}}\\
    & = v_{q},
  \end{align*}
  where $P_{\rm{v}_{1}}, \ldots, P_{\rm{v}_{q}}$ are projectors onto 
  $\mspan(V)$, i.e., it holds $P_{\rm{v}_{j}} v = v$ for all $v \in \mspan(V)$
  and their recursive application gives the identity above.
  Analogously, one can show that
  \begin{align*}
    W \hat{w}_{\eta} & = w_{\eta},
  \end{align*}
 where $w_{1}, \ldots, w_{\eta} 
  \in \mspan(W)$.
  Combining this last equality together with $V \hat{v}_q = v_q$ yields
  \begin{align*}
    & c^{\herm} \hmG{q + \eta}{\sigma_{1}, \ldots, \sigma_{q},
      \varsigma_{\kappa-\eta+1}, \ldots, \varsigma_{\kappa}}
      {d^{(1)}, \ldots, d^{(q-1)}, z, \delta^{(\kappa-\eta+1)}, \ldots, 
      \delta^{(\kappa-1)}} b \\
    & = \hat{w}_{\eta}^{\herm} W^{\herm} \mN{\sigma_{q}}{z} V \hat{v}_{q}\\
    & = w_{\eta}^{\herm} \mN{\sigma_{q}}{z} v_{q}\\
    & = c^{\herm} \mG{q + \eta}{\sigma_{1}, \ldots, \sigma_{q},
      \varsigma_{\kappa-\eta+1}, \ldots, \varsigma_{\kappa}}
      {d^{(1)}, \ldots, d^{(q-1)}, z, \delta^{(\kappa-\eta+1)}, \ldots, 
      \delta^{(\kappa-1)}} b,
  \end{align*}
  which proves Part~(c).
\end{proof}

\begin{remark}[Implicit realization of blockwise interpolation]
  Part~(c) of \Cref{thm:tangmod} highlights an interesting interpolation
  property: The modified bilinear term in the middle between the interpolation
  by left and right projection allows for a completely arbitrary scaling
  vector $z$.
  Especially, by concatenation of higher-order transfer functions with respect
  to $z$, blockwise interpolation conditions hold true corresponding to the 
  centering bilinear term.
  To further illustrate this point via a simple example, construct $\mspan(V)$
  and $\mspan(W)$ as in~\Cref{thm:tangmod} such that $\mGf{\sigma} b$ and
  $c^{\herm} \mGf{\varsigma}$ are actively interpolated for 
  chosen interpolation points $\sigma, \varsigma \in \C$, and tangential
  directions $b \in \C^{m}$ and $c \in \C^{p}$.
  Then, by two-sided projection it holds additionally (Part~(c)
  of~\Cref{thm:tangmod}) that
  \begin{align*}
    \mG{2}{\sigma, \varsigma}{z} & = \hmG{2}{\sigma, \varsigma}{z},
  \end{align*}
  for all $z \in \C^{m}$.
  Choosing $z = \begin{bmatrix} 1 & 0 \end{bmatrix}^{\trans}$ and
  $z = \begin{bmatrix} 0 & 1 \end{bmatrix}^{\trans}$ yields the blockwise
  bi-tangential interpolation condition by concatenation:
  \begin{align*}
    c^{\herm} G_{2}(\sigma, \varsigma) (I_{m} \otimes b)
      & = c^{\herm} \hG_{2}(\sigma, \varsigma) (I_{m} \otimes b),
  \end{align*}
  More details on structure-preserving blockwise tangential interpolation and
  its relation to the unifying framework are shown later in
  \Cref{subsec:tangblock}.
\end{remark}

In addition to matching transfer function values, in practice, the interpolation
of sensitivities with respect to the frequency points, i.e., partial
derivatives, is crucial. 
The following theorem extends the interpolation results for modified transfer 
functions to Hermite interpolation.

\begin{theorem}[Modified structured tangential Hermite interpolation]%
  \label{thm:tangmodhermite}
  Let $G$ be a bilinear system, associated with the modified transfer 
  functions $\mGs_{k}$ in~\cref{eqn:tfmod}, and $\hG$ the reduced-order bilinear
  system, constructed by~\cref{eqn:proj} with its modified transfer functions
  $\hmGs_{k}$.
  Given sets of interpolation points $\sigma_{1}, \ldots, \sigma_{k} \in \C$ and
  $\varsigma_{1}, \ldots, \varsigma_{\kappa} \in \C$ such that the matrix 
  functions $\cC(s)$, $\cK(s)^{-1}$, $\cN(s)$, $\cB(s), \hcK(s)^{-1}$ are
  analytic in $s \in \{ \sigma_{1}, \ldots, \sigma_{k}, \varsigma_{1}, \ldots,
  \varsigma_{\kappa} \}$, two tangential directions $b \in \C^{m}$ and 
  $c \in \C^{p}$, and two sets of scaling vectors $d^{(1)}, \ldots, 
  d^{(k-1)} \in \C^{m}$ and $\delta^{(1)}, \ldots, \delta^{(\kappa-1)} \in
  \C^{m}$, the following statements hold:
  \begin{enumerate}[label=(\alph*)]
    \item If $V$ is constructed as
      \begin{align*}
        v_{1, j_{1}} & = \partial_{s^{j_{1}}} (\cK^{-1} \cB) (\sigma_{1}) b, &
          j_{1} & = 0, \ldots, \ell_{1},\\
        v_{2, j_{2}} & = \partial_{s^{j_{2}}} \cK^{-1} (\sigma_{2}) 
          \partial_{s^{\ell_{1}}} (\mN{.}{d^{(1)}} \cK^{-1} \cB) 
          (\sigma_{1}) b, &
          j_{2} & = 0, \ldots, \ell_{2},\\
        & \,\,\,\vdots \\
        v_{k, j_{k}} & = \partial_{s^{j_{k}}} \cK^{-1} (\sigma_{k})
          \left( \prod\limits_{j = 1}^{k - 2} \partial_{s^{\ell_{k-j}}} 
          (\mN{.}{d^{(k-j)}} \cK^{-1}) (\sigma_{k-j}) \right)\\
        & \quad{}\times{}
          \partial_{s^{\ell_{1}}} \left(\mN{.}{d^{(1)}} \cK^{-1} \cB\right) 
          (\sigma_{1}) b,
          & j_{k} & = 0,\ldots,\ell_{k},\\
        \mspan(V) & \supseteq \mspan([v_{1,0}, \ldots, v_{k, \ell_{k}}]),
      \end{align*}
      then the following interpolation conditions hold true:
      \vspace{-.25\baselineskip}
      \begin{align*} \arraycolsep=1.4pt \def\arraystretch{1.75}
        \begin{array}{rclrcl}
          \partial_{s_{1}^{j_{1}}} \mGf{\sigma_{1}} b & = &
            \partial_{s_{1}^{j_{1}}} \hmGf{\sigma_{1}} b, &
            \hspace{3em} j_{1} & = & 0,\ldots,\ell_{1},\\
          \partial_{s_{1}^{\ell_{1}} s_{2}^{j_{2}}}
            \mG{2}{\sigma_{1}, \sigma_{2}}{d^{(1)}} b & = & 
            \partial_{s_{1}^{\ell_{1}} s_{2}^{j_{2}}}
            \hmG{2}{\sigma_{1}, \sigma_{2}}{d^{(1)}} b, &
            j_{2} & = & 0,\ldots,\ell_{2},\\
          & \vdots & \\
          \multicolumn{3}{c}{
            \partial_{s_{1}^{\ell_{1}} \cdots s_{k-1}^{\ell_{k-1}} s_{k}^{j_{k}}}
            \mG{k}{\sigma_{1}, \ldots, \sigma_{k}}
            {d_{1}, \ldots, d_{k-1}} b}\qquad\qquad\\
          \multicolumn{3}{c}{
            = \partial_{s_{1}^{\ell_{1}} \cdots s_{k-1}^{\ell_{k-1}} 
              s_{k}^{j_{k}}} \hmG{k}{\sigma_{1}, \ldots, \sigma_{k}}
              {d_{1}, \ldots, d_{k-1}} b,} &
            j_{k} & = & 0,\ldots,\ell_{k}.
        \end{array}
      \end{align*}
    \item If $W$ is constructed as\\
      \resizebox{\linewidth}{!}{
      \begin{minipage}{1.01\linewidth}
      \begin{align*}
        w_{1,i_{\kappa}} & = \partial_{s^{i_{\kappa}}} \left(\cK^{-\herm}
          \cC^{\herm}\right)(\varsigma_{\kappa}) c,
          & i_{\kappa} & = 0,\ldots,\nu_{\kappa},\\
        w_{2,i_{\kappa-1}} & = \partial_{s^{i_{\kappa-1}}}\left(\cK^{-\herm} 
          \mN{.}{\delta^{(\kappa-1)}}^{\herm}\right) (\varsigma_{\kappa-1})
          \partial_{s^{\nu_{\kappa}}} (\cK^{-\herm} \cC^{\herm})
          (\varsigma_{\kappa}) c,
          & i_{\kappa-1} & = 0,\ldots,\nu_{\kappa-1},\\
        & \,\,\, \vdots\\
        w_{\kappa,i_{1}} & = \partial_{s^{i_{1}}} \left(\cK^{-\herm}
          \mN{.}{\delta^{(1)}}^{\herm}\right) (\varsigma_{1})\\
        & \quad{}\times{} \left( \prod\limits_{i = 2}^{\kappa - 1}
          \partial_{s^{\nu_{i}}} \left(\cK^{-\herm}
          \mN{.}{\delta^{(i)}}^{\herm}\right)
          (\varsigma_{i}) \right)\\
        & \quad{}\times{} \partial_{s^{\nu_{\kappa}}} \left(\cK^{-\herm}
          \cC^{\herm}\right) (\varsigma_{\kappa}) c,
          & i_{1} & = 0,\ldots,\nu_{1},\\
        \mathrm{span}(W) & \supseteq \mspan([w_{1,0}, \ldots, 
          w_{\kappa, \nu_{\kappa}}]),
      \end{align*}
      \end{minipage}}\\[\baselineskip]
      then the following interpolation conditions hold true:
      \vspace{-.25\baselineskip}
      \begin{align*} \arraycolsep=1.4pt \def\arraystretch{1.75}
        \begin{array}{rclrcl}
          c^{\herm} \partial_{s_{1}^{i_{\kappa}}} \mGf{\varsigma_{\kappa}}
            & = & c^{\herm} \partial_{s_{1}^{i_{\kappa}}}
            \hmGf{\varsigma_{\kappa}},
            & \hspace{1.2em} i_{\kappa} & = & 0,\ldots,\nu_{\kappa}, \\
          c^{\herm} \partial_{s_{1}^{i_{\kappa-1}} s_{2}^{\nu_{\kappa}}}
            \mG{2}{\varsigma_{\kappa-1}, \varsigma_{\kappa}}
            {\delta^{(\kappa-1)}} & = & c^{\herm} \partial_{s_{1}^{i_{\kappa-1}}
            s_{2}^{\nu_{\kappa}}} \hmG{2}{\varsigma_{\kappa-1}, 
            \varsigma_{\kappa}}{\delta^{(\kappa-1)}},
            & i_{\kappa-1} & = & 0,\ldots,\nu_{\kappa-1}, \\
          & \vdots & \\
          \multicolumn{3}{c}{c^{\herm} \partial_{s_{1}^{i_{1}} s_{2}^{\nu_{2}}
            \cdots s_{\kappa}^{\nu_{\kappa}}} \mG{\kappa}{\varsigma_{1}, \ldots,
            \varsigma_{\kappa}}{\delta_{1}, \ldots, \delta_{\kappa-1}}}
            \qquad\qquad\\
          \multicolumn{3}{c}{= c^{\herm} \partial_{s_{1}^{i_{1}} s_{2}^{\nu_{2}}
            \cdots s_{\kappa}^{\nu_{\kappa}}} \hmG{\kappa}{\varsigma_{1},\ldots,
            \varsigma_{\kappa}}{\delta_{1}, \ldots, \delta_{\kappa-1}},}
            & i_{1} & = & 0,\ldots,\nu_{1}.
        \end{array}
      \end{align*}
    \item Let $V$ be constructed as in Part~(a) and $W$ as in Part~(b).
      Then, additionally to the results in~(a) and~(b), the following
      conditions hold:
      \begin{align*}
        & c^{\herm} \partial_{s_{1}^{\ell_{1}} \cdots s_{q-1}^{\ell_{q-1}}
          s_{q}^{j_{q}} s_{q+1}^{i_{\kappa - \eta + 1}}
          s_{q+2}^{\nu_{\kappa - \eta + 2}} s_{q + \eta}^{\nu_{\kappa}}}
          \mGs_{q + \eta} (\sigma_{1}, \ldots, \sigma_{q},
          \varsigma_{\kappa-\eta+1}, \ldots, \varsigma_{\kappa} \,|\, \\
        & \qquad\quad d^{(1)}, \ldots, d^{(q-1)}, z,
          \delta^{(\kappa-\eta+1)}, \ldots, \delta^{(\kappa-1)}) b\\
        & \quad = c^{\herm} \partial_{s_{1}^{\ell_{1}} \cdots s_{q-1}^{\ell_{q-1}}
          s_{q}^{j_{q}} s_{q+1}^{i_{\kappa - \eta + 1}}
          s_{q+2}^{\nu_{\kappa - \eta + 2}} s_{q + \eta}^{\nu_{\kappa}}}
          \hmGs_{q + \eta} (\sigma_{1}, \ldots, \sigma_{q},
          \varsigma_{\kappa-\eta+1}, \ldots, \varsigma_{\kappa} \,|\, \\
        & \qquad\quad d^{(1)}, \ldots, d^{(q-1)}, z,
          \delta^{(\kappa-\eta+1)}, \ldots, \delta^{(\kappa-1)}) b,
      \end{align*}
      for $j_{q} = 0, \ldots, \ell_{q}$; $i_{\kappa - \eta + 1} = 0, \ldots,
      \nu_{\kappa - \eta + 1}$; $1 \leq q \leq k$, $1 \leq \eta \leq \kappa$,
      and an additional arbitrary scaling vector $z \in \C^{m}$.
  \end{enumerate}
\end{theorem}
\begin{proof}
  The proof works analogously to \Cref{thm:tangmod}, using appropriate
  projectors onto $\mspan(V)$ or $\mspan(W)$ and the ideas from
  the proof of~\cite[Thm.~9]{BenGW21a} for fixed scaling vectors
  $d^{(1)}, \ldots, d^{(k-1)}$ and $\delta^{(1)}, \ldots, \delta^{(\kappa-1)}$.
\end{proof}

To complete the theory for our new unifying interpolation framework,
we consider the special cases of \Cref{thm:tangmod,thm:tangmodhermite} 
by using identical sets of interpolation points and scaling vectors in the 
bi-tangential interpolation case.
As in~\cite{BeaG09, BenGW21a}, this allows interpolation of partial
derivatives implicitly.
Due to the dependency of the modified transfer functions on the scaling vectors,
we will also interpolate now derivatives with respect to those scaling vectors.
Therefore, the notion of the Jacobian matrix~\cref{eqn:jacobi} for the modified
transfer functions will be given as
\begin{align*}
  \nabla \mGs_{k} & = \left[ \partial_{s_{1}} \mGs_{k}, \ldots, 
    \partial_{s_{k}} \mGs_{k}, \partial_{d^{(1)}_{1}} \mGs_{k}, \ldots, 
    \partial_{d^{(1)}_{m}} \mGs_{k}, \ldots, \partial_{d^{(k-1)}_{1}} \mGs_{k}, 
    \ldots, \partial_{d^{(k-1)}_{m}} \mGs_{k} \right].
\end{align*}

\begin{theorem}[Modified structured bi-tangential interpolation with 
  identical point sets]%
  \label{thm:tangmodtwosided}
  Let $G$ be a bilinear system, associated with the modified transfer 
  functions $\mGs_{k}$ in~\cref{eqn:tfmod}, and $\hG$ the reduced-order bilinear
  system, constructed by~\cref{eqn:proj} with its modified transfer functions
  $\hmGs_{k}$.
  Given a set of interpolation points $\sigma_{1}, \ldots, \sigma_{k} \in \C$
  such that the matrix functions $\cC(s)$, $\cK(s)^{-1}$, $\cN(s)$, $\cB(s)$,
  $\hcK(s)^{-1}$ are analytic in $s \in \{ \sigma_{1}, \ldots, \sigma_{k} \}$,
  two tangential directions $b \in \C^{m}$ and $c \in \C^{p}$, and scaling
  vectors $d^{(1)}, \ldots, d^{(k-1)} \in \C^{m}$, the following statements
  hold:
  \begin{enumerate}[label=(\alph*)]
    \item Let $V$ and $W$ be constructed as in \Cref{thm:tangmod} 
      Parts~(a) and~(b) for the interpolation points
      $\sigma_{1} = \varsigma_{1}$, $\ldots$, $\sigma_{k} = \varsigma_{k}$ and
      the scaling vectors $d^{(1)} = \delta^{(1)}$, $\ldots$,
      $d^{(k-1)} = \delta^{(k-1)}$. Then, in addition to the interpolation
      conditions in \Cref{thm:tangmod}, it holds
      \begin{align*}
        & \nabla \left( c^{\herm} \mGs_{k} b \right)
          \margs{\sigma_{1}, \ldots, \sigma_{k}}{d^{(1)}, \ldots, d^{(k-1)}}\\
        & \qquad = \nabla \left( c^{\herm} \hmGs_{k} b \right)
          \margs{\sigma_{1}, \ldots, \sigma_{k}}{d^{(1)}, \ldots, d^{(k-1)}}.
      \end{align*}
    \item Let $V$ and $W$ be constructed as in \Cref{thm:tangmodhermite} 
      Parts~(a) and~(b) for the interpolation points
      $\sigma_{1} = \varsigma_{1}$, $\ldots$, $\sigma_{k} = \varsigma_{k}$,
      the derivative orders $\ell_{1} = \nu_{1}$, $\ldots$,
      $\ell_{k} = \nu_{k}$,
      and the scaling vectors $d^{(1)} = \delta^{(1)}$, $\ldots$,
      $d^{(k-1)} = \delta^{(k-1)}$.
      Then, in addition to the interpolation conditions in
      \Cref{thm:tangmodhermite}, it holds
      \begin{align*}
        & \nabla \left( c^{\herm} \partial_{s_{1}^{\ell_{1}} \cdots
          s_{k}^{\ell_{k}}} \mGs_{k} b \right)
          \margs{\sigma_{1}, \ldots, \sigma_{k}}{d^{(1)}, \ldots, d^{(k-1)}}\\
        & \qquad = \nabla \left( c^{\herm} \partial_{s_{1}^{\ell_{1}} \cdots
          s_{k}^{\ell_{k}}} \hmGs_{k} b \right)
          \margs{\sigma_{1}, \ldots, \sigma_{k}}{d^{(1)}, \ldots, d^{(k-1)}}.
      \end{align*}
  \end{enumerate}
\end{theorem}
\begin{proof} \allowdisplaybreaks
  First, we consider the partial derivatives with respect to the scaling
  vectors.
  For arbitrary $1 \leq j \leq k-1$ and $1 \leq i \leq m$, we obtain
  \begin{align*}
    & \partial_{d^{(j)}_{i}} \left( c^{\herm} \hmGs_{k} b \right)
      \margs{\sigma_{1}, \ldots, \sigma_{k}}{d^{(1)}, \ldots, d^{(k-1)}}\\
    & = \underbrace{c^{\herm} \hcC(\sigma_{k}) \hcK(\sigma_{k})^{-1} \left(
      \prod\limits_{\ell = 1}^{k-j-1} \hmN{\sigma_{k-\ell}}{d^{(k-\ell)}}
      \hcK(\sigma_{k-\ell})^{-1} \right)}_{%
      \phantom{\,\hat{w}_{k-j-1}^{\herm}}=:\,\hat{w}_{k-j-1}^{\herm}}
      \left( \partial_{d^{(j)}_{i}}
      \hmN{\sigma_{j}}{d^{(j)}} \right)\\
    & \quad{}\times{}
      \underbrace{\left( \prod\limits_{\ell = j+1}^{k-1} 
      \hmN{\sigma_{k-\ell}}{d^{(k-\ell)}} \hcK(\sigma_{k-\ell})^{-1} \right)
      \hcB(s_{1}) b}_{%
      \phantom{\,\hat{v}_{k-j-1}}=:\,\hat{v}_{k-j-1}}\\
    & = \hat{w}_{k-j-1}^{\herm} \left( \partial_{d^{(j)}_{i}}
      \hmN{\sigma_{j}}{d^{(j)}} \right) \hat{v}_{k-j-1}\\
    & = \hat{w}_{k-j-1}^{\herm} W^{\herm} \left( \partial_{d^{(j)}_{i}}
      \mN{\sigma_{j}}{d^{(j)}} \right) V \hat{v}_{k-j-1}
  \end{align*}
  such that only the modified bilinear term corresponding to the scaling vector
  $d^{(j)}$ needs to be differentiated.
  Using  the same approach as in the proof of \Cref{thm:tangmod} and the
  construction of $\mspan(V)$ and $\mspan(W)$ yields the two equalities
  \begin{align*}
    \begin{aligned}
      V \hat{v}_{k-j-1} & = v_{k-j-1} & \text{and} &&
        W \hat{w}_{k-j-1} & = w_{k-j-1},
    \end{aligned}
  \end{align*}
  which gives
  \begin{align*}
    & \partial_{d^{(j)}_{i}} \left( c^{\herm} \hmGs_{k} b \right)
      \margs{\sigma_{1}, \ldots, \sigma_{k}}{d^{(1)}, \ldots, d^{(k-1)}}\\
    & \qquad = \hat{w}_{k-j-1}^{\herm} W^{\herm} \left( \partial_{d^{(j)}_{i}}
      \mN{\sigma_{j}}{d^{(j)}} \right) V \hat{v}_{k-j-1}\\
    & \qquad = w_{k-j-1}^{\herm} \left( \partial_{d^{(j)}_{i}}
      \mN{\sigma_{j}}{d^{(j)}} \right) v_{k-j-1}\\
    & \qquad = \partial_{d^{(j)}_{i}} \left( c^{\herm} \mGs_{k} b \right)
      \margs{\sigma_{1}, \ldots, \sigma_{k}}{d^{(1)}, \ldots, d^{(k-1)}},
  \end{align*}
  for all $1 \leq j \leq k-1$ and $1 \leq i \leq m$.
  Therefore, the interpolation condition holds for all partial derivatives with
  respect to the scaling vectors.
  The results for the partial derivatives with respect to the frequency
  arguments can be proven analogously and in principle follow the ideas
  from~\cite[Cor.~2]{BenGW21a}.
  This proves Part~(a).
  Part~(b) can be proven analogously to Part~(a) by replacing the simple
  interpolation by the Hermite version from \Cref{thm:tangmodhermite}.
  For brevity of the paper, we skip those details.
\end{proof}

\begin{remark}[Using multiple sets of interpolation points]
  While all results in this section are formulated for a single set of
  interpolation points $\sigma_{1}, \ldots, \sigma_{k} \in \C$, they can
  be extended to multiple sets by concatenation of the model reduction
  bases.
  Consider, for example, Part~(a) of \Cref{thm:tangmod}.
  Let $\sigma_{1}^{(1)}, \ldots, \sigma_{k}^{(1)}$, $\ldots$,
  $\sigma_{1}^{(n_{\rm{s}})}, \ldots, \sigma_{k}^{(n_{\rm{s}})} \in \C$
  be $n_{\rm{s}}$ sets of interpolation points and $V^{(1)}, \ldots,
  V^{(n_{\rm{s}})}$ be the corresponding basis matrices such that the
  corresponding reduced-order models (tangentially) interpolate the original
  model for the given sets of interpolation points.
  Then, another reduced-order model can be constructed to satisfy all
  interpolation conditions associated with $V^{(1)}, \ldots,
  V^{(n_{\rm{s}})}$ by choosing
  \begin{align*}
    \mspan(V) \supseteq \mspan([V^{(1)}, \ldots, V^{(n_{\rm{s}})}]),
  \end{align*}
  as the new truncation matrix $V$, and any $W$ of appropriate dimension and
  full column rank.
\end{remark}

%%%%%%%%%%%%%%%%%%%%%%%%%%%%%%%%%%%%%%%%%%%%%%%%%%%%%%%%%%%%%%%%%%%%%%%%%%%%%%%%

\subsection{Special case: Structured blockwise tangential interpolation}%
\label{subsec:tangblock}

As mentioned in \Cref{subsec:tangmod}, the new unifying tangential interpolation
framework can also be used to obtain results for blockwise tangential
interpolation.
Due to its relevance and common use in model reduction of MIMO bilinear systems,
we will state the corresponding results in this section in more detail.

First, we will generalize the idea of blockwise tangential interpolation
introduced in \Cref{subsec:tangblockstd} to the structured case.
Therefore, we start by analyzing the multivariate transfer
functions~\cref{eqn:tf}.
Multiplying out the Kronecker products, we observe that~\cref{eqn:tf} is
actually given as concatenation of products of the linear dynamics and the
bilinear terms 
\begin{align} \label{eqn:tfblocks}
  \begin{aligned}
    G_{k}(s_{1}, \ldots, s_{k}) & = \big[ \cC(s_{k}) \cK(s_{k})^{-1} 
      \cN_{1}(s_{k-1}) \cK(s_{k-1})^{-1} \cdots \cN_{1}(s_{1}) \cK(s_{1})^{-1} 
      \cB(s_{1}),\\
    & \phantom{= \big[~} \cC(s_{k}) \cK(s_{k})^{-1} \cN_{1}(s_{k-1})
      \cK(s_{k-1})^{-1} \cdots \cN_{2}(s_{1}) \cK(s_{1})^{-1} \cB(s_{1}), \\
    & \phantom{= \big[~} \ldots\\
    & \phantom{= \big[~} \cC(s_{k}) \cK(s_{k})^{-1} \cN_{m}(s_{k-1}) 
      \cK(s_{k-1})^{-1} \cdots \cN_{m}(s_{1}) \cK(s_{1})^{-1} \cB(s_{1}) \big].
  \end{aligned}
\end{align}
Extending on the ideas from \Cref{subsec:tangblockstd},
we consider each block entry of~\cref{eqn:tfblocks} as separate transfer
function and for each of them use tangential interpolation with the same
directions.
In other words, given the right tangential direction $b \in \C^{m}$, we consider
\begin{align*}
  \begin{aligned}
    G_{k}(s_{1}, \ldots, s_{k}) (I_{m} \otimes b) 
    & = \big[ \cC(s_{k}) \cK(s_{k})^{-1} \cN_{1}(s_{k-1}) \cK(s_{k-1})^{-1} 
      \cdots \cN_{1}(s_{1}) \cK(s_{1})^{-1} \cB(s_{1}) b, \\
    & \phantom{= \big[~} \cC(s_{k}) \cK(s_{k})^{-1} \cN_{1}(s_{k-1})
      \cK(s_{k-1})^{-1} \cdots \cN_{2}(s_{1}) \cK(s_{1})^{-1} \cB(s_{1}) b, \\
    & \phantom{= \big[~} \ldots\\
    & \phantom{= \big[~} \cC(s_{k}) \cK(s_{k})^{-1} \cN_{m}(s_{k-1})
      \cK(s_{k-1})^{-1} \cdots \cN_{m}(s_{1}) \cK(s_{1})^{-1} \cB(s_{1}) b \big]
  \end{aligned}
\end{align*}
as blockwise evaluation of the transfer function in the direction $b$.
This formulation extends the blockwise tangential interpolation
problem from~\cref{eqn:tangblock1,eqn:tangblock2,eqn:tangblock3} to the
structure-preserving setting.

The modified tangential interpolation framework can now be used to obtain 
the subspace conditions on the blockwise tangential interpolation.
Choose the scaling vectors $d^{(j)}$ in~\cref{eqn:tfmod} to be
columns of the $m$-dimensional identity matrix.
Then, the single block entries of~\cref{eqn:tfblocks} are given as the modified
transfer functions~\cref{eqn:tfmod} for specific choices of 
scaling vectors. For example, choosing $d^{(1)} = \ldots = d^{(k-1)} = e_{1}$
to be the first column of the $m$-dimensional identity matrix yields
\begin{align*}
  \mG{k}{s_{1}, \ldots, s_{k}}{e_{1}, \ldots, e_{1}} = \cC(s_{k})
    \cK(s_{k})^{-1} \cN_{1}(s_{k-1}) \cK(s_{k-1})^{-1} \cdots \cN_{1}(s_{1}) 
    \cK(s_{1})^{-1} \cB(s_{1}),
\end{align*}
which is the first block in~\cref{eqn:tfblocks}.
By column concatenation of these modified transfer functions,
\cref{eqn:tfblocks}~can be completely recovered:
\begin{align} \label{eqn:blockmod}
  \begin{aligned}
    G_{k}(s_{1}, \ldots, s_{1}) &
      = \big[ \mG{k}{s_{1}, \ldots, s_{k}}{e_{1}, \ldots, e_{1}},\\
    & \phantom{= \big[~} \mG{k}{s_{1}, \ldots, s_{k}}{e_{1}, \ldots, e_{2}}, \\
    & \phantom{= \big[~} \ldots,\\
    & \phantom{= \big[~} \mG{k}{s_{1}, \ldots, s_{k}}{e_{m}, \ldots, e_{m}}
      \big].
  \end{aligned}
\end{align}
Consequently, the blockwise interpolation results are given by concatenation of
the corresponding model reduction bases constructed for all necessary modified
transfer functions and the tangential directions.
Due to the significance of the blockwise tangential interpolation in the
literature~\cite{BenB12, RodGB18} and the complexity of its recovery from
the unifying framework, we will state in the following the
structure-preserving interpolation results for blockwise tangential
interpolation.
Note that the proofs directly follow from the previous section and by
concatenation as discussed above.

\begin{remark}[Matrix interpolation]
  It should be noted that the matrix interpolation results
  from~\cite{BenGW21a} can also be recovered from the modified tangential
  interpolation framework.
  As the relation~\cref{eqn:blockmod} shows, removing the
  tangential directions in the construction of the projection spaces
  will yield the matrix interpolation results.
  Thus matrix interpolation is also a special case of the modified tangential
  interpolation framework.
\end{remark}

The first result follows from \Cref{thm:tangmod}.

\begin{corollary}[Structured blockwise tangential interpolation]%
  \label{cor:tangblock}
  Let $G$ be a bilinear system, described by its subsystem transfer functions
  in~\cref{eqn:tf}, and $\hG$ the reduced-order bilinear system, constructed
  by~\cref{eqn:proj} with the corresponding subsystem transfer functions
  $\hG_{k}$.
  Given sets of interpolation points $\sigma_{1}, \ldots, \sigma_{k} \in \C$ and
  $\varsigma_{1}, \ldots, \varsigma_{\kappa} \in \C$ such that the matrix 
  functions $\cC(s)$, $\cK(s)^{-1}$, $\cN(s)$, $\cB(s), \hcK(s)^{-1}$ are
  defined for $s \in \{ \sigma_{1}, \ldots, \sigma_{k}, \varsigma_{1}, \ldots,
  \varsigma_{\kappa} \}$, and two tangential directions $b \in \C^{m}$ 
  and $c \in \C^{p}$, the following statements hold:
  \begin{enumerate}[label=(\alph*)]
    \item If $V$ is constructed as
      \begin{align*}
        V_{1} & = \cK(\sigma_{1})^{-1}\cB(\sigma_{1})b,\\
        V_{j} & = \cK(\sigma_{j})^{-1}\cN(\sigma_{j-1})
          (I_{m} \otimes V_{j-1}), & 2 \leq j \leq k,\\
        \mspan(V) & \supseteq \mspan\left([V_{1}, \ldots,
        V_{k}]\right),
      \end{align*}
      then the following interpolation conditions hold true:
      \begin{align*}
        \begin{aligned}
          G_{1}(\sigma_{1})b & = \hG_{1}(\sigma_{1})b,\\
          G_{2}(\sigma_{1}, \sigma_{2}) (I_{m} \otimes b) & = 
            \hG_{1}(\sigma_{1}) (I_{m} \otimes b),\\
          & \,\,\,\vdots \\
          G_{k}(\sigma_{1}, \ldots, \sigma_{k}) (I_{m^{k-1}} \otimes b) &
            = \hG_{k}(\sigma_{1}, \ldots, \sigma_{k}) (I_{m^{k-1}} \otimes b).
        \end{aligned}
      \end{align*}
    \item If $W$ is constructed as
      \begin{align*}
        W_{1} & = \cK(\varsigma_{\kappa})^{-\herm}\cC
          (\varsigma_{\kappa})^{\herm} c,\\
        W_{i} & = \cK(\varsigma_{\kappa-i+1})^{-\herm}
          \tcN(\varsigma_{k-i+1})^{\herm} (I_{m} \otimes W_{i-1}),
          & 2 \leq i \leq \kappa,\\
        \mathrm{span}(W) & \supseteq \mathrm{span}\left([W_{1}, \ldots,
          W_{\kappa}] \right),
      \end{align*}
      then the following interpolation conditions hold true:
      \begin{align*}
        \begin{aligned}
          c^{\herm} G_{1}(\varsigma_{\kappa}) &
            = c^{\herm} \hG_{1}(\varsigma_{\kappa}),\\
          c^{\herm} G_{2}(\varsigma_{\kappa-1}, \varsigma_{\kappa}) &
            = c^{\herm} \hG_{2}(\varsigma_{\kappa-1}, \varsigma_{\kappa}),\\
          & \,\,\,\vdots \\
          c^{\herm} G_{\kappa}(\varsigma_{1}, \ldots, \varsigma_{\kappa}) &
            = c^{\herm} \hG_{\kappa}(\varsigma_{1}, \ldots, \varsigma_{\kappa}).
        \end{aligned}
      \end{align*}
    \item Let $V$ be constructed as in Part~(a) and $W$ as in Part~(b).
      Then, additionally to the results in~(a) and~(b), the following conditions
      hold:
      \begin{align*}
        & c^{\herm} G_{q + \eta}(\sigma_{1}, \ldots, \sigma_{q},
          \varsigma_{\kappa-\eta+1}, \ldots, \varsigma_{\kappa})
          (I_{m^{q+\eta-1}} \otimes b)\\
        & \qquad = c^{\herm} \hG_{q + \eta}(\sigma_{1}, \ldots, \sigma_{q},
          \varsigma_{\kappa-\eta+1}, \ldots, \varsigma_{\kappa})
          (I_{m^{q+\eta-1}} \otimes b),
      \end{align*}
      for $1 \leq q \leq k$ and  $1 \leq \eta \leq \kappa$.
  \end{enumerate}
\end{corollary}

The next corollary corresponds to \Cref{thm:tangmodhermite} stating the results
for Hermite interpolation.

\begin{corollary}[Structured blockwise tangential Hermite interpolation]%
  \label{cor:tangblockhermite}
  Let $G$ be a bilinear system, described by its subsystem transfer functions
  in~\cref{eqn:tf}, and $\hG$ the reduced-order bilinear system, constructed
  by~\cref{eqn:proj} with the corresponding subsystem transfer functions
  $\hG_{k}$.
  Given sets of interpolation points $\sigma_{1}, \ldots, \sigma_{k} \in \C$ and
  $\varsigma_{1}, \ldots, \varsigma_{\kappa} \in \C$ such that the matrix 
  functions $\cC(s)$, $\cK(s)^{-1}$, $\cN(s)$, $\cB(s)$, $\hcK(s)^{-1}$ are
  analytic in $s \in \{ \sigma_{1}, \ldots, \sigma_{k}, \varsigma_{1}, \ldots,
  \varsigma_{\kappa} \}$, and two tangential directions $b \in \C^{m}$ 
  and $c \in \C^{p}$, the following statements hold:
  \begin{enumerate}[label=(\alph*)]
    \item If $V$ is constructed as
      \begingroup \small
      \begin{align*}
        V_{1, j_{1}} & = \partial_{s^{j_{1}}} (\cK^{-1} \cB b) (\sigma_{1}),
          & j_{1} & = 0, \ldots, \ell_{1},\\
        V_{2, j_{2}} & = \partial_{s^{j_{2}}} \cK^{-1} (\sigma_{2})
          \partial_{s^{\ell_{1}}} (\cN (I_{m} \otimes \cK^{-1} \cB b))
          (\sigma_{1}),
          & j_{2} & = 0, \ldots, \ell_{2},\\
        & \,\,\,\vdots\\
        V_{k,j_{k}} & = \partial_{s^{j_{k}}} \cK^{-1} (\sigma_{k})
          \left( \prod\limits_{j = 1}^{k-2} \partial_{s^{\ell_{k -j}}}
          ( (I_{m^{j-1}} \otimes \cN) (I_{m^{j}} \otimes \cK) )
          (\sigma_{k-j}) \right)\\
        & \quad{}\times{} \partial_{s^{\ell_{1}}}
          ((I_{m^{k-2}} \otimes \cN)(I_{m^{k-1}} \otimes \cK \cB b))
          (\sigma_{1}),
          & j_{k} & = 0, \ldots, \ell_{k},\\
        \mspan(V) & \supseteq \mspan([V_{1,0}, \ldots, V_{k, \ell_{k}}]),
      \end{align*}
      \endgroup
      then the following interpolation conditions hold true:
      \begin{align*}
        & \partial_{s_{1}^{j_{1}}} G_{1} (\sigma_{1}) b
          = \partial_{s_{1}^{j_{1}}} \hG_{1} (\sigma_{1}) b,
          & j_{1} & = 0, \ldots, \ell_{1}, \\
        & \hspace*{5.5em} \vdots \\
        & \partial_{s_{1}^{\ell_{1}} \cdots s_{k-1}^{\ell_{k-1}} s_{k}^{j_{k}}}
          G_{k} (\sigma_{1}, \ldots, \sigma_{k}) (I_{m^{k-1}} \otimes b) \\
          &\qquad = \partial_{s_{1}^{\ell_{1}} \cdots s_{k-1}^{\ell_{k-1}}
          s_{k}^{j_{k}}} \hG_{k} (\sigma_{1}, \ldots, \sigma_{k})
          (I_{m^{k-1}} \otimes b),
          & j_{k} & = 0, \ldots, \ell_{k}.
      \end{align*}
    \item If $W$ is constructed as
      \begin{align*}
        W_{1, i_{\kappa}} & = \partial_{s^{i_{\kappa}}} (\cK^{-\herm}
          \cC^{\herm} c) (\varsigma_{\kappa}),
          & i_{\kappa} & = 0,\ldots,\nu_{\kappa},\\
        W_{2,i_{\kappa-1}} & = \partial_{s^{i_{\kappa-1}}} (\cK^{-\herm}
          \tcN^{\herm}) (\varsigma_{\kappa-1})
          \left( I_{m} \otimes \partial_{s^{\nu_{\kappa}}} (\cK^{-\herm}
          \cC^{\herm} c) (\varsigma_{\kappa}) \right),
        & i_{\kappa-1} & = 0,\ldots,\nu_{\kappa-1},\\
        & \,\,\,\vdots\\
        W_{\kappa,i_{1}} & = \partial_{s^{i_{1}}} (\cK^{-\herm} \tcN^{\herm})
          (\varsigma_{1}) \left( \prod\limits_{i = 2}^{\kappa - 1}
          \partial_{s^{\nu_{i}}} (I_{m^{i-1}} \otimes \cK^{-\herm} \tcN^{\herm})
          (\varsigma_{i}) \right)\\
        & \quad{}\times{} \left( I_{m^{\kappa-1}} \otimes
          \partial_{s^{\nu_{\kappa}}} (\cK^{-\herm} \cC^{\herm} c)
          (\varsigma_{\kappa}) \right),
          & i_{1} & = 0,\ldots,\nu_{1},\\
        \mathrm{span}(W) & \supseteq \mspan([W_{1,0}, \ldots, 
          W_{\kappa, \nu_{\kappa}}]),
      \end{align*}
      then the following interpolation conditions hold true:
      \begin{align*}
        c^{\herm} \partial_{s_{1}^{i_{\kappa}}} G_{1} (\varsigma_{\kappa})
          & = c^{\herm} \partial_{s_{1}^{i_{\kappa}}} \hG_{1}
          (\varsigma_{\kappa}),
          & i_{\kappa} = 0,\ldots,\nu_{\kappa},\\
        & \,\,\,\vdots\\
        c^{\herm} \partial_{s_{1}^{i_{1}} s_{2}^{\nu_{2}} \cdots 
          s_{\kappa}^{\nu_{\kappa}}} G_{\kappa} (\varsigma_{1}, \ldots,
          \varsigma_{\kappa})
          & = c^{\herm} \partial_{s_{1}^{i_{1}} s_{2}^{\nu_{2}} \cdots 
          s_{\kappa}^{\nu_{\kappa}}} \hG_{\kappa} (\varsigma_{1}, \ldots,
          \varsigma_{\kappa}),
          & i_{1} = 0,\ldots,\nu_{1}.
      \end{align*}
    \item Let $V$ be constructed as in Part~(a) and $W$ as in Part~(b).
      Then, additionally to the interpolation conditions in~(a) and~(b), the 
      following conditions hold:
      \begin{align*}
        & c^{\herm} \partial_{s_{1}^{\ell_{1}} \cdots s_{q-1}^{\ell_{q-1}}
          s_{q}^{j_{q}} s_{q+1}^{i_{\kappa - \eta + 1}} 
          s_{q+2}^{\nu_{\kappa - \eta + 2}} \cdots s_{q + \eta}^{\nu_{\kappa}}}
          G_{q + \eta} (\sigma_{1}, \ldots, \sigma_{q},
          \varsigma_{\kappa - \eta + 1}, \ldots, \varsigma_{\kappa})
          (I_{m^{q+\eta-1}} \otimes b)\\
        & = c^{\herm} \partial_{s_{1}^{\ell_{1}} \cdots s_{q-1}^{\ell_{q-1}}
          s_{q}^{j_{q}} s_{q+1}^{i_{\kappa - \eta + 1}} 
          s_{q+2}^{\nu_{\kappa - \eta + 2}} \cdots s_{q + \eta}^{\nu_{\kappa}}}
          \hG_{q + \eta} (\sigma_{1}, \ldots, \sigma_{q},
          \varsigma_{\kappa - \eta + 1}, \ldots, \varsigma_{\kappa})
          (I_{m^{q+\eta-1}} \otimes b),
      \end{align*}
      for $j_{q} = 0, \ldots, \ell_{q}$; $i_{\kappa - \eta + 1} = 0, \ldots,
      \nu_{\kappa - \eta + 1}$; $1 \leq q \leq k$ and $1 \leq \eta \leq \kappa$.
  \end{enumerate}
\end{corollary}

Last, we give the results on implicit blockwise tangential interpolation of
additional partial derivatives by using two-sided projection, corresponding
to \Cref{thm:tangmodtwosided}.

\begin{corollary}[Structured blockwise bi-tangential interpolation with 
  identical point sets]%
  \label{cor:tangblocktwosided}
  Let $G$ be a bilinear system, described by its subsystem transfer functions
  in~\cref{eqn:tf}, and $\hG$ the reduced-order bilinear system, constructed
  by~\cref{eqn:proj} with the corresponding subsystem transfer functions
  $\hG_{k}$.
  Given a set of interpolation points $\sigma_{1}, \ldots, \sigma_{k} \in \C$
  such that the matrix functions $\cC(s)$, $\cK(s)^{-1}$, $\cN(s)$, $\cB(s)$,
  $\hcK(s)^{-1}$ are analytic in $s \in \{ \sigma_{1}, \ldots, \sigma_{k} \}$,
  and two tangential directions $b \in \C^{m}$ and $c \in \C^{p}$, the
  following statements hold:
  \begin{enumerate}[label=(\alph*)]
    \item Let $V$ and $W$ be constructed as in \Cref{cor:tangblock} Parts~(a)
      and~(b) for the interpolation points $\sigma_{1} = \varsigma_{1}$,
      $\ldots$, $\sigma_{k} = \varsigma_{k}$.
      Then, in addition to the interpolation conditions in \Cref{cor:tangblock},
      it holds
      \begin{align*}
        & \nabla  \big( c^{\herm}  G_{k} (I_{m^{k-1}} \otimes b) \big)
          (\sigma_{1}, \ldots, \sigma_{k})
         = \nabla \big( c^{\herm} \hG_{k} (I_{m^{k-1}} \otimes b) \big)
          (\sigma_{1}, \ldots, \sigma_{k}).
      \end{align*}
    \item Let $V$ and $W$ be constructed as in
      \Cref{cor:tangblockhermite} Parts~(a) and~(b) for the interpolation
      points $\sigma_{1} = \varsigma_{1}$, $\ldots$,
      $\sigma_{k} = \varsigma_{k}$ and derivative orders $\ell_{1} = \nu_{1}$,
      $\ldots$, $\ell_{k} = \nu_{k}$.
      Then, in addition to the interpolation conditions in
      \Cref{cor:tangblockhermite}, it holds
      \begin{align*}
        & \nabla \left( c^{\herm} \partial_{s_{1}^{\ell_{1}} \cdots
          s_{k}^{\ell_{k}}} G_{k} (I_{m^{k-1}} \otimes b) \right)
          (\sigma_{1}, \ldots, \sigma_{k}) \\
        & \qquad = \nabla \left( c^{\herm} \partial_{s_{1}^{\ell_{1}} \cdots
          s_{k}^{\ell_{k}}} \hG_{k} (I_{m^{k-1}} \otimes b) \right)
          (\sigma_{1}, \ldots, \sigma_{k}).
      \end{align*}
  \end{enumerate}
\end{corollary}

\begin{remark}[Projection space dimensions]
  It will be useful to understand the growth of the size of the model reduction
  bases and thus the order of the resulting interpolatory reduced-order model
  for the different interpolation approaches.
  Let $n_{\rm{s}}$ be the number of sets of interpolation points and tangential 
  directions at which we want to enforce interpolation.
  Also, assume w.l.o.g.\ the recursively generated columns in $V$ and $W$ are
  all linearly independent (since otherwise, the dimensions of the corresponding
  projection spaces can be reduced while still enforcing interpolation).
  Then, for the matrix interpolation approach from~\cite[Thm.~8]{BenGW21a},
  we obtain
  \begin{align} \label{eqn:dimMtx}
    \begin{aligned}
      \dim(\mspan(V_{\rm{mtx}})) & \geq n_{\rm{s}}
        \left( \sum_{j = 1}^{k} m^{k} \right) &
        \text{and} && \dim(\mspan(W_{\rm{mtx}})) & \geq n_{\rm{s}}
        \left( \sum_{j = 1}^{k} p m^{k-1} \right)
    \end{aligned}
  \end{align}
  for the right and left projection spaces, respectively.
  The blockwise tangential approach from \Cref{cor:tangblock} reduces those 
  dimensions to
  \begin{align} \label{eqn:dimBwt}
    \dim(\mspan(V_{\rm{bwt}})) = \dim(\mspan(W_{\rm{bwt}})) &
      \geq n_{\rm{s}} \left( \sum_{j = 1}^{k} m^{k-1} \right).
  \end{align}
  Comparing~\cref{eqn:dimMtx} and~\cref{eqn:dimBwt} shows that 
  the blockwise tangential interpolation approach, similar to matrix
  interpolation, has exponentially growing dimensions of the projection spaces.
  In contrast, the new modified tangential interpolation approach as in 
  \Cref{thm:tangmod} yields
  \begin{align*}
    \dim(\mspan(V_{\rm{st}})) = \dim(\mspan(W_{\rm{st}})) &
      \geq n_{\rm{s}} k,
  \end{align*}
  which now grows only linearly.
  This gives more freedom in the choice of the order of interpolating
  reduced-order models, as well as more possibilities to adapt the choice of
  interpolation points to the problem.
\end{remark}

%%%%%%%%%%%%%%%%%%%%%%%%%%%%%%%%%%%%%%%%%%%%%%%%%%%%%%%%%%%%%%%%%%%%%%%%%%%%%%%%
% EXAMPLES.                                                                    %
%%%%%%%%%%%%%%%%%%%%%%%%%%%%%%%%%%%%%%%%%%%%%%%%%%%%%%%%%%%%%%%%%%%%%%%%%%%%%%%%

\section{Numerical examples}
\label{sec:examples}

In this section, we will compare different structure-preserving interpolation
frameworks.
We compute reduced-order models by:
\begin{description}
  \item[MtxInt] the structure-preserving matrix interpolation
    from~\cite{BenGW21a},
  \item[BwtInt] the structure-preserving blockwise tangential interpolation as
    in \Cref{subsec:tangblock},
  \item[SftInt] the modified structure-preserving tangential interpolation
    framework motivated in the frequency domain (\Cref{subsec:tangfreq}), and
  \item[SttInt] the generalized structure-preserving tangential interpolation
    framework motivated in the time domain (\Cref{subsec:tangtime}).
\end{description}
In the experiments, we use  MATLAB notation to define the
interpolation points:
We write \texttt{logspace(a, b, k)} to denote $k$ logarithmically equidistant
points in the interval~$[10^{a}, 10^{b}]$.

For the qualitative analysis of the computed reduced-order models, we will 
consider approximation errors in time and frequency domains.
In time domain, we consider the point-wise relative output error for a given
input signal, i.e.,
\begin{align*}
  \frac{\lVert y(t) - \hy(t) \rVert_{2}}{\lVert y(t) \rVert_{2}},
\end{align*}
where $y$ and $\hy$ denote the original and reduced-order system outputs,
respectively, in the time range $t \in [0, t_{\mathrm{f}}]$.
Additionally, we compute the maximum error over time by
\begin{align*}
  \err_{\rm{sim}} & := \max\limits_{t \in [0, t_{\rm{f}}]}\frac{\lVert y(t) - 
    \hy(t) \rVert_{2}}{\lVert y(t) \rVert_{2}}.
\end{align*}
In frequency domain, the point-wise relative error of the first and second
transfer functions on the imaginary axis in the spectral norm is considered, 
i.e.,
\begin{align*}
  \begin{aligned}
    \frac{\lVert G_{1}(\i\,\omega_{1}) - \hG_{1}(\i\,\omega_{1}) \rVert_{2}}
      {\lVert G_{1}(\i\,\omega_{1}) \rVert_{2}} && \text{and} &&
    \frac{\lVert G_{2}(\i\,\omega_{1}, \i\,\omega_{2}) -
      \hG_{2}(\i\,\omega_{1}, \i\,\omega_{2}) \rVert_{2}}
      {\lVert G_{2}(\i\,\omega_{1}, \i\,\omega_{2}) \rVert_{2}},
  \end{aligned}
\end{align*}
in the frequency range $\omega_{1}, \omega_{2} \in [\omega_{\min}, 
\omega_{\max}]$
together with the corresponding  maximum errors over the 
frequency of interest defined as
\begin{align*}
  \err_{\rm{G_{1}}} & := \max\limits_{\omega_{1} \in [\omega_{\min}, 
    \omega_{\max}]} \frac{\lVert G_{1}(\i\,\omega_{1}) - \hG_{1}(\i\,\omega_{1})
    \rVert_{2}}{\lVert G_{1}(\i\,\omega_{1}) \rVert_{2}}, \\
  \err_{\rm{G_{2}}} & := \max\limits_{\omega_{1}, \omega_{2} \in [\omega_{\min},
    \omega_{\max}]} \frac{\lVert G_{2}(\i\,\omega_{1}, \i\,\omega_{2}) -
    \hG_{2}(\i\,\omega_{1}, \i\,\omega_{2}) \rVert_{2}}{\lVert G_{2}
    (\i\,\omega_{1}, \i\,\omega_{2}) \rVert_{2}}.
\end{align*}
Note that the time and frequency domain errors reported are actually
approximated by evaluating the above expressions on a fine grid covering
$[0,t_{\mathrm{f}}]$ or $[\omega_{\min},\omega_{\max} ]$, respectively. 

The experiments reported here have been executed on machines with 2 Intel(R)
Xeon(R) Silver 4110 CPU processors running at 2.10\,GHz and equipped with
either 192\,GB or 384\,GB total main memory.
The computers run on CentOS Linux release 7.5.1804 (Core) with
MATLAB 9.9.0.1467703 (R2020b).
The source code, data and results of the numerical experiments are
open source/open access and available at~\cite{supWer22c}.

%%%%%%%%%%%%%%%%%%%%%%%%%%%%%%%%%%%%%%%%%%%%%%%%%%%%%%%%%%%%%%%%%%%%%%%%%%%%%%%%

\subsection{Cooling of steel profiles}

\begin{figure}[t]
  \centering
  \begin{subfigure}[b]{.49\textwidth}
    \centering
  \tikzexternalenable%
  \tikzsetnextfilename{rail_time_sim}%
  \begin{tikzpicture}
  \pgfplotstableread{graphics/data/rail_time_sim.dat}\tableSIM
  
  \begin{axis}[%
    width  = .7\textwidth,
    height = .4\textwidth,
    scale only axis,
    xmin = 0,
    xmax = 4500,
    ymin = -0.017,
    ymax = 0.012,
    xminorticks = false,
    yminorticks = false,
    xlabel = {Time $t$},
    ylabel = {Amplitude},
    ylabel style   = {yshift = -.3em},
    scaled x ticks = false,
    x tick label style = {/pgf/number format/1000 sep={\,}},
    cycle list name    = plotlist]
    
    \foreach \p in {1, 2, ..., 6}{
      \foreach[evaluate=\y as \res using int(\y + \p)] \y in {0, 6, ..., 24} {
        \addplot table[x index = 0, y index = \res,
          each nth point = 10,
          filter discard warning=false,
          unbounded coords=discard] {\tableSIM};
      }
    }
  \end{axis}
\end{tikzpicture}%
  \tikzexternaldisable%

    \subcaption{Time response.}
  \end{subfigure}
  \hfill
  \begin{subfigure}[b]{.49\textwidth}
    \centering
  \tikzexternalenable%
  \tikzsetnextfilename{rail_time_err}%
  \begin{tikzpicture}
  \pgfplotstableread{graphics/data/rail_time_err.dat}\tableRELERR
  
  \begin{semilogyaxis}[%
    width  = .7\textwidth,
    height = .4\textwidth,
    scale only axis,
    xmin = 0,
    xmax = 4500,
    ymin = 1e-6,
    ymax = 5e-1,
    xminorticks = false,
    yminorticks = false,
    xlabel = {Time $t$},
    ylabel = {Magnitude},
    ylabel style = {yshift = -.3em},
    scaled x ticks = false,
    x tick label style = {/pgf/number format/1000 sep={\,}},
    cycle list name    = plotlist]
        
    \pgfplotsset{cycle list shift = 1}
    \foreach \y in {1, 2, 3, 4} {
      \addplot table[x index = 0, y index = \y,
        each nth point = 10,
        filter discard warning=false,
        unbounded coords=discard] {\tableRELERR};
    }
  \end{semilogyaxis}
\end{tikzpicture}%
  \tikzexternaldisable%

    \subcaption{Relative errors.}
  \end{subfigure}
  \vspace{.5\baselineskip}

  \tikzexternalenable%
  \tikzsetnextfilename{rail_time_legend}%
  \begin{tikzpicture}  
  \begin{axis}[%
    hide axis,
    scale only axis,
    width = 1mm,
    legend columns = 5, 
    legend style = {
      at     = {(0,0)},
      anchor = center,
      /tikz/every even column/.append style = {column sep = 0.5cm}},
    cycle list name = plotlist]
    
    \pgfplotsinvokeforeach{1,...,5}{\addplot coordinates {(0,0)};}
        
    \addlegendentry{Original system};
    \addlegendentry{\mtxint};
    \addlegendentry{\bwtint};
    \addlegendentry{\sftint};
    \addlegendentry{\sttint};
  \end{axis}
\end{tikzpicture}%
  \tikzexternaldisable%

  \caption{Time domain simulation results for the steel profile.}
  \label{fig:rail_time}
\end{figure}
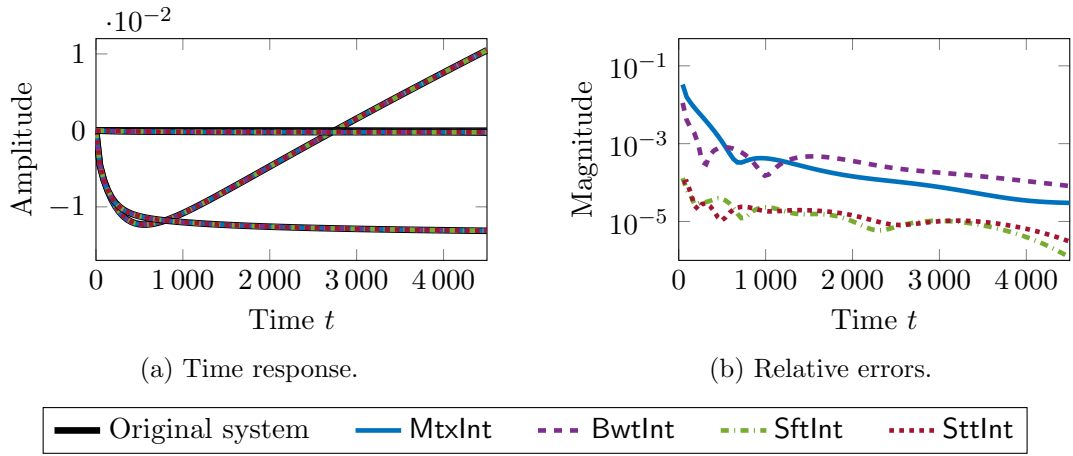

We first consider a classical, unstructured bilinear system 
as in~\cref{eqn:bsys}.
For the optimal cooling of steel profiles, the heat transfer process is 
described by the two dimensional heat equation
\begin{align*}
  c \rho \partial_{t} v(t, \zeta) - \lambda \Delta v(t, \zeta) & = 0,\\
  v(0, \zeta) & = v_{0}(\zeta),
\end{align*}
with $(t, \zeta) \in (0, t_{\mathrm{f}}) \times \Omega$, the initial value $v_{0}(\zeta) 
\in \Omega$, and the Robin boundary conditions
\begin{align*}
  \lambda \partial_{\nu} v(t, \zeta) & = 
    \left\{ \begin{aligned}
      q_{i} u_{i}(t) \big(1 - v(t, \zeta)\big), &&& \text{on}~\Gamma_{i},
        i = 1, \ldots, 6,\\
      q_{7} \big(u_{7}(t) - v(t, \zeta)\big), &&& \text{on}~\Gamma_{7},
    \end{aligned} \right.
\end{align*}
such that $\bigcup_{i = 1}^{7} \Gamma_{i} = \partial \Omega$ and $\Gamma_{i} 
\cap \Gamma_{j} = \emptyset$ for $i \neq j$, where $\partial_{\nu}$ denotes the 
derivative in direction of the outer normal $\nu$ and $u_{i}(t)$ are the 
exterior cooling fluid temperatures used as controls.
The spatial discretization of the rail shaped domain and parameters are chosen 
as described in~\cite{morwiki_steel, Saa03}.
As a result, we consider a system of structure~\cref{eqn:bsys} with
$n = 5\,054\,209$ states, $m = 7$ inputs, non-zero bilinear terms 
corresponding to the first $6$ inputs, and $p = 6$ outputs.
The data for this example is available in~\cite{SaaKB21}.

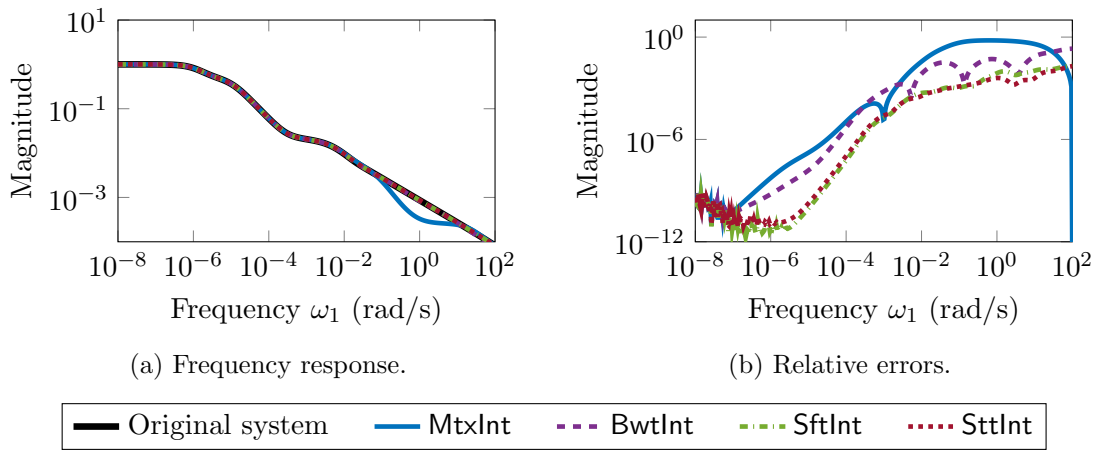
\begin{figure}[t]
  \centering
  \begin{subfigure}[b]{.49\textwidth}
    \centering
  \tikzexternalenable%
  \tikzsetnextfilename{rail_freq_g1_tf}%
  \begin{tikzpicture}
  \pgfplotstableread{graphics/data/rail_freq_g1_tf.dat}\tableTF
  
  \begin{loglogaxis}[%
    width  = .675\textwidth,
    height = .4\textwidth,
    scale only axis,
    xmin = 1e-8,
    xmax = 1e+2,
    xtick = {1e-8, 1e-6, 1e-4, 1e-2, 1e+0, 1e+2},
    ymin = 1e-4,
    ymax = 1e+1,
    xminorticks = false,
    yminorticks = false,
    xlabel = {Frequency $\omega_{1}$ (rad/s)},
    ylabel = {Magnitude},
    ylabel style   = {yshift = -.3em},
    scaled x ticks = false,
    x tick label style = {/pgf/number format/fixed},
    cycle list name    = plotlist]
    
    \foreach \y in {1, 2, ..., 5} {
      \addplot table[x index = 0, y index = \y] {\tableTF};
    }
  \end{loglogaxis}
\end{tikzpicture}%
  \tikzexternaldisable%

    \subcaption{Frequency response.}
  \end{subfigure}
  \hfill
  \begin{subfigure}[b]{.49\textwidth}
    \centering
  \tikzexternalenable%
  \tikzsetnextfilename{rail_freq_g1_err}%
  \begin{tikzpicture}
  \pgfplotstableread{graphics/data/rail_freq_g1_err.dat}\tableRELERR
  
  \begin{loglogaxis}[%
    width  = .675\textwidth,
    height = .4\textwidth,
    scale only axis,
    xmin = 1e-8,
    xmax = 1e+2,
    xtick = {1e-8, 1e-6, 1e-4, 1e-2, 1e+0, 1e+2},
    ymin = 1e-12,
    ymax = 1e+1,
    xminorticks = false,
    yminorticks = false,
    xlabel = {Frequency $\omega_{1}$ (rad/s)},
    ylabel = {Magnitude},
    ylabel style   = {yshift = -.3em},
    scaled x ticks = false,
    x tick label style = {/pgf/number format/fixed},
    cycle list name    = plotlist]
    
    \pgfplotsset{cycle list shift = 1}
    \foreach \y in {1, 2, 3, 4} {
      \addplot table[x index = 0, y index = \y] {\tableRELERR};
    }
  \end{loglogaxis}
\end{tikzpicture}%
  \tikzexternaldisable%

    \subcaption{Relative errors.}
  \end{subfigure}
  \vspace{.5\baselineskip}

  \tikzexternalenable%
  \tikzsetnextfilename{rail_freq_g1_legend}%
  \begin{tikzpicture}  
  \begin{axis}[%
    hide axis,
    scale only axis,
    width = 1mm,
    legend columns = 5, 
    legend style = {
      at     = {(0,0)},
      anchor = center,
      /tikz/every even column/.append style = {column sep = 0.5cm}},
    cycle list name = plotlist]
    
    \pgfplotsinvokeforeach{1,...,5}{\addplot coordinates {(0,0)};}
        
    \addlegendentry{Original system};
    \addlegendentry{\mtxint};
    \addlegendentry{\bwtint};
    \addlegendentry{\sftint};
    \addlegendentry{\sttint};
  \end{axis}
\end{tikzpicture}%
  \tikzexternaldisable%

  \caption{Frequency domain results of the first transfer functions for the 
    steel profile.}
  \label{fig:rail_freq_g1}
\end{figure}

\begin{figure}[t]
  \centering
  \begin{subfigure}[b]{.49\textwidth}
    \centering
  \tikzexternalenable%
  \tikzsetnextfilename{rail_freq_g2_err_mtx}%
  \begin{tikzpicture}
  \begin{loglogaxis}[
    view   = {0}{90},
    width  = .675\textwidth,
    height = .4\textwidth,
    scale only axis,
    axis on top,
    xmin   = 1e-8,
    xmax   = 1e+2,
    ymin   = 1e-8,
    ymax   = 1e+2,
    xtick  = {1e-8, 1e-6, 1e-4, 1e-2, 1e+0, 1e+2},
    ytick  = {1e-8, 1e-6, 1e-4, 1e-2, 1e+0, 1e+2},
    xminorticks = false,
    yminorticks = false,
    xlabel = {Frequency $\omega_{1}$ (rad/s)},
    ylabel = {Frequency $\omega_{2}$ (rad/s)},
    ylabel style = {yshift = -.3em},
    scaled x ticks = false,
    x tick label style = {/pgf/number format/fixed}]
        
      \addplot graphics[xmin = 1e-8, xmax = 1e+2, ymin = 1e-8, ymax = 1e+2]
        {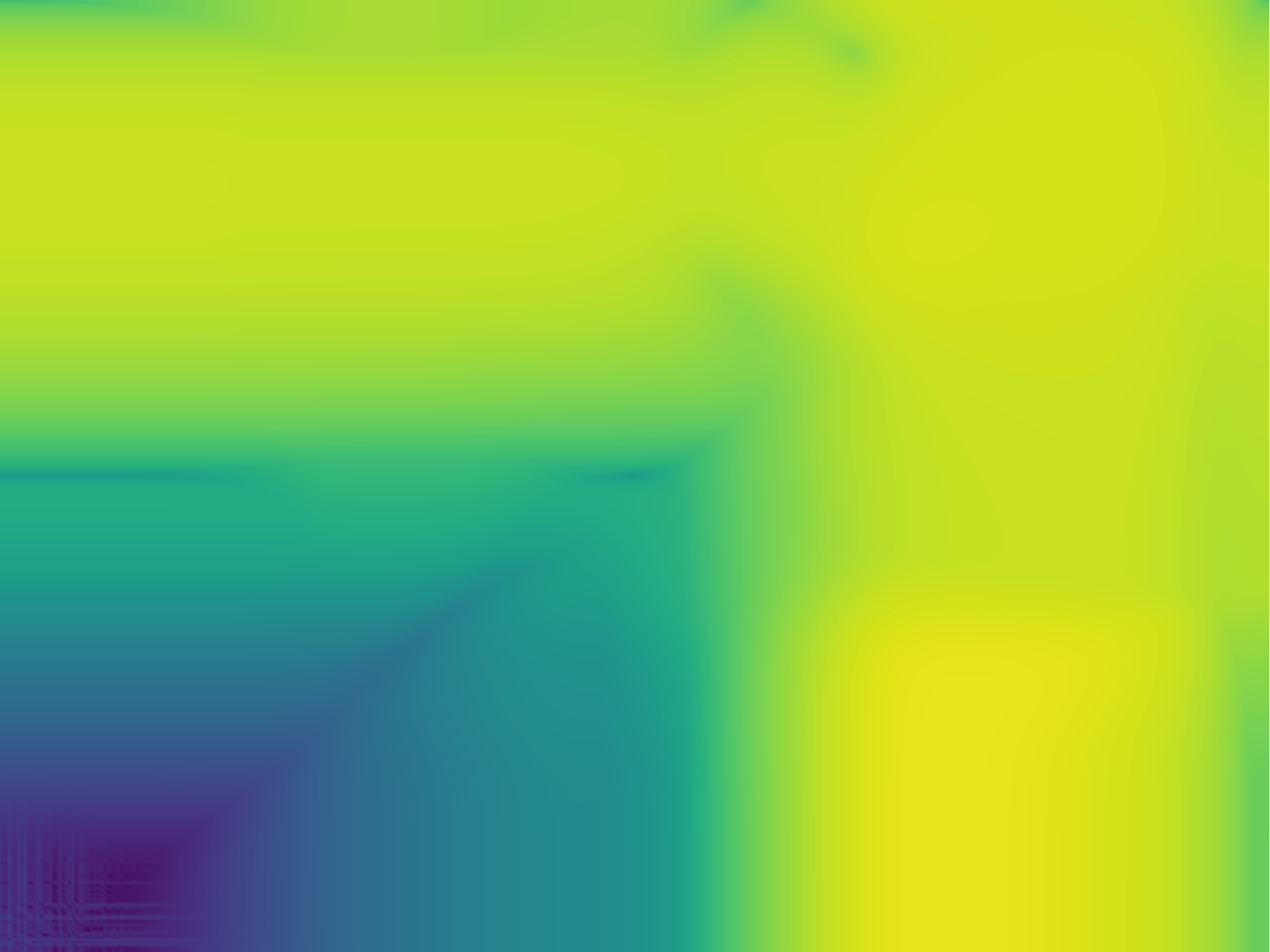};
            
  \end{loglogaxis}
\end{tikzpicture}%
  \tikzexternaldisable%

    \subcaption{\mtxint{}.}
  \end{subfigure}
  \hfill
  \begin{subfigure}[b]{.49\textwidth}
    \centering
  \tikzexternalenable%
  \tikzsetnextfilename{rail_freq_g2_err_bwt}%
  \begin{tikzpicture}
  \begin{loglogaxis}[
    view   = {0}{90},
    width  = .675\textwidth,
    height = .4\textwidth,
    scale only axis,
    axis on top,
    xmin   = 1e-8,
    xmax   = 1e+2,
    ymin   = 1e-8,
    ymax   = 1e+2,
    xtick  = {1e-8, 1e-6, 1e-4, 1e-2, 1e+0, 1e+2},
    ytick  = {1e-8, 1e-6, 1e-4, 1e-2, 1e+0, 1e+2},
    xminorticks = false,
    yminorticks = false,
    xlabel = {Frequency $\omega_{1}$ (rad/s)},
    ylabel = {Frequency $\omega_{2}$ (rad/s)},
    ylabel style = {yshift = -.3em},
    scaled x ticks = false,
    x tick label style = {/pgf/number format/fixed}]
        
      \addplot graphics[xmin = 1e-8, xmax = 1e+2, ymin = 1e-8, ymax = 1e+2]
        {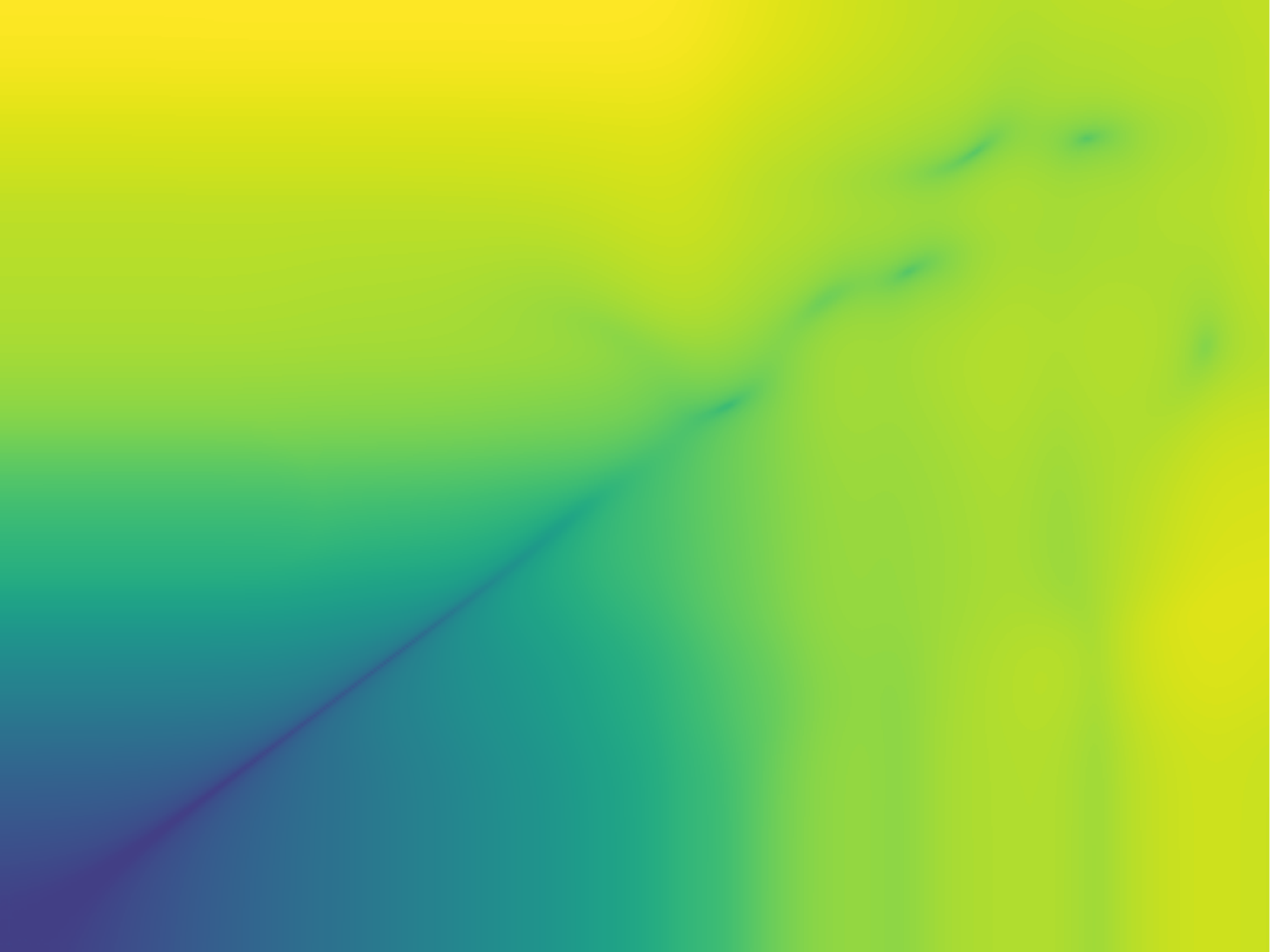};
            
  \end{loglogaxis}
\end{tikzpicture}%
  \tikzexternaldisable%

    \subcaption{\bwtint{}.}
  \end{subfigure}
    
  \begin{subfigure}[b]{.49\textwidth}
    \centering
  \tikzexternalenable%
  \tikzsetnextfilename{rail_freq_g2_err_sft}%
  \begin{tikzpicture}
  \begin{loglogaxis}[
    view   = {0}{90},
    width  = .675\textwidth,
    height = .4\textwidth,
    scale only axis,
    axis on top,
    xmin   = 1e-8,
    xmax   = 1e+2,
    ymin   = 1e-8,
    ymax   = 1e+2,
    xtick  = {1e-8, 1e-6, 1e-4, 1e-2, 1e+0, 1e+2},
    ytick  = {1e-8, 1e-6, 1e-4, 1e-2, 1e+0, 1e+2},
    xminorticks = false,
    yminorticks = false,
    xlabel = {Frequency $\omega_{1}$ (rad/s)},
    ylabel = {Frequency $\omega_{2}$ (rad/s)},
    ylabel style = {yshift = -.3em},
    scaled x ticks = false,
    x tick label style = {/pgf/number format/fixed}]
        
      \addplot graphics[xmin = 1e-8, xmax = 1e+2, ymin = 1e-8, ymax = 1e+2]
        {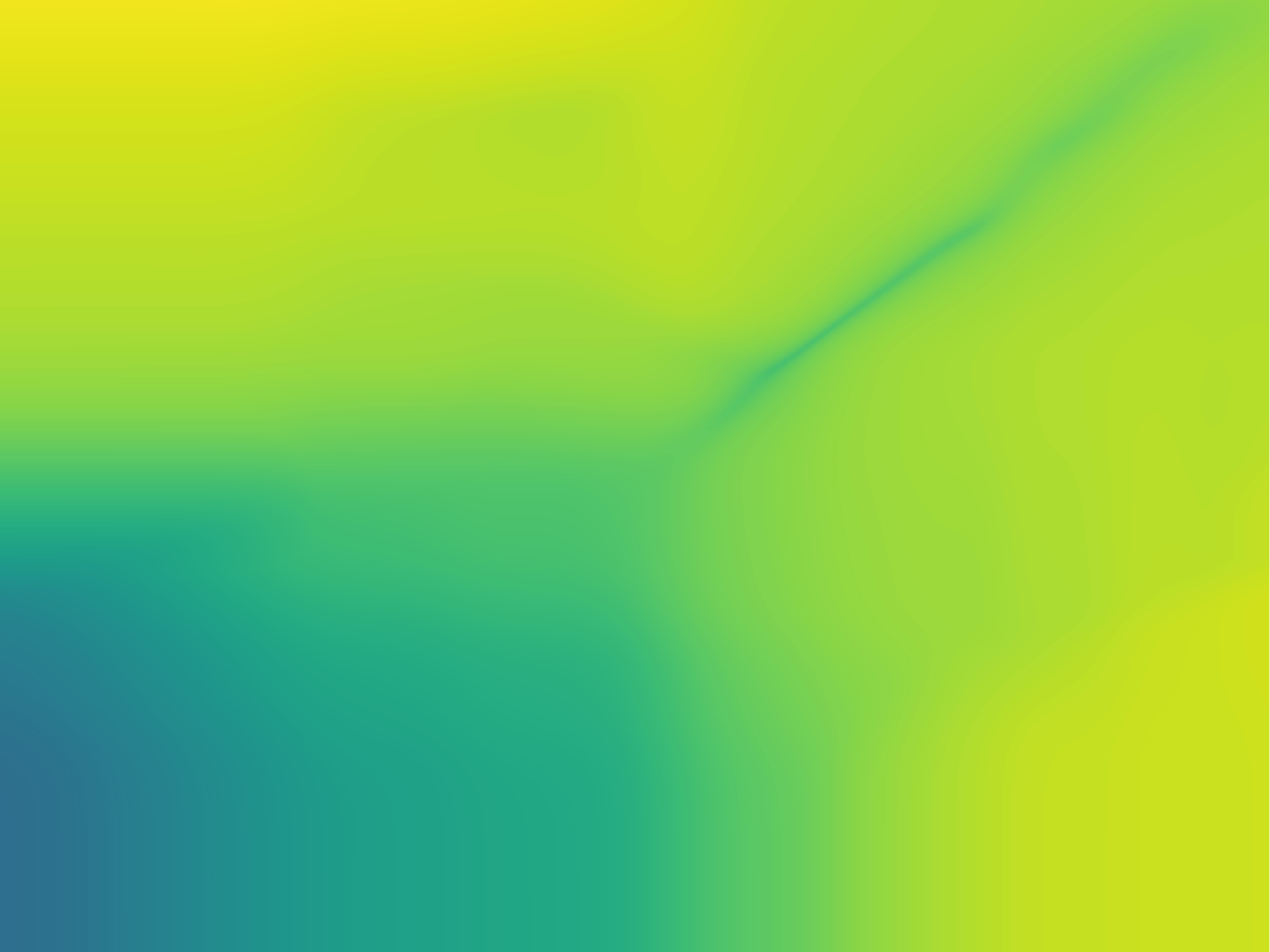};
            
  \end{loglogaxis}
\end{tikzpicture}%
  \tikzexternaldisable%

    \subcaption{\sftint{}.}
  \end{subfigure}
  \hfill
  \begin{subfigure}[b]{.49\textwidth}
    \centering
  \tikzexternalenable%
  \tikzsetnextfilename{rail_freq_g2_err_stt}%
  \begin{tikzpicture}
  \begin{loglogaxis}[
    view   = {0}{90},
    width  = .675\textwidth,
    height = .4\textwidth,
    scale only axis,
    axis on top,
    xmin   = 1e-8,
    xmax   = 1e+2,
    ymin   = 1e-8,
    ymax   = 1e+2,
    xtick  = {1e-8, 1e-6, 1e-4, 1e-2, 1e+0, 1e+2},
    ytick  = {1e-8, 1e-6, 1e-4, 1e-2, 1e+0, 1e+2},
    xminorticks = false,
    yminorticks = false,
    xlabel = {Frequency $\omega_{1}$ (rad/s)},
    ylabel = {Frequency $\omega_{2}$ (rad/s)},
    ylabel style = {yshift = -.3em},
    scaled x ticks = false,
    x tick label style = {/pgf/number format/fixed}]
        
      \addplot graphics[xmin = 1e-8, xmax = 1e+2, ymin = 1e-8, ymax = 1e+2]
        {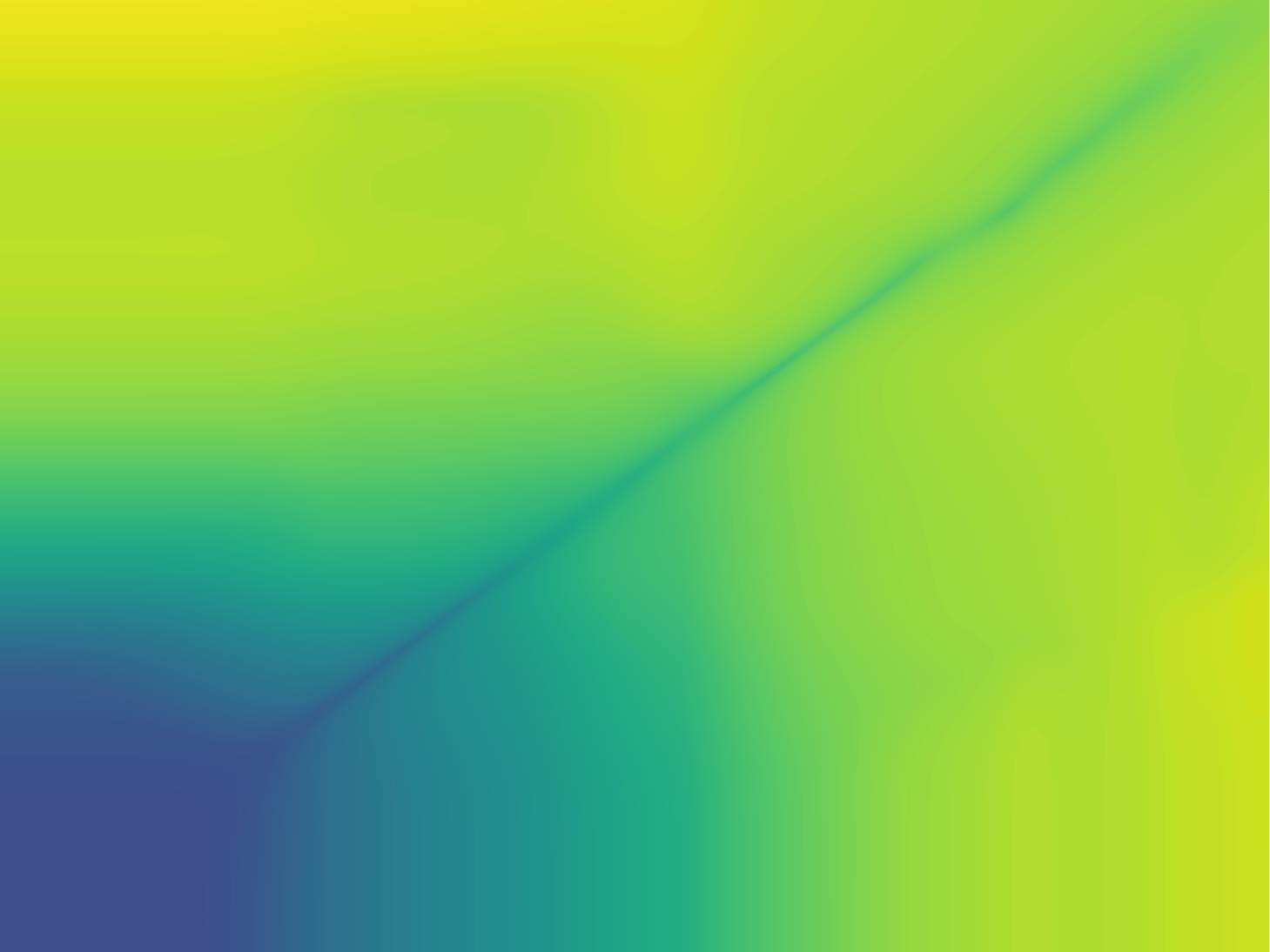};
            
  \end{loglogaxis}
\end{tikzpicture}%
  \tikzexternaldisable%

    \subcaption{\sttint{}.}
  \end{subfigure}
  \vspace{.5\baselineskip}

  \tikzexternalenable%
  \tikzsetnextfilename{rail_freq_g2_legend}%
  \begin{tikzpicture}
  \node[draw = none, minimum width = 0cm, inner sep = 0cm](start){};
  \node(leg) at (start.north east) [anchor = north west]{\tikz
  \begin{axis}[%
    hide axis,
    scale only axis,
    width  = 10cm,
    height = .1cm,
    point meta min = -10.8792,
    point meta max = 0.7036,
    colorbar,
    colorbar horizontal,
    colorbar style = {
      xticklabel = $10^{\pgfmathparse{\tick}
        \pgfmathprintnumber\pgfmathresult}$,
      at = {(.5, 0)},
      anchor = north},
    scaled x ticks = false,
    x tick label style = {/pgf/number format/fixed}]
  \end{axis};};
  \node[draw = none, minimum width = 0cm, inner sep = 0cm](end)
    at (leg.north east) [anchor = north west]{};
\end{tikzpicture}%
  \tikzexternaldisable%

  \caption{Frequency domain results of the second transfer functions for the 
    steel profile.}
  \label{fig:rail_freq_g2}
\end{figure}

The reduced-order models are constructed as follows:
\begin{description}
  \item[MtxInt] with the interpolation points $\pm \texttt{logspace(-8, 2, 3)}
    \i$ for the first and second subsystem transfer functions.
    Due to the rank deficiency in the generated columns, a rank truncation is
    performed to compress the model reduction basis, which yields a 
    reduced-order model size of $r_{\sf{mtx}} = 146$.
  \item[BwtInt] with the interpolation points $\pm \texttt{logspace(-8, 2, 8)} 
    \i$ for the first and second subsystem transfer functions
    resulting in the reduced order $r_{\sf{bwt}} = 112$.
  \item[SftInt] with the interpolation points $\pm \texttt{logspace(-8, 2, 28)} 
    \i$ and the scaling vectors $d^{(i)} = \mathds{1}_{m}$ for the first and 
    second subsystem transfer functions resulting in the reduced order
    $r_{\sf{sft}} = 112$.
  \item[SttInt] with the interpolation points $\pm \texttt{logspace(-8, 2, 28)} 
    \i$ and the scaling vectors $d^{(i)} = b^{(i)}$ for the first and second 
    subsystem transfer functions such that the reduced-order model size is
    $r_{\sf{stt}} = 112$.
\end{description}
For all reduced-order models, we have chosen the same interval for 
the interpolation points.
However, since the reduced-order dimension grows differently for different
approaches, the number of interpolation points over the same interval differs
so that the reduced-order models have the same (or at least comparable) order.
For all directions, normalized random vectors from a uniform distribution on 
$[0, 1]$ have been used.
For all reduced-order models only one-sided projections ($W$ is set to
$W = V$) have been applied resulting in reduced-order models having
asymptotically stable linear parts.
Note that the matrix interpolation has a much larger reduced order as
anticipated.

\Cref{fig:rail_time} shows the results for a time simulation using a unit step
signal as input.
All reduced-order models yield accurate approximations.
The relative errors  reveal that overall, \sftint{} and \sttint{} perform best,
while \mtxint{} and \bwtint{} are several orders of magnitude worse in accuracy
over the whole time interval.
The maximum errors attained are given in \Cref{tab:rail}.
There, the two new tangential approaches \sftint{} and \sttint{} are both three
orders of magnitude better than \mtxint{} for the time domain simulation.

The frequency domain analysis (\Cref{fig:rail_freq_g1,fig:rail_freq_g2})
illustrates similar conclusions.
In the case of the first subsystem transfer function, \mtxint{} performs overall
worst followed by \bwtint{}.
The new approaches \sftint{} and \sttint{} again show the smallest errors over
the full frequency range.
For the second transfer function level, all approaches behave comparable.
The tangential approaches provide better errors than \mtxint{} if both frequency
arguments are close to each other and \mtxint{} is more accurate for very small
frequencies.
For both transfer function levels, the maximum errors are given in 
\Cref{tab:rail}.

\begin{table}[t]
  \caption{Maximum relative errors for the steel profile.}
  \label{tab:rail}
  \centering
  {\renewcommand{\arraystretch}{1.25}%
  \setlength{\tabcolsep}{.8em}
  \begin{tabular}{crrrr}
    \hline
    &
      \multicolumn{1}{c}{$\boldsymbol{\mtxint}$} &
      \multicolumn{1}{c}{$\boldsymbol{\bwtint}$} &
      \multicolumn{1}{c}{$\boldsymbol{\sftint}$} &
      \multicolumn{1}{c}{$\boldsymbol{\sttint}$} \\
    \hline\noalign{\medskip}
    $\err_{\rm{sim}}$ &
      $4.0019\texttt{e-}01$ &
      $1.6248\texttt{e-}02$ &
      $3.2016\texttt{e-}04$ &
      $3.3950\texttt{e-}04$ \\
    $\err_{\rm{G_{1}}}$ &
      $6.3977\texttt{e-}01$ &
      $2.0925\texttt{e-}01$ &
      $2.3971\texttt{e-}02$ &
      $1.9581\texttt{e-}02$ \\
    $\err_{\rm{G_{2}}}$ &
      $2.1258\texttt{e+}00$ &
      $5.0533\texttt{e+}00$ &
      $3.1073\texttt{e+}00$ &
      $2.6049\texttt{e+}00$ \\
    \noalign{\medskip}\hline\noalign{\smallskip}
  \end{tabular}}
\end{table}

%%%%%%%%%%%%%%%%%%%%%%%%%%%%%%%%%%%%%%%%%%%%%%%%%%%%%%%%%%%%%%%%%%%%%%%%%%%%%%%%

\subsection{Time-delayed heated rod}

Here, we consider the single-input/single-output structured bilinear system
from~\cite{GosPBetal19,BenGW21a} that models a heated rod with distributed
control and homogeneous Dirichlet boundary conditions, which is cooled by a
delayed feedback.
The underlying dynamics are described by the one dimensional heat equation
\begin{align} \label{eqn:time_delay}
  \partial_{t} v(t, \zeta) & = \Delta v(t, \zeta)
    - 2 \sin(\zeta) v(t, \zeta)
    + 2 \sin(\zeta) v(t - 1, \zeta)
    + u(t),
\end{align}
with $(t, \zeta) \in (0, t_{\mathrm{f}}) \times (0, \pi)$ and boundary
conditions $v(t, 0) = v(t, \pi) = 0$ for all $t \in [0, t_{\mathrm{f}}]$.
As extension of~\cref{eqn:time_delay} to the MIMO case, we consider independent 
control signals on equally sized sections of the rod as well as analogous
measurements.
Using centered differences for the spatial discretization, we obtain the
bilinear time-delay system
\begin{align*}
  \dot{x}(t) & = Ax(t) + A_{\rm{d}} x(t - 1)
    + \sum\limits_{k = 1}^{m} N_{k} x(t) u_{k}(t) + B u(t), \\
  y(t) & = C x(t),
\end{align*}
with $A,  A_{\rm{d}}, N_{k} \in \R^{n \times n}$, for $k = 1, \ldots, m$,
$B \in \R^{n \times m}$ and $C \in \R^{p \times n}$.
For our experiments, we have chosen $n = 5\,000$, $m = 5$ and $p = 2$.

\begin{figure}[t]
  \centering
  \begin{subfigure}[b]{.49\textwidth}
    \centering
  \tikzexternalenable%
  \tikzsetnextfilename{time_delay_time_sim}%
  \begin{tikzpicture}
  \pgfplotstableread{graphics/data/time_delay_time_sim.dat}\tableSIM
  
  \begin{axis}[%
    width  = .675\textwidth,
    height = .4\textwidth,
    scale only axis,
    xmin = 0,
    xmax = 10,
    ymin = -0.15,
    ymax = 0.55,
    xminorticks = false,
    yminorticks = false,
    xlabel = {Time $t$},
    ylabel = {Amplitude},
    ylabel style   = {yshift = -.3em},
    scaled x ticks = false,
    x tick label style = {/pgf/number format/fixed},
    cycle list name    = plotlist]
    
    \foreach \p in {1, 2}{
      \foreach[evaluate=\y as \res using int(\y + \p)] \y in {0, 2, ..., 8} {
        \addplot table[x index = 0, y index = \res] {\tableSIM};
      }
    }
  \end{axis}
\end{tikzpicture}%
  \tikzexternaldisable%

    \subcaption{Time response.}
  \end{subfigure}
  \hfill
  \begin{subfigure}[b]{.49\textwidth}
    \centering
  \tikzexternalenable%
  \tikzsetnextfilename{time_delay_time_err}%
  \begin{tikzpicture}
  \pgfplotstableread{graphics/data/time_delay_time_err.dat}\tableRELERR
  
  \begin{semilogyaxis}[%
    width  = .675\textwidth,
    height = .4\textwidth,
    scale only axis,
    xmin = 0,
    xmax = 10,
    ymin = 1e-11,
    ymax = 1e-4,
    xminorticks = false,
    yminorticks = false,
    xlabel = {Time $t$},
    ylabel = {Magnitude},
    ylabel style = {yshift = -.3em},
    scaled x ticks = false,
    x tick label style = {/pgf/number format/fixed},
    cycle list name    = plotlist]
        
    \pgfplotsset{cycle list shift = 1}
    \foreach \y in {1, 2, 3, 4} {
      \addplot table[x index = 0, y index = \y] {\tableRELERR};
    }
  \end{semilogyaxis}
\end{tikzpicture}%
  \tikzexternaldisable%

    \subcaption{Relative errors.}
  \end{subfigure}
  \vspace{.5\baselineskip}

  \tikzexternalenable%
  \tikzsetnextfilename{time_delay_time_legend}%
  \begin{tikzpicture}  
  \begin{axis}[%
    hide axis,
    scale only axis,
    width = 1mm,
    legend columns = 5, 
    legend style = {
      at     = {(0,0)},
      anchor = center,
      /tikz/every even column/.append style = {column sep = 0.5cm}},
    cycle list name = plotlist]
    
    \pgfplotsinvokeforeach{1,...,5}{\addplot coordinates {(0,0)};}
        
    \addlegendentry{Original system};
    \addlegendentry{\mtxint};
    \addlegendentry{\bwtint};
    \addlegendentry{\sftint};
    \addlegendentry{\sttint};
  \end{axis}
\end{tikzpicture}%
  \tikzexternaldisable%

  \caption{Time domain simulation results for the time-delayed heated rod.}
  \label{fig:time_delay_time}
\end{figure}

The reduced-order models are constructed as follows:

\begin{description}
  \item[MtxInt] with the interpolation points $\pm 1 \i$ for the 
    first and second subsystem transfer functions.
    To overcome stability issues, only a one-sided projection was applied.
    The generated columns for the basis are rank deficient, therefore, a rank 
    truncation has been performed to compress the model reduction basis
    resulting in the reduced order $r_{\sf{mtx}} = 36$.
  \item[BwtInt] with the interpolation points $\pm \texttt{logspace(-4, 4, 3)} 
    \i$ for the first and second subsystem transfer functions with two-sided
    projection yielding the reduced order $r_{\sf{bwt}} = 36$.
  \item[SftInt] with the interpolation points $\pm \texttt{logspace(-4, 4, 9)} 
    \i$ and the scaling vectors $d^{(i)} = \mathds{1}_{m}$ for the first and 
    second subsystem transfer functions with two-sided projection to get 
    a reduced-order model of size $r_{\sf{sft}} = 36$.
  \item[SttInt] with the interpolation points $\pm \texttt{logspace(-4, 4, 9)} 
    \i$ and the scaling vectors $d^{(i)} = b^{(i)}$ for the first and second 
    subsystem transfer functions with two-sided projection to get 
    a reduced-order model of size $r_{\sf{stt}} = 36$.
\end{description}
For all directions, normalized random vectors from a uniform distribution on 
$[0, 1]$ have been used.
Note that all reduced-order models have the same time-delay structure as the 
original system~\cref{eqn:time_delay}.
All reduced-order models are chosen to be of the same size.

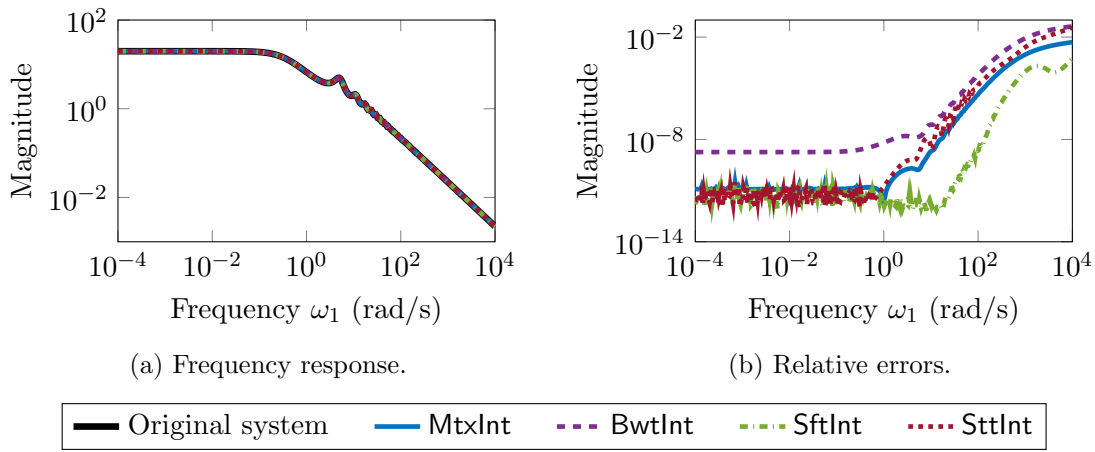
\begin{figure}[t]
  \centering
  \begin{subfigure}[b]{.49\textwidth}
    \centering
  \tikzexternalenable%
  \tikzsetnextfilename{time_delay_freq_g1_tf}%
  \begin{tikzpicture}
  \pgfplotstableread{graphics/data/time_delay_freq_g1_tf.dat}\tableTF
  
  \begin{loglogaxis}[%
    width  = .675\textwidth,
    height = .4\textwidth,
    scale only axis,
    xmin = 1e-4,
    xmax = 1e+4,
    xtick = {1e-4, 1e-2, 1e+0, 1e+2, 1e+4},
    ymin = 1e-3,
    ymax = 1e+2,
    xminorticks = false,
    yminorticks = false,
    xlabel = {Frequency $\omega_{1}$ (rad/s)},
    ylabel = {Magnitude},
    ylabel style   = {yshift = -.3em},
    scaled x ticks = false,
    x tick label style = {/pgf/number format/fixed},
    cycle list name    = plotlist]
    
    \foreach \y in {1, 2, ..., 5} {
      \addplot table[x index = 0, y index = \y] {\tableTF};
    }
  \end{loglogaxis}
\end{tikzpicture}%
  \tikzexternaldisable%

    \subcaption{Frequency response.}
  \end{subfigure}
  \hfill
  \begin{subfigure}[b]{.49\textwidth}
    \centering
  \tikzexternalenable%
  \tikzsetnextfilename{time_delay_freq_g1_err}%
  \begin{tikzpicture}
  \pgfplotstableread{graphics/data/time_delay_freq_g1_err.dat}\tableRELERR
  
  \begin{loglogaxis}[%
    width  = .675\textwidth,
    height = .4\textwidth,
    scale only axis,
    xmin = 1e-4,
    xmax = 1e+4,
    xtick = {1e-4, 1e-2, 1e+0, 1e+2, 1e+4},
    ymin = 1e-14,
    ymax = 1e-1,
    xminorticks = false,
    yminorticks = false,
    xlabel = {Frequency $\omega_{1}$ (rad/s)},
    ylabel = {Magnitude},
    ylabel style   = {yshift = -.3em},
    scaled x ticks = false,
    x tick label style = {/pgf/number format/fixed},
    cycle list name    = plotlist]
    
    \pgfplotsset{cycle list shift = 1}
    \foreach \y in {1, 2, 3, 4} {
      \addplot table[x index = 0, y index = \y] {\tableRELERR};
    }
  \end{loglogaxis}
\end{tikzpicture}%
  \tikzexternaldisable%

    \subcaption{Relative errors.}
  \end{subfigure}
  \vspace{.5\baselineskip}

  \tikzexternalenable%
  \tikzsetnextfilename{time_delay_freq_g1_legend}%
  \begin{tikzpicture}  
  \begin{axis}[%
    hide axis,
    scale only axis,
    width = 1mm,
    legend columns = 5, 
    legend style = {
      at     = {(0,0)},
      anchor = center,
      /tikz/every even column/.append style = {column sep = 0.5cm}},
    cycle list name = plotlist]
    
    \pgfplotsinvokeforeach{1,...,5}{\addplot coordinates {(0,0)};}
        
    \addlegendentry{Original system};
    \addlegendentry{\mtxint};
    \addlegendentry{\bwtint};
    \addlegendentry{\sftint};
    \addlegendentry{\sttint};
  \end{axis}
\end{tikzpicture}%
  \tikzexternaldisable%

  \caption{Frequency domain results of the first transfer functions for the 
    time-delayed heated rod.}
  \label{fig:time_delay_freq_g1}
\end{figure}

\begin{figure}[t]
  \centering
  \begin{subfigure}[b]{.49\textwidth}
    \centering
  \tikzexternalenable%
  \tikzsetnextfilename{time_delay_freq_g2_err_mtx}%
  \begin{tikzpicture}
  \begin{loglogaxis}[
    view   = {0}{90},
    width  = .675\textwidth,
    height = .4\textwidth,
    scale only axis,
    axis on top,
    xmin   = 1e-4,
    xmax   = 1e+4,
    ymin   = 1e-4,
    ymax   = 1e+4,
    xtick  = {1e-4, 1e-2, 1e0, 1e+2, 1e+4},
    ytick  = {1e-4, 1e-2, 1e0, 1e+2, 1e+4},
    xminorticks = false,
    yminorticks = false,
    xlabel = {Frequency $\omega_{1}$ (rad/s)},
    ylabel = {Frequency $\omega_{2}$ (rad/s)},
    ylabel style = {yshift = -.3em},
    scaled x ticks = false,
    x tick label style = {/pgf/number format/fixed}]
        
      \addplot graphics[xmin = 1e-4, xmax = 1e+4, ymin = 1e-4, ymax = 1e+4]
        {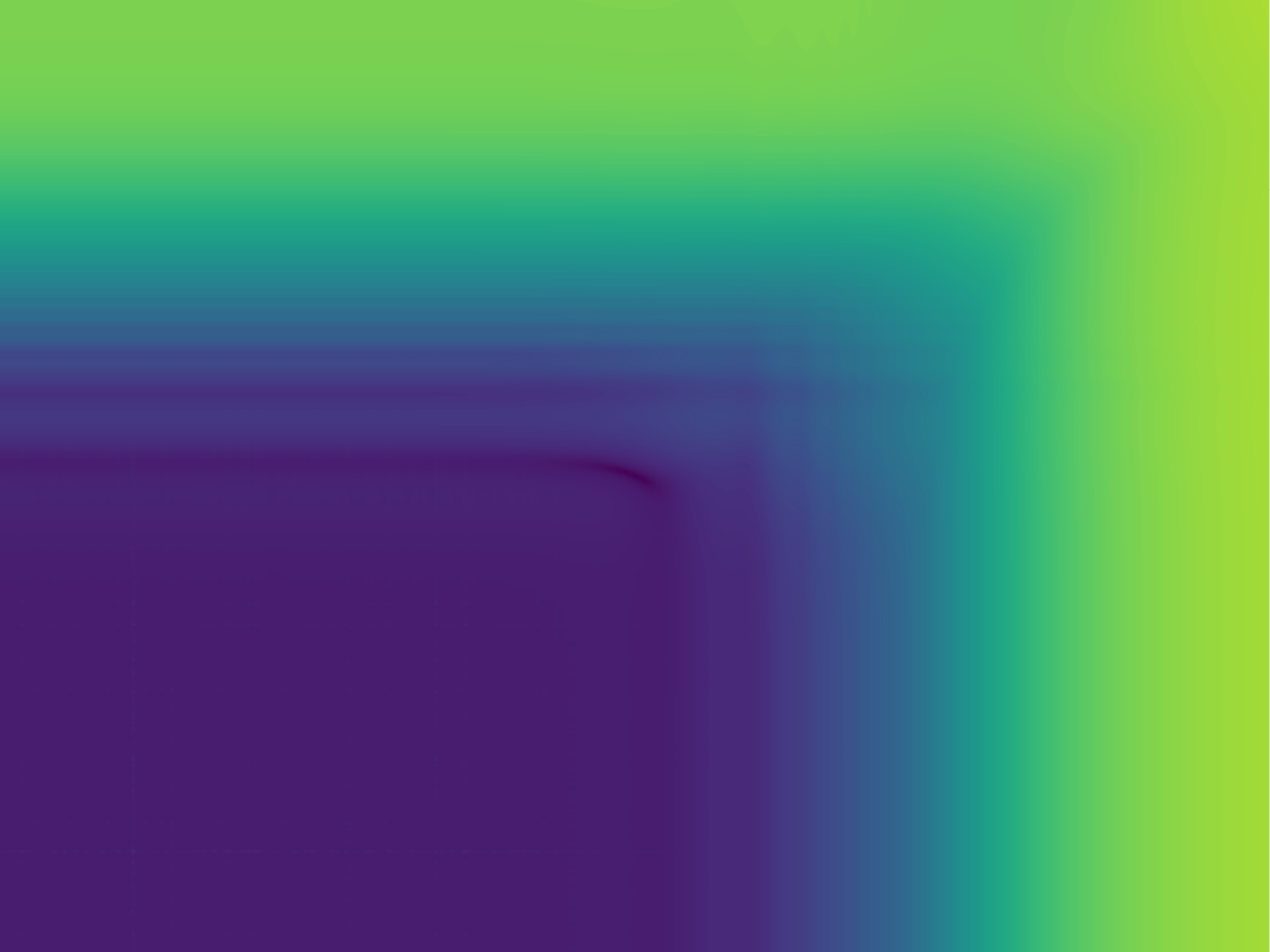};
            
  \end{loglogaxis}
\end{tikzpicture}%
  \tikzexternaldisable%

    \subcaption{\mtxint{}.}
  \end{subfigure}
  \hfill
  \begin{subfigure}[b]{.49\textwidth}
    \centering
  \tikzexternalenable%
  \tikzsetnextfilename{time_delay_freq_g2_err_bwt}%
  \begin{tikzpicture}
  \begin{loglogaxis}[
    view   = {0}{90},
    width  = .675\textwidth,
    height = .4\textwidth,
    scale only axis,
    axis on top,
    xmin   = 1e-4,
    xmax   = 1e+4,
    ymin   = 1e-4,
    ymax   = 1e+4,
    xtick  = {1e-4, 1e-2, 1e0, 1e+2, 1e+4},
    ytick  = {1e-4, 1e-2, 1e0, 1e+2, 1e+4},
    xminorticks = false,
    yminorticks = false,
    xlabel = {Frequency $\omega_{1}$ (rad/s)},
    ylabel = {Frequency $\omega_{2}$ (rad/s)},
    ylabel style = {yshift = -.3em},
    scaled x ticks = false,
    x tick label style = {/pgf/number format/fixed}]
        
      \addplot graphics[xmin = 1e-4, xmax = 1e+4, ymin = 1e-4, ymax = 1e+4]
        {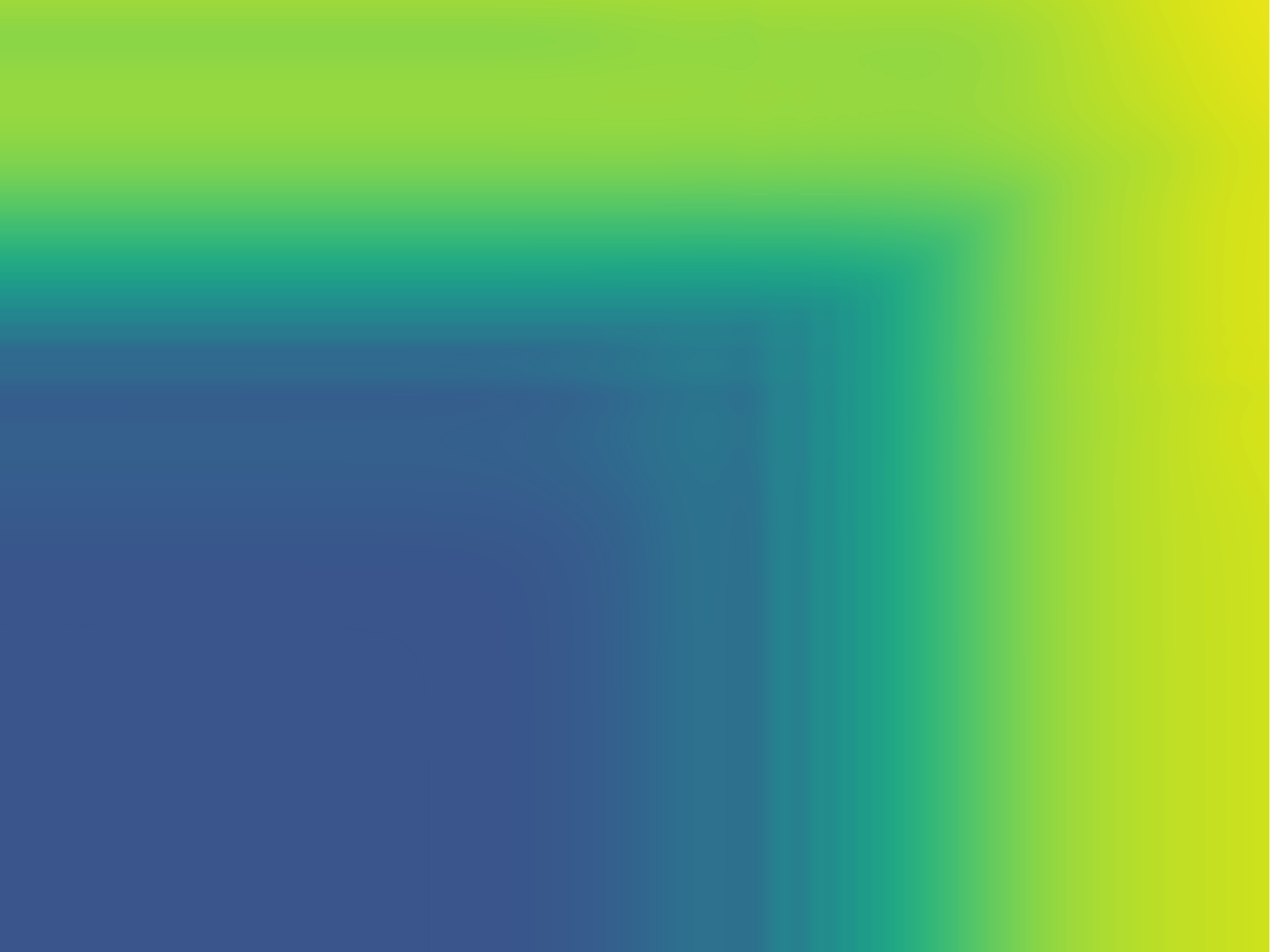};
            
  \end{loglogaxis}
\end{tikzpicture}%
  \tikzexternaldisable%

    \subcaption{\bwtint{}.}
  \end{subfigure}
    
  \begin{subfigure}[b]{.49\textwidth}
    \centering
  \tikzexternalenable%
  \tikzsetnextfilename{time_delay_freq_g2_err_sft}%
  \begin{tikzpicture}
  \begin{loglogaxis}[
    view   = {0}{90},
    width  = .675\textwidth,
    height = .4\textwidth,
    scale only axis,
    axis on top,
    xmin   = 1e-4,
    xmax   = 1e+4,
    ymin   = 1e-4,
    ymax   = 1e+4,
    xtick  = {1e-4, 1e-2, 1e0, 1e+2, 1e+4},
    ytick  = {1e-4, 1e-2, 1e0, 1e+2, 1e+4},
    xminorticks = false,
    yminorticks = false,
    xlabel = {Frequency $\omega_{1}$ (rad/s)},
    ylabel = {Frequency $\omega_{2}$ (rad/s)},
    ylabel style = {yshift = -.3em},
    scaled x ticks = false,
    x tick label style = {/pgf/number format/fixed}]
        
      \addplot graphics[xmin = 1e-4, xmax = 1e+4, ymin = 1e-4, ymax = 1e+4]
        {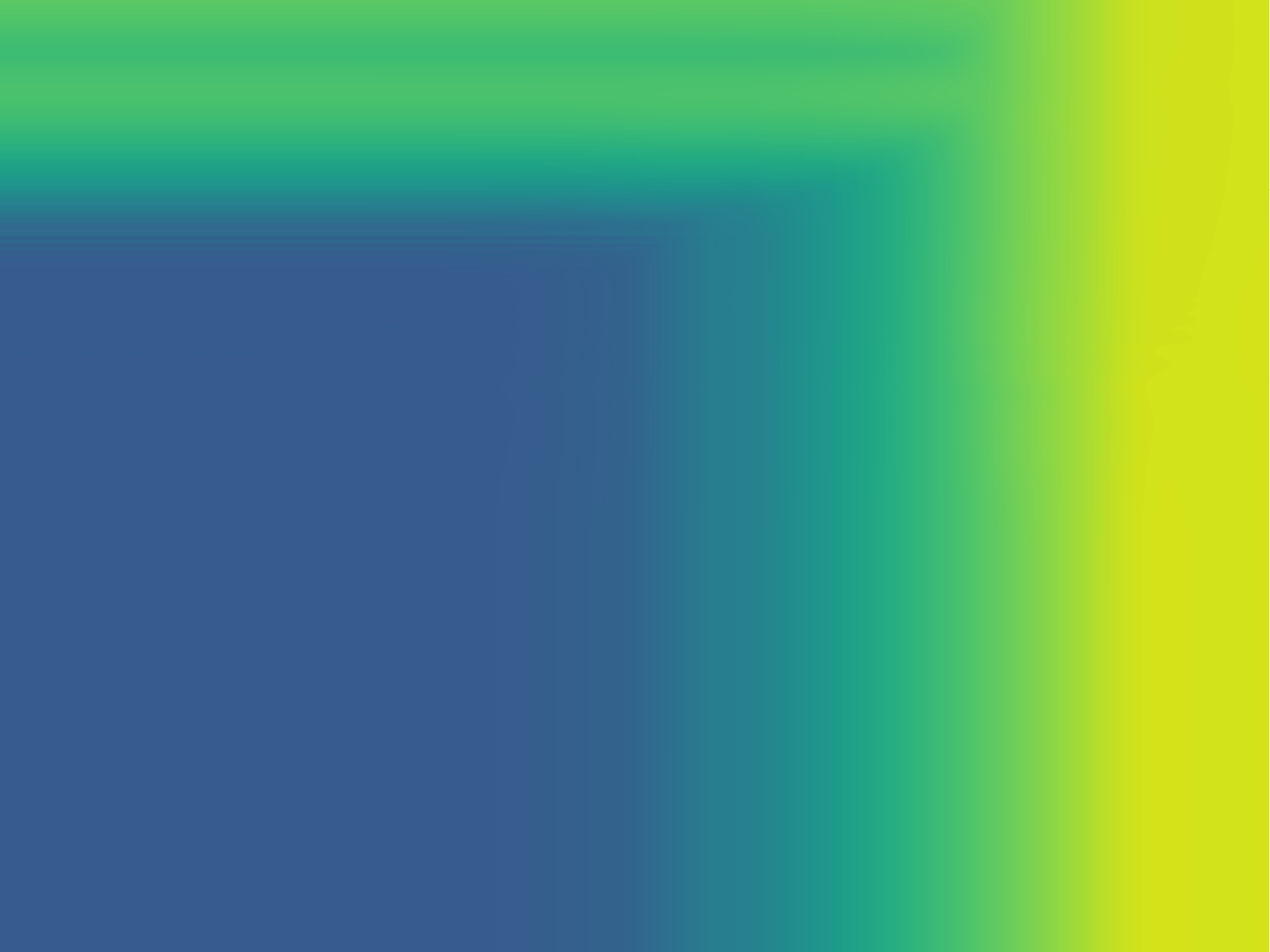};
            
  \end{loglogaxis}
\end{tikzpicture}%
  \tikzexternaldisable%

    \subcaption{\sftint{}.}
  \end{subfigure}
  \hfill
  \begin{subfigure}[b]{.49\textwidth}
    \centering
  \tikzexternalenable%
  \tikzsetnextfilename{time_delay_freq_g2_err_stt}%
  \begin{tikzpicture}
  \begin{loglogaxis}[
    view   = {0}{90},
    width  = .675\textwidth,
    height = .4\textwidth,
    scale only axis,
    axis on top,
    xmin   = 1e-4,
    xmax   = 1e+4,
    ymin   = 1e-4,
    ymax   = 1e+4,
    xtick  = {1e-4, 1e-2, 1e0, 1e+2, 1e+4},
    ytick  = {1e-4, 1e-2, 1e0, 1e+2, 1e+4},
    xminorticks = false,
    yminorticks = false,
    xlabel = {Frequency $\omega_{1}$ (rad/s)},
    ylabel = {Frequency $\omega_{2}$ (rad/s)},
    ylabel style = {yshift = -.3em},
    scaled x ticks = false,
    x tick label style = {/pgf/number format/fixed}]
        
      \addplot graphics[xmin = 1e-4, xmax = 1e+4, ymin = 1e-4, ymax = 1e+4]
        {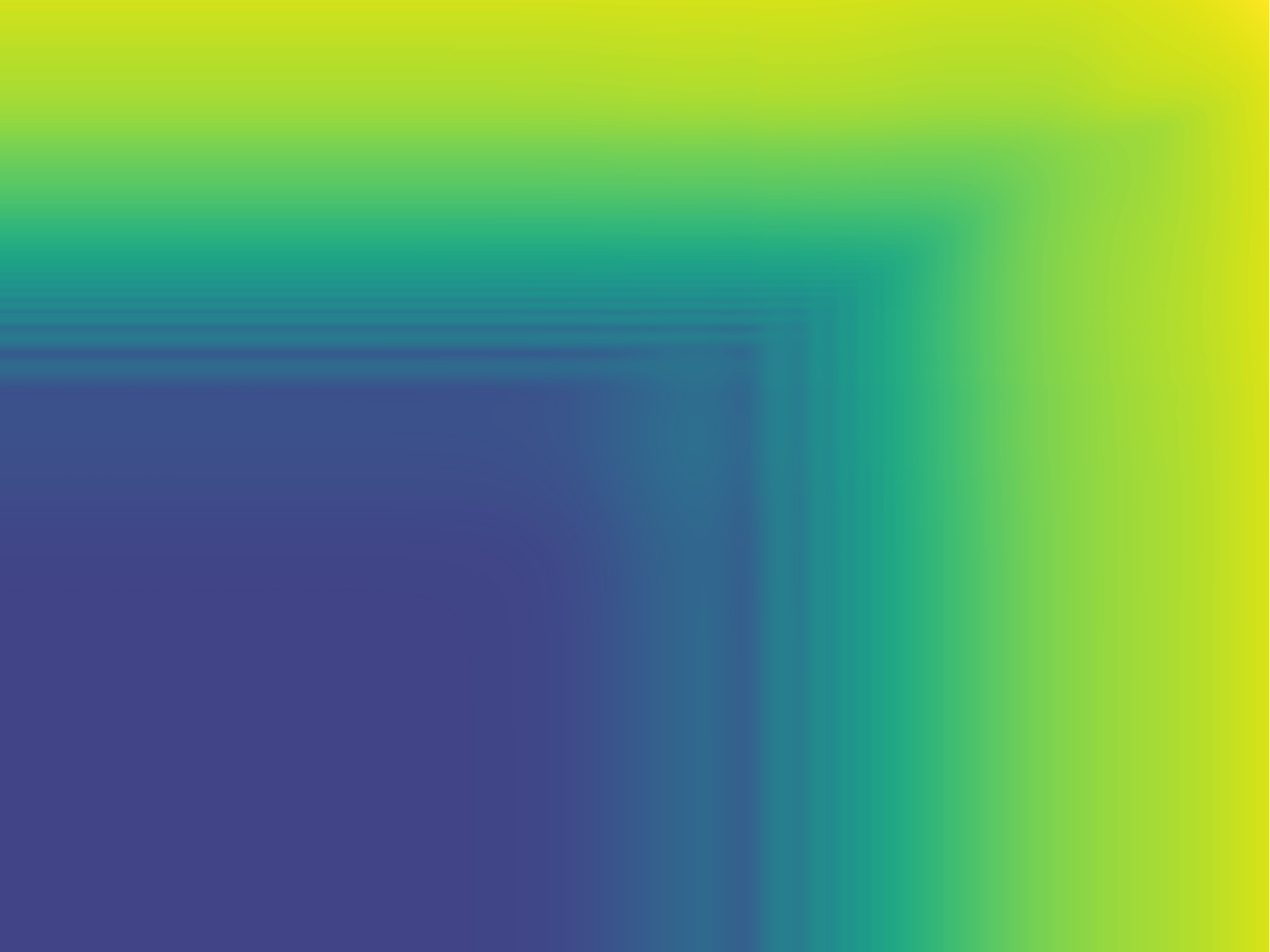};
            
  \end{loglogaxis}
\end{tikzpicture}%
  \tikzexternaldisable%

    \subcaption{\sttint{}.}
  \end{subfigure}
  \vspace{.5\baselineskip}

  \tikzexternalenable%
  \tikzsetnextfilename{time_delay_freq_g2_legend}%
  \begin{tikzpicture}
  \node[draw = none, minimum width = .1cm, inner sep = 0cm](start){};
  \node(leg) at (start.north east) [anchor = north west]{\tikz
  \begin{axis}[%
    hide axis,
    scale only axis,
    width  = 10cm,
    height = .1cm,
    point meta min = -11.4696,
    point meta max = -0.6962,
    colorbar,
    colorbar horizontal,
    colorbar style = {
      xticklabel = $10^{\pgfmathparse{\tick}
        \pgfmathprintnumber\pgfmathresult}$,
      at = {(.5, 0)},
      anchor = north},
    scaled x ticks = false,
    x tick label style = {/pgf/number format/fixed}]
  \end{axis};};
  \node[draw = none, minimum width = 0cm, inner sep = 0cm](end)
    at (leg.north east) [anchor = north west]{};
\end{tikzpicture}%
  \tikzexternaldisable%

  \caption{Frequency domain results of the second transfer functions for the 
    time-delayed heated rod.}
  \label{fig:time_delay_freq_g2}
\end{figure}
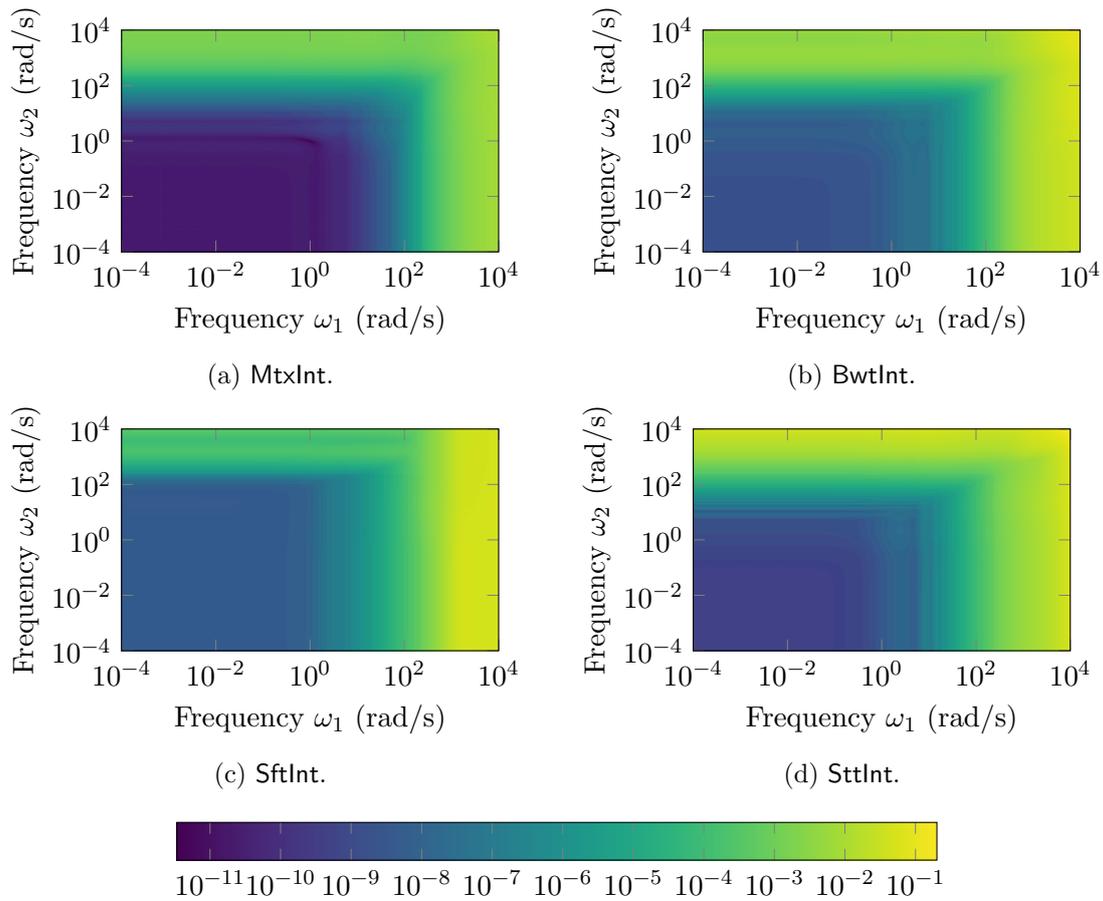

\Cref{fig:time_delay_time} shows the results in time domain for the input
signal
\begin{align*}
  u(t) & = \begin{bmatrix}
    0.05(\cos(10 t) + \cos(5 t)) &
    0.05(\sin(10 t) + \sin(5 t)) &
    0.01 &
    0.01 &
    0.01
    \end{bmatrix}^{\trans}.
\end{align*}
This time, \mtxint{} is a few orders of magnitude better than the other methods
in the overall behavior closely followed by \sttint{}, then \bwtint{} and
\sftint{}.
But in terms of the maximum errors (\Cref{tab:time_delay}), \sftint{} and
\sttint{} are almost one order of magnitude better than \mtxint{}.
The results are different in frequency domain.
\Cref{fig:time_delay_freq_g1} shows the results for the first subsystem transfer
functions.
While \bwtint{} still performs worst, \sftint{} performs now better than 
\mtxint{}, which is also shown in \Cref{tab:time_delay}.
The error of \sttint{} is mainly following \mtxint{} over the whole frequency
range and only minorly diverging at the end.
This changes for the second transfer functions in 
\Cref{fig:time_delay_freq_g2}.
Here, \mtxint{} performs best with \sttint{} having comparable accuracy.
\bwtint{} and \sftint{} are worse than the other two approaches but both with a
comparable error.
In terms of the maximum errors (\Cref{tab:time_delay}), \bwtint{} and \mtxint{}
perform the best.

Further results for tangential interpolation of a related example using
different choices of interpolation points can be found
in~\cite[Sec.~5.6.5.2]{Wer21}.

\begin{table}[t]
  \caption{Maximum relative errors for the time-delayed heated rod.}
  \label{tab:time_delay}
  \centering
  {\renewcommand{\arraystretch}{1.25}%
  \setlength{\tabcolsep}{.8em}
  \begin{tabular}{crrrr}
    \hline
    &
      \multicolumn{1}{c}{$\boldsymbol{\mtxint}$} &
      \multicolumn{1}{c}{$\boldsymbol{\bwtint}$} &
      \multicolumn{1}{c}{$\boldsymbol{\sftint}$} &
      \multicolumn{1}{c}{$\boldsymbol{\sttint}$} \\
    \hline\noalign{\medskip}
    $\err_{\rm{sim}}$ &
      $1.2251\texttt{e-}06$ &
      $1.2111\texttt{e-}05$ &
      $2.1393\texttt{e-}07$ &
      $5.2099\texttt{e-}07$ \\
    $\err_{\rm{G_{1}}}$ &
      $5.0048\texttt{e-}03$ &
      $4.2078\texttt{e-}02$ &
      $5.1565\texttt{e-}04$ &
      $3.6159\texttt{e-}02$ \\
    $\err_{\rm{G_{2}}}$ &
      $8.6292\texttt{e-}03$ &
      $8.3005\texttt{e-}02$ &
      $4.2940\texttt{e-}02$ &
      $2.0130\texttt{e-}01$ \\
    \noalign{\medskip}\hline\noalign{\smallskip}
  \end{tabular}}
\end{table}

%%%%%%%%%%%%%%%%%%%%%%%%%%%%%%%%%%%%%%%%%%%%%%%%%%%%%%%%%%%%%%%%%%%%%%%%%%%%%%%%

\subsection{Damped mass-spring system with bilinear springs}

As the third and final example, we consider the MIMO bilinear damped mass-spring
system from~\cite{BenGW21a}.
The system has a mechanical second-order structure as the 
example~\cref{eqn:bsosys} and takes the form
\begin{align} \label{eqn:msd}
  \begin{aligned}
    M \ddot{x}(t) + D \dot{x}(t) + K x(t) & =  N_{\rm{p}, 1} x(t) u_{1}(t) +
      N_{\rm{p}, 2} x(t) u_{2}(t) + B_{\rm{u}} u(t), \\
    y(t) & = C_{\rm{p}} x(t),
  \end{aligned}
\end{align}
where $M, D, K \in \R^{n \times n}$ are symmetric positive definite matrices
chosen as in~\cite{MehS05}.
The external forces are applied to the first and last masses, $B_{\rm{u}} = 
[e_{1}, -e_{n}]$, the displacement of the second and fifth
masses is observed, $C_{\rm{p}} = [e_{2}, e_{5}]^{\trans}$; thus the system has 
$m = p = 2$ inputs and outputs.
The bilinear springs are chosen to be
\begin{align*}
  \begin{aligned}
    N_{\rm{p}, 1} & = -S_{1} K S_{1} & \text{and} &&
      N_{\rm{p}, 2} & = S_{2} K S_{2},
  \end{aligned}
\end{align*}
where $S_{1}$ is a diagonal matrix with entries $\texttt{linspace(0.2,0,n)}$ and
$S_{2}$ a diagonal matrix with $\texttt{linspace(0,0.2,n)}$.
For the experiments, we chose $n = 1\,000$.

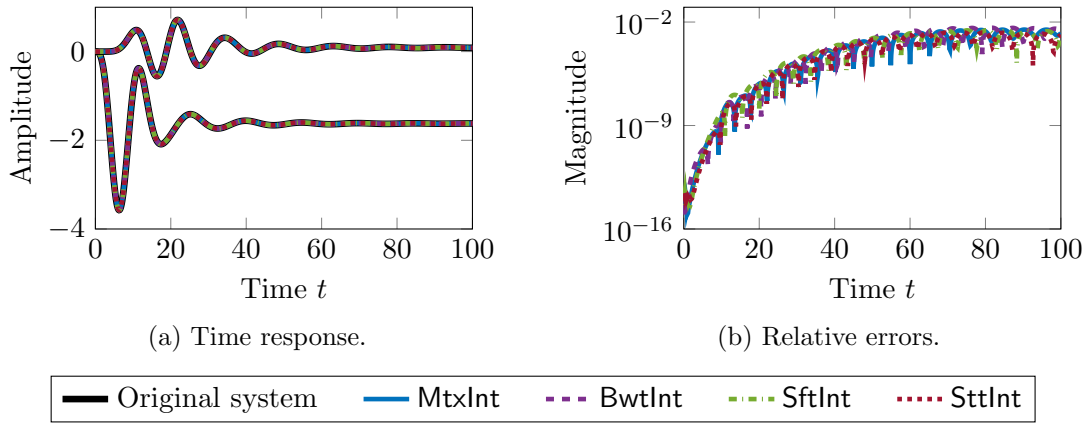
\begin{figure}[t]
  \centering
  \begin{subfigure}[b]{.49\textwidth}
    \centering
  \tikzexternalenable%
  \tikzsetnextfilename{msd_time_sim}%
  \begin{tikzpicture}
  \pgfplotstableread{graphics/data/msd_time_sim.dat}\tableSIM
  
  \begin{axis}[%
    width  = .675\textwidth,
    height = .4\textwidth,
    scale only axis,
    xmin = 0,
    xmax = 100,
    ymin = -4,
    ymax = 1,
    xminorticks = false,
    yminorticks = false,
    xlabel = {Time $t$},
    ylabel = {Amplitude},
    ylabel style   = {yshift = -.3em},
    scaled x ticks = false,
    x tick label style = {/pgf/number format/fixed},
    cycle list name    = plotlist]
    
    \foreach \p in {1, 2}{
      \foreach[evaluate=\y as \res using int(\y + \p)] \y in {0, 2, ..., 8} {
        \addplot table[x index = 0, y index = \res] {\tableSIM};
      }
    }
  \end{axis}
\end{tikzpicture}%
  \tikzexternaldisable%

    \subcaption{Time response.}
  \end{subfigure}
  \hfill
  \begin{subfigure}[b]{.49\textwidth}
    \centering
  \tikzexternalenable%
  \tikzsetnextfilename{msd_time_err}%
  \begin{tikzpicture}
  \pgfplotstableread{graphics/data/msd_time_err.dat}\tableRELERR
  
  \begin{semilogyaxis}[%
    width  = .675\textwidth,
    height = .4\textwidth,
    scale only axis,
    xmin = 0,
    xmax = 100,
    ymin = 1e-16,
    ymax = 1e-1,
    xminorticks = false,
    yminorticks = false,
    xlabel = {Time $t$},
    ylabel = {Magnitude},
    ylabel style = {yshift = -.3em},
    scaled x ticks = false,
    x tick label style = {/pgf/number format/fixed},
    cycle list name    = plotlist]
        
    \pgfplotsset{cycle list shift = 1}
    \foreach \y in {1, 2, 3, 4} {
      \addplot table[x index = 0, y index = \y] {\tableRELERR};
    }
  \end{semilogyaxis}
\end{tikzpicture}%
  \tikzexternaldisable%

    \subcaption{Relative errors.}
  \end{subfigure}
  \vspace{.5\baselineskip}

  \tikzexternalenable%
  \tikzsetnextfilename{msd_time_legend}%
  \begin{tikzpicture}  
  \begin{axis}[%
    hide axis,
    scale only axis,
    width = 1mm,
    legend columns = 5, 
    legend style = {
      at     = {(0,0)},
      anchor = center,
      /tikz/every even column/.append style = {column sep = 0.5cm}},
    cycle list name = plotlist]
    
    \pgfplotsinvokeforeach{1,...,5}{\addplot coordinates {(0,0)};}
        
    \addlegendentry{Original system};
    \addlegendentry{\mtxint};
    \addlegendentry{\bwtint};
    \addlegendentry{\sftint};
    \addlegendentry{\sttint};
  \end{axis}
\end{tikzpicture}%
  \tikzexternaldisable%

  \caption{Time domain simulation results for the damped mass-spring system.}
  \label{fig:msd_time}
\end{figure}

It has already been shown in~\cite{BenGW21a} that only the 
structure-preserving approximations give reasonable results for this example.
Therefore, we only compare the structured approaches in this paper, i.e.,
all reduced-order models also have the mechanical system structure 
as~\cref{eqn:msd}.
The reduced-order models are constructed as follows:

\begin{description}
  \item[MtxInt] with the interpolation points $\pm \texttt{logspace(-4, 4, 2)}
    \i$ for the first and second subsystem transfer functions, which yields the
    reduced order $r_{\sf{mtx}} = 24$.
  \item[BwtInt] with the interpolation points $\pm \texttt{logspace(-4, 4, 4)} 
    \i$ for the first and second subsystem transfer functions such that the
    reduced order is $r_{\sf{bwt}} = 24$.
  \item[SftInt] with the interpolation points $\pm \texttt{logspace(-4, 4, 6)} 
    \i$ and the scaling vectors $d^{(i)} = \mathds{1}_{m}$ for the first and 
    second subsystem transfer functions such that the reduced order is 
    $r_{\sf{sft}} = 24$.
  \item[SttInt] with interpolation points $\pm \texttt{logspace(-4, 4, 6)} 
    \i$ and the scaling vectors $d^{(i)} = b^{(i)}$ for the first and second 
    subsystem transfer functions such that the reduced order is
    $r_{\sf{stt}} = 24$.
\end{description}
To preserve the symmetry of the system matrices, only one-sided projections
have been used for the construction.
For all directions, normalized random vectors have been generated by drawing
their entries from a uniform distribution on $[0, 1]$.
All reduced-order models have the same order.

\Cref{fig:msd_time} shows the time simulation results for
\begin{align*}
  u(t) & = \begin{bmatrix} \sin(200 t) + 200 \\ -\cos(200 t) - 200
    \end{bmatrix}.
\end{align*}
All reduced-order models yield accurate results with practically the 
same approximation quality.
As \Cref{tab:msd} shows, the new tangential approaches perform a little bit 
better than \mtxint{} but still have the same order of accuracy.
Also, in the frequency domain, the tangential interpolation as well as the
matrix interpolation behave in principle all the same, where the matrix
interpolation is again a bit worse than the tangential approaches as it can be
seen in \Cref{fig:msd_freq_g1,fig:msd_freq_g2}, and \Cref{tab:msd}.

Further results for tangential interpolation of a related example using
different choices of interpolation points can be found
in~\cite[Sec.~5.6.5.1]{Wer21}.

\begin{figure}[t]
  \centering
  \begin{subfigure}[b]{.49\textwidth}
    \centering
  \tikzexternalenable%
  \tikzsetnextfilename{msd_freq_g1_tf}%
  \begin{tikzpicture}
  \pgfplotstableread{graphics/data/msd_freq_g1_tf.dat}\tableTF
  
  \begin{loglogaxis}[%
    width  = .675\textwidth,
    height = .4\textwidth,
    scale only axis,
    xmin = 1e-2,
    xmax = 1e+2,
    ymin = 1e-10,
    ymax = 1e+0,
    xminorticks = false,
    yminorticks = false,
    xlabel = {Frequency $\omega_{1}$ (rad/s)},
    ylabel = {Magnitude},
    ylabel style   = {yshift = -.3em},
    scaled x ticks = false,
    x tick label style = {/pgf/number format/fixed},
    cycle list name    = plotlist]
    
    \foreach \y in {1, 2, ..., 5} {
      \addplot table[x index = 0, y index = \y] {\tableTF};
    }
  \end{loglogaxis}
\end{tikzpicture}%
  \tikzexternaldisable%

    \subcaption{Frequency response.}
  \end{subfigure}
  \hfill
  \begin{subfigure}[b]{.49\textwidth}
    \centering
  \tikzexternalenable%
  \tikzsetnextfilename{msd_freq_g1_err}%
  \begin{tikzpicture}
  \pgfplotstableread{graphics/data/msd_freq_g1_err.dat}\tableRELERR
  
  \begin{loglogaxis}[%
    width  = .675\textwidth,
    height = .4\textwidth,
    scale only axis,
    xmin = 1e-2,
    xmax = 1e+2,
    ymin = 1e-20,
    ymax = 1e-2,
    xminorticks = false,
    yminorticks = false,
    xlabel = {Frequency $\omega_{1}$ (rad/s)},
    ylabel = {Magnitude},
    ylabel style   = {yshift = -.3em},
    scaled x ticks = false,
    x tick label style = {/pgf/number format/fixed},
    cycle list name    = plotlist]
    
    \pgfplotsset{cycle list shift = 1}
    \foreach \y in {1, 2, 3, 4} {
      \addplot table[x index = 0, y index = \y] {\tableRELERR};
    }
  \end{loglogaxis}
\end{tikzpicture}%
  \tikzexternaldisable%

    \subcaption{Relative errors.}
  \end{subfigure}
  \vspace{.5\baselineskip}

  \tikzexternalenable%
  \tikzsetnextfilename{msd_freq_g1_legend}%
  \begin{tikzpicture}  
  \begin{axis}[%
    hide axis,
    scale only axis,
    width = 1mm,
    legend columns = 5, 
    legend style = {
      at     = {(0,0)},
      anchor = center,
      /tikz/every even column/.append style = {column sep = 0.5cm}},
    cycle list name = plotlist]
    
    \pgfplotsinvokeforeach{1,...,5}{\addplot coordinates {(0,0)};}
        
    \addlegendentry{Original system};
    \addlegendentry{\mtxint};
    \addlegendentry{\bwtint};
    \addlegendentry{\sftint};
    \addlegendentry{\sttint};
  \end{axis}
\end{tikzpicture}%
  \tikzexternaldisable%

  \caption{Frequency domain results of the first transfer functions for the 
    damped mass-spring system.}
  \label{fig:msd_freq_g1}
\end{figure}
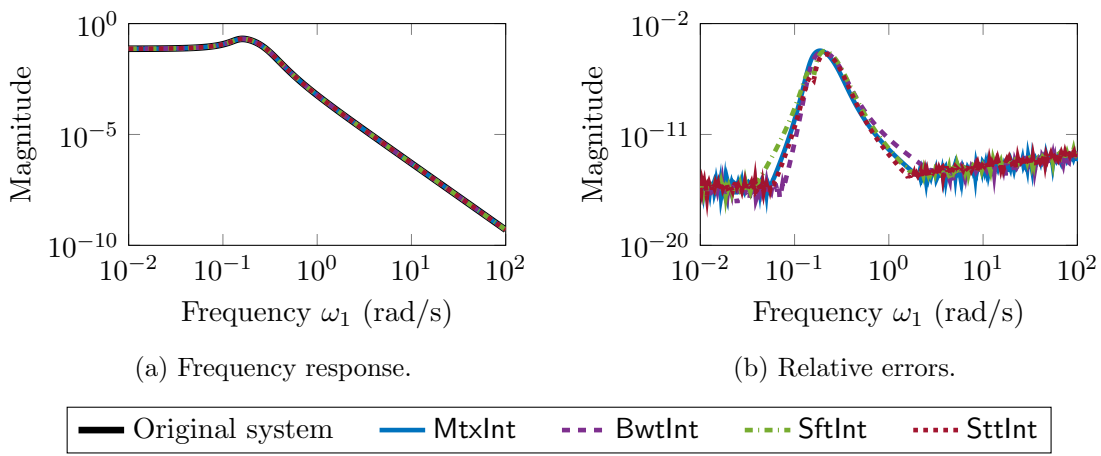

\begin{figure}[t]
  \centering
  \begin{subfigure}[b]{.49\textwidth}
    \centering
  \tikzexternalenable%
  \tikzsetnextfilename{msd_freq_g2_err_mtx}%
  \begin{tikzpicture}
  \begin{loglogaxis}[
    view   = {0}{90},
    width  = .675\textwidth,
    height = .4\textwidth,
    scale only axis,
    axis on top,
    xmin   = 1e-2,
    xmax   = 1e+2,
    ymin   = 1e-2,
    ymax   = 1e+2,
    xtick  = {1e-2, 1e-1, 1e0, 1e+1, 1e+2},
    ytick  = {1e-2, 1e-1, 1e0, 1e+1, 1e+2},
    xminorticks = false,
    yminorticks = false,
    xlabel = {Frequency $\omega_{1}$ (rad/s)},
    ylabel = {Frequency $\omega_{2}$ (rad/s)},
    ylabel style = {yshift = -.3em},
    scaled x ticks = false,
    x tick label style = {/pgf/number format/fixed}]
        
      \addplot graphics[xmin = 1e-2, xmax = 1e+2, ymin = 1e-2, ymax = 1e+2]
        {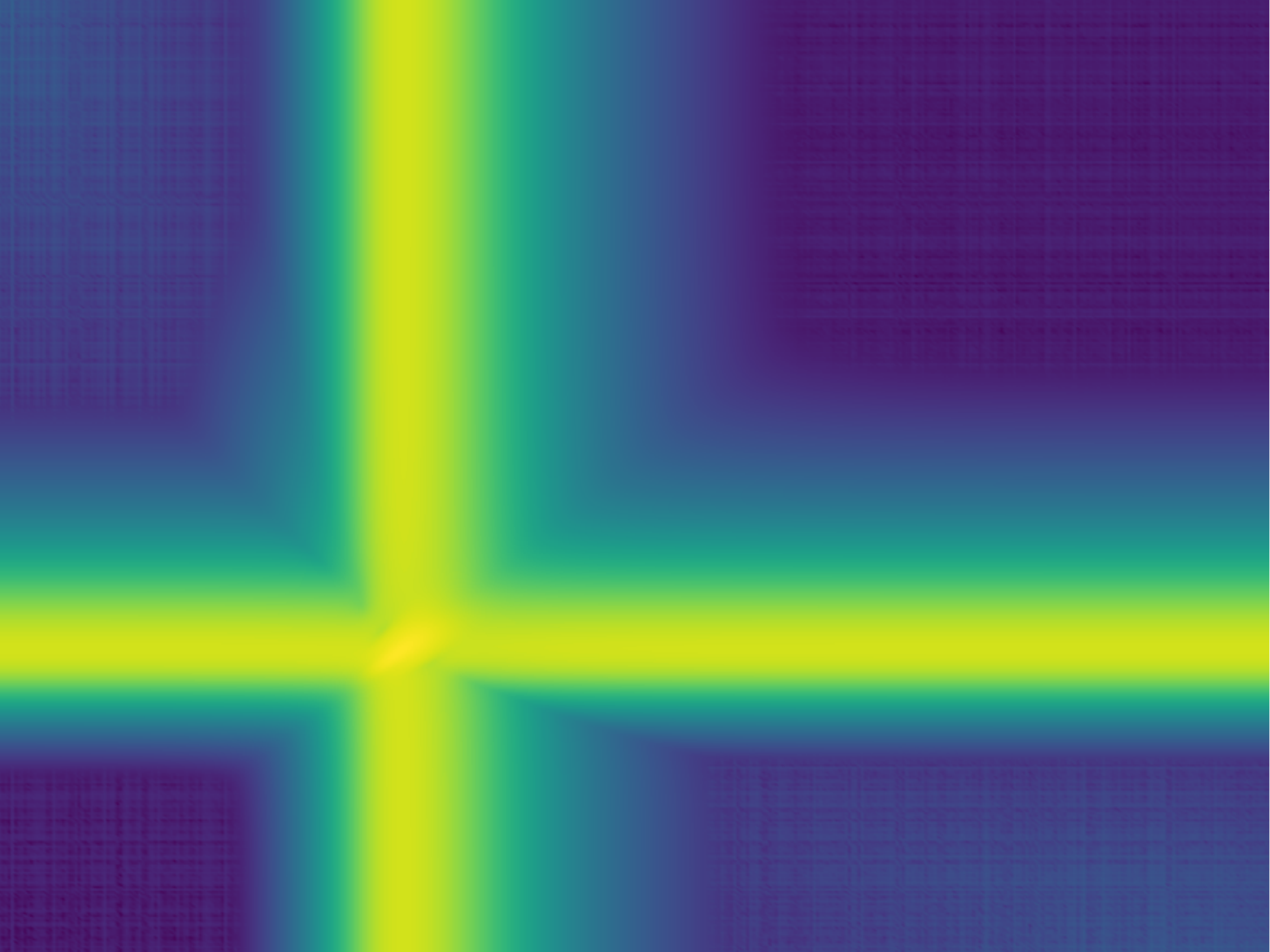};
            
  \end{loglogaxis}
\end{tikzpicture}%
  \tikzexternaldisable%

    \subcaption{\mtxint{}.}
  \end{subfigure}
  \hfill
  \begin{subfigure}[b]{.49\textwidth}
    \centering
  \tikzexternalenable%
  \tikzsetnextfilename{msd_freq_g2_err_bwt}%
  \begin{tikzpicture}
  \begin{loglogaxis}[
    view   = {0}{90},
    width  = .675\textwidth,
    height = .4\textwidth,
    scale only axis,
    axis on top,
    xmin   = 1e-2,
    xmax   = 1e+2,
    ymin   = 1e-2,
    ymax   = 1e+2,
    xtick  = {1e-2, 1e-1, 1e0, 1e+1, 1e+2},
    ytick  = {1e-2, 1e-1, 1e0, 1e+1, 1e+2},
    xminorticks = false,
    yminorticks = false,
    xlabel = {Frequency $\omega_{1}$ (rad/s)},
    ylabel = {Frequency $\omega_{2}$ (rad/s)},
    ylabel style = {yshift = -.3em},
    scaled x ticks = false,
    x tick label style = {/pgf/number format/fixed}]
        
      \addplot graphics[xmin = 1e-2, xmax = 1e+2, ymin = 1e-2, ymax = 1e+2]
        {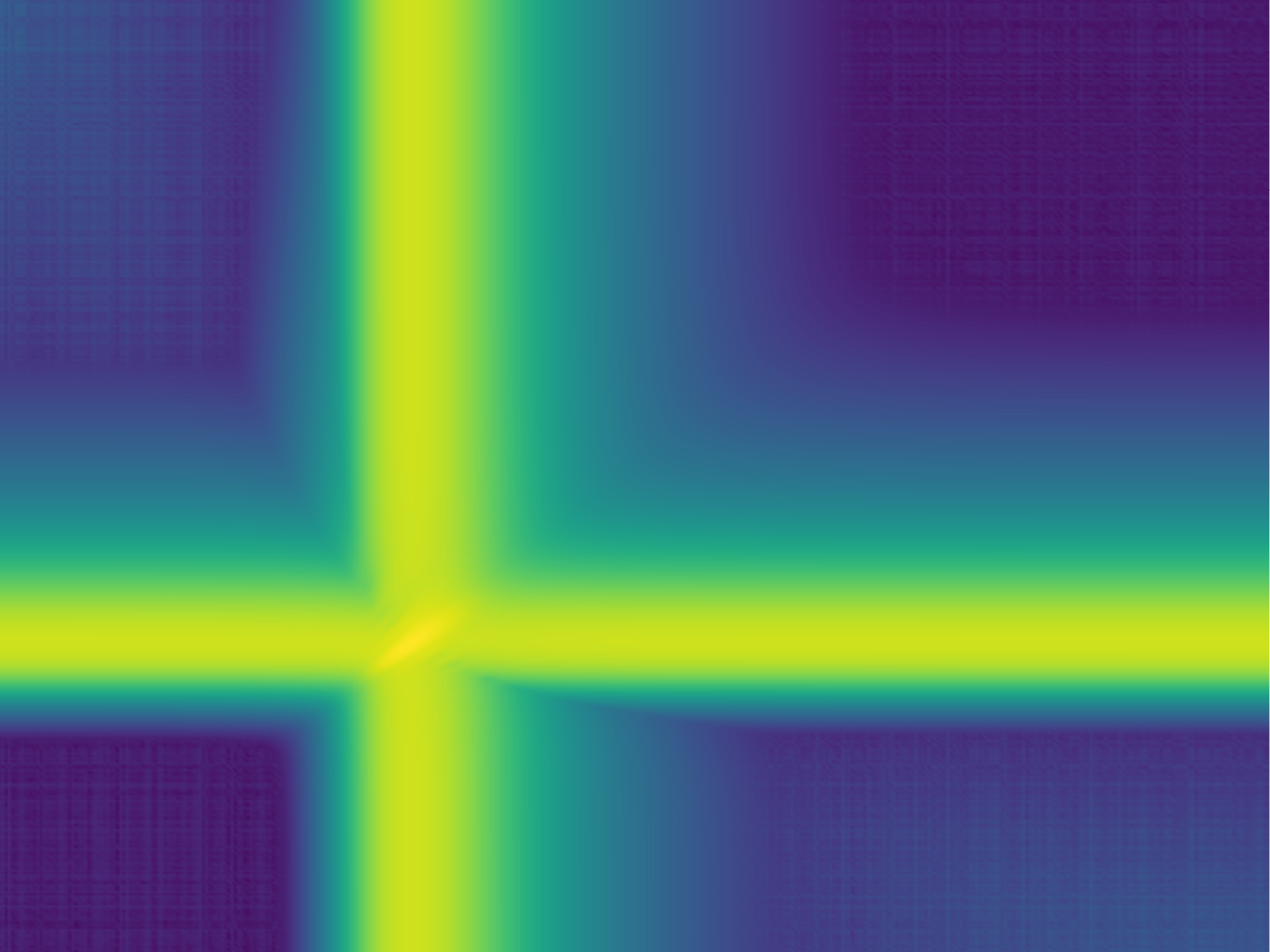};
            
  \end{loglogaxis}
\end{tikzpicture}%
  \tikzexternaldisable%

    \subcaption{\bwtint{}.}
  \end{subfigure}
    
  \begin{subfigure}[b]{.49\textwidth}
    \centering
  \tikzexternalenable%
  \tikzsetnextfilename{msd_freq_g2_err_sft}%
  \begin{tikzpicture}
  \begin{loglogaxis}[
    view   = {0}{90},
    width  = .675\textwidth,
    height = .4\textwidth,
    scale only axis,
    axis on top,
    xmin   = 1e-2,
    xmax   = 1e+2,
    ymin   = 1e-2,
    ymax   = 1e+2,
    xtick  = {1e-2, 1e-1, 1e0, 1e+1, 1e+2},
    ytick  = {1e-2, 1e-1, 1e0, 1e+1, 1e+2},
    xminorticks = false,
    yminorticks = false,
    xlabel = {Frequency $\omega_{1}$ (rad/s)},
    ylabel = {Frequency $\omega_{2}$ (rad/s)},
    ylabel style = {yshift = -.3em},
    scaled x ticks = false,
    x tick label style = {/pgf/number format/fixed}]
        
      \addplot graphics[xmin = 1e-2, xmax = 1e+2, ymin = 1e-2, ymax = 1e+2]
        {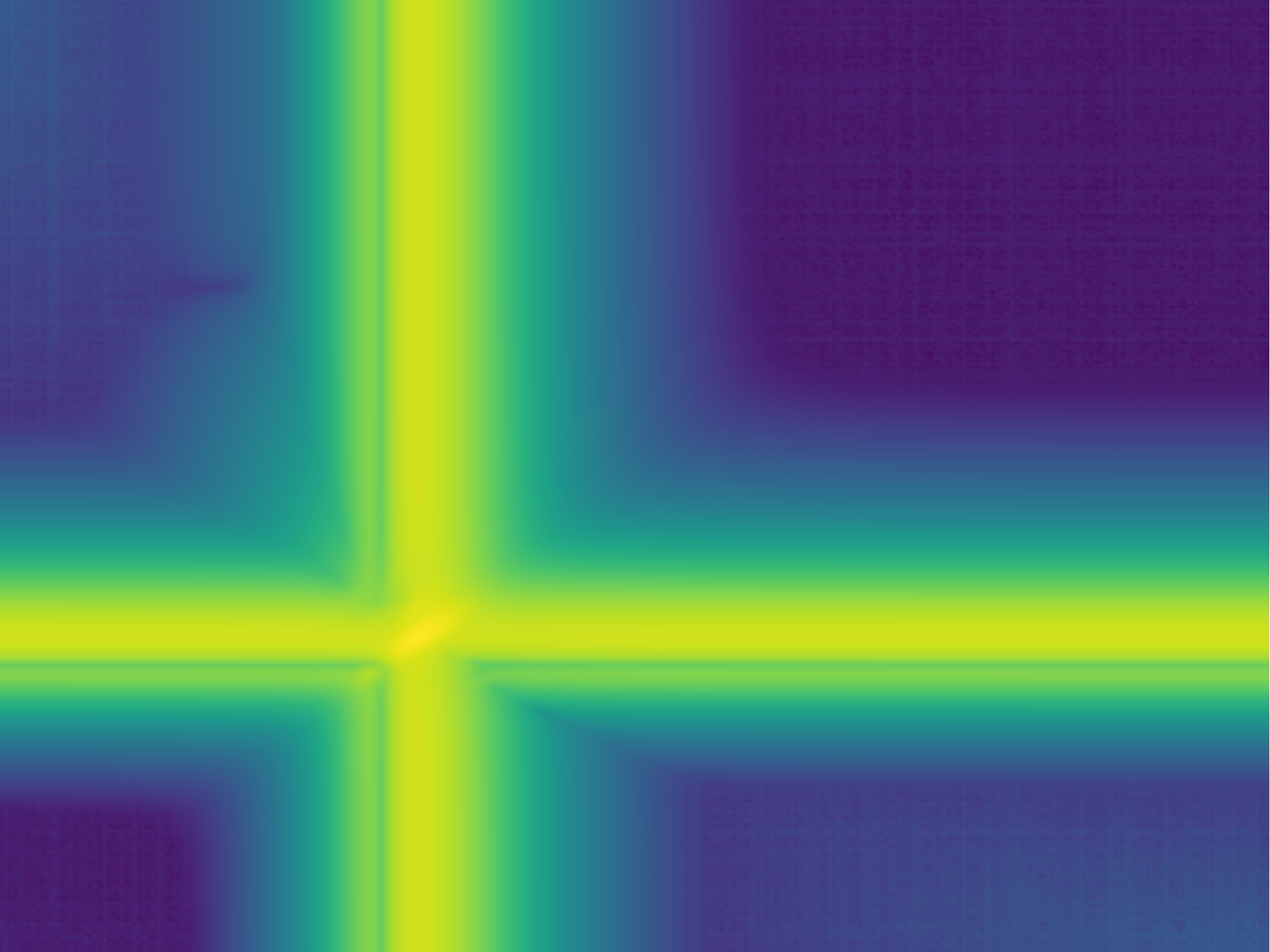};
            
  \end{loglogaxis}
\end{tikzpicture}%
  \tikzexternaldisable%

    \subcaption{\sftint{}.}
  \end{subfigure}
  \hfill
  \begin{subfigure}[b]{.49\textwidth}
    \centering
  \tikzexternalenable%
  \tikzsetnextfilename{msd_freq_g2_err_stt}%
  \begin{tikzpicture}
  \begin{loglogaxis}[
    view   = {0}{90},
    width  = .675\textwidth,
    height = .4\textwidth,
    scale only axis,
    axis on top,
    xmin   = 1e-2,
    xmax   = 1e+2,
    ymin   = 1e-2,
    ymax   = 1e+2,
    xtick  = {1e-2, 1e-1, 1e0, 1e+1, 1e+2},
    ytick  = {1e-2, 1e-1, 1e0, 1e+1, 1e+2},
    xminorticks = false,
    yminorticks = false,
    xlabel = {Frequency $\omega_{1}$ (rad/s)},
    ylabel = {Frequency $\omega_{2}$ (rad/s)},
    ylabel style = {yshift = -.3em},
    scaled x ticks = false,
    x tick label style = {/pgf/number format/fixed}]
        
      \addplot graphics[xmin = 1e-2, xmax = 1e+2, ymin = 1e-2, ymax = 1e+2]
        {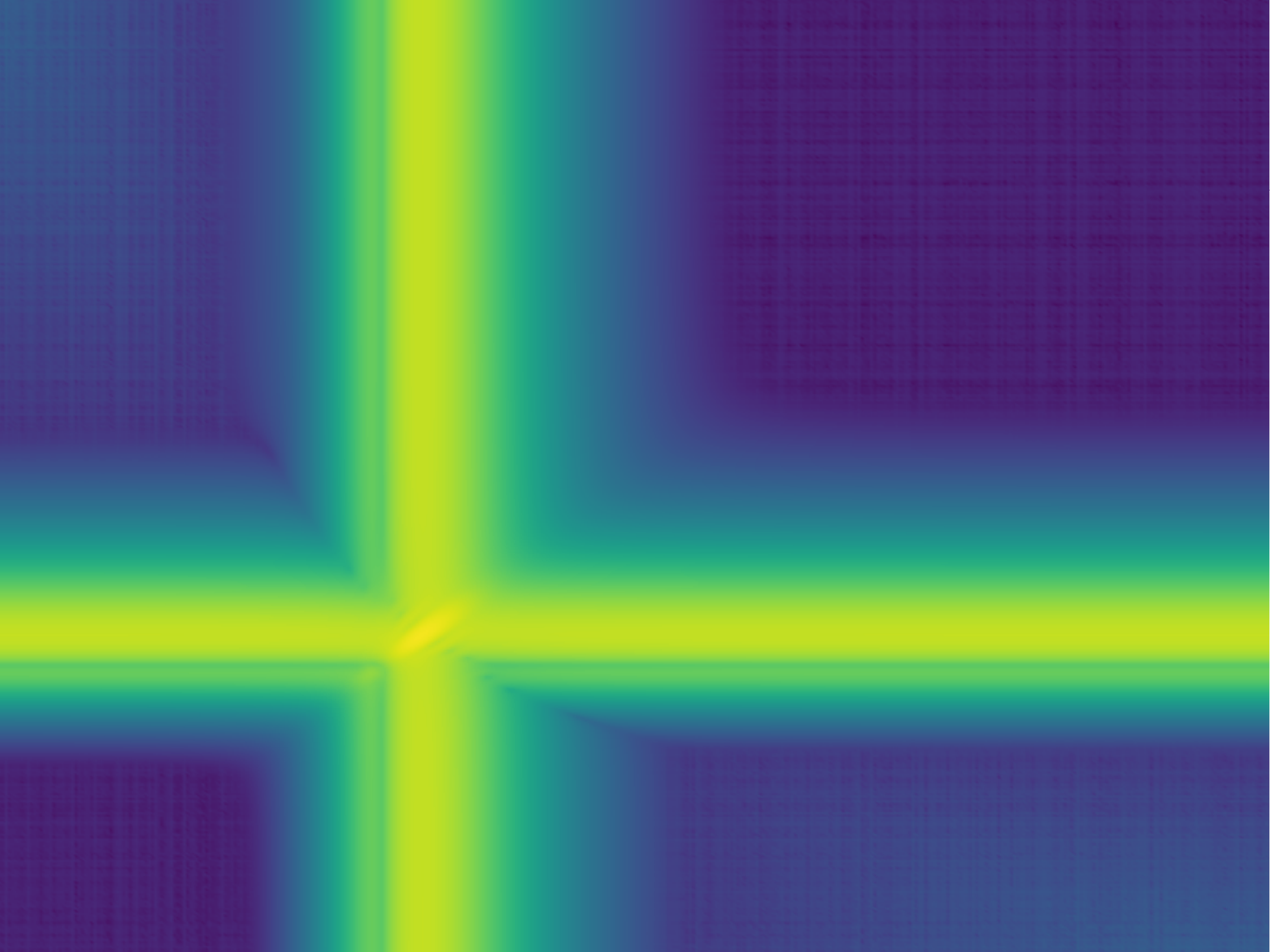};
            
  \end{loglogaxis}
\end{tikzpicture}%
  \tikzexternaldisable%

    \subcaption{\sttint{}.}
  \end{subfigure}
  \vspace{.5\baselineskip}

  \tikzexternalenable%
  \tikzsetnextfilename{msd_freq_g2_legend}%
  \begin{tikzpicture}
  \node[draw = none, minimum width = 0cm, inner sep = 0cm](start){};
  \node(leg) at (start.north east) [anchor = north west]{\tikz
  \begin{axis}[%
    hide axis,
    scale only axis,
    width  = 10cm,
    height = .1cm,
    point meta min = -16.1177,
    point meta max = -3.3418,
    colorbar,
    colorbar horizontal,
    colorbar style = {
      xticklabel = $10^{\pgfmathparse{\tick}
        \pgfmathprintnumber\pgfmathresult}$,
      at = {(.5, 0)},
      anchor = north},
    scaled x ticks = false,
    x tick label style = {/pgf/number format/fixed}]
  \end{axis};};
  \node[draw = none, minimum width = 0cm, inner sep = 0cm](end)
    at (leg.north east) [anchor = north west]{};
\end{tikzpicture}%
  \tikzexternaldisable%

  \caption{Frequency domain results of the second transfer functions for the 
    damped mass-spring system.}
  \label{fig:msd_freq_g2}
\end{figure}

\begin{table}[t]
  \caption{Maximum relative errors for the damped mass-spring system.}
  \label{tab:msd}
  \centering
  {\renewcommand{\arraystretch}{1.25}%
  \setlength{\tabcolsep}{.8em}
  \begin{tabular}{crrrr}
    \hline
    &
      \multicolumn{1}{c}{$\boldsymbol{\mtxint}$} &
      \multicolumn{1}{c}{$\boldsymbol{\bwtint}$} &
      \multicolumn{1}{c}{$\boldsymbol{\sftint}$} &
      \multicolumn{1}{c}{$\boldsymbol{\sttint}$} \\
    \hline\noalign{\medskip}
    $\err_{\rm{sim}}$ &      
      $3.0779\texttt{e-}03$ &
      $4.0813\texttt{e-}03$ &
      $2.8056\texttt{e-}03$ &
      $1.9722\texttt{e-}03$ \\
    $\err_{\rm{G_{1}}}$ &      
      $6.3187\texttt{e-}05$ &
      $5.0642\texttt{e-}05$ &
      $5.7109\texttt{e-}05$ &
      $3.2660\texttt{e-}05$ \\
    $\err_{\rm{G_{2}}}$ &
      $4.5523\texttt{e-}04$ &
      $4.3227\texttt{e-}04$ &
      $4.2240\texttt{e-}04$ &
      $2.8460\texttt{e-}04$ \\
    \noalign{\medskip}\hline\noalign{\smallskip}
  \end{tabular}}
\end{table}

%%%%%%%%%%%%%%%%%%%%%%%%%%%%%%%%%%%%%%%%%%%%%%%%%%%%%%%%%%%%%%%%%%%%%%%%%%%%%%%%
% CONCLUSIONS.                                                                 %
%%%%%%%%%%%%%%%%%%%%%%%%%%%%%%%%%%%%%%%%%%%%%%%%%%%%%%%%%%%%%%%%%%%%%%%%%%%%%%%%

\section{Conclusions}%
\label{sec:conclusions}

We developed the tangential interpolation framework for structure-preserving
interpolation of multi-input/multi-output bilinear control systems.
By revisiting the classical tangential interpolation in  frequency domain
and its interpretation in time domain, we developed a new unifying
tangential interpolation framework for structure-preserving model reduction
of MIMO bilinear systems and proved conditions on the model reduction subspaces
to satisfy interpolation conditions in this new framework.
We also used the new framework to obtain results on the blockwise tangential
interpolation approach and extended the theory from the literature to structured
bilinear systems.
While the new framework was motivated by classical tangential interpolation in
frequency domain and its interpretation in time domain, the generality
of this new approach and the corresponding theorems is that not only the
existing framework to tangential interpolation can be obtained as a special case
of the new framework but also it offers even more flexibility and options in the
model reduction procedure than explored in this paper.
The numerical examples illustrate that the new approach is as good as and even
better in many situation than the full matrix or the blockwise tangential
interpolation methods.
In other words, the new approach gives sufficiently accurate results while
allowing more freedom in choosing the order of the reduced-order model compared
to the existing approaches.

While we used a rather simple choice for interpolation points (logarithmically
equidistant on the imaginary axis), the question of better or 
even optimal choices of interpolation points remains open.
Other choices for interpolation point selections, heuristically inspired by
the linear system case, have been used in numerical examples in~\cite{Wer21}.
Also, in the setting of tangential interpolation, the question of appropriate 
tangential directions needs to be answered.
For our new framework, we gave two approaches for choosing the 
scaling vectors.
Still the influence of the choice of the scaling vectors needs to be
investigated as well as the question of an optimal choice.
These issue will be considered in future works.

%%%%%%%%%%%%%%%%%%%%%%%%%%%%%%%%%%%%%%%%%%%%%%%%%%%%%%%%%%%%%%%%%%%%%%%%%%%%%%%%
% ACKNOWLEDGEMENTS                                                             %
%%%%%%%%%%%%%%%%%%%%%%%%%%%%%%%%%%%%%%%%%%%%%%%%%%%%%%%%%%%%%%%%%%%%%%%%%%%%%%%%

\section*{Acknowledgment}%
\addcontentsline{toc}{section}{Acknowledgment}

Benner and Werner were supported by the German Research Foundation
(DFG) Research Training Group~2297 \textquotedblleft{}Mathematical Complexity
Reduction (MathCoRe)\textquotedblright{}, Magdeburg.
Gugercin was supported in parts by National Science Foundation under Grant
No.~DMS-1819110.
Part of this material is based upon work supported by the National Science 
Foundation under Grant No.~DMS-1439786 and by the Simons Foundation Grant
No.~50736 while all authors were in residence at the Institute for
Computational and Experimental Research in Mathematics in Providence, RI, during
the \textquotedblleft{}Model and dimension reduction in 
uncertain and dynamic systems\textquotedblright{} program.

Parts of this work were carried out while Werner was at the Max Planck Institute
for Dynamics of Complex Technical Systems in Magdeburg, Germany.

We would like to thank Jens Saak for providing the data for the bilinear 
semi-discretized steel profile and inspiring discussions about the
interpretation of tangential interpolation for linear systems.

%%%%%%%%%%%%%%%%%%%%%%%%%%%%%%%%%%%%%%%%%%%%%%%%%%%%%%%%%%%%%%%%%%%%%%%%%%%%%%%%
% REFERENCES                                                                   %
%%%%%%%%%%%%%%%%%%%%%%%%%%%%%%%%%%%%%%%%%%%%%%%%%%%%%%%%%%%%%%%%%%%%%%%%%%%%%%%%

\addcontentsline{toc}{section}{References}
\bibliographystyle{plainurl}
\bibliography{bibtex/myref}

\begin{thebibliography}{10}
\providecommand{\url}[1]{{#1}}
\providecommand{\urlprefix}{URL }
\expandafter\ifx\csname urlstyle\endcsname\relax
  \providecommand{\doi}[1]{DOI~\discretionary{}{}{}#1}\else
  \providecommand{\doi}{DOI~\discretionary{}{}{}\begingroup
  \urlstyle{rm}\Url}\fi

\bibitem{AlbFB93}
Al-Baiyat, S., Farag, A.S., Bettayeb, M.: Transient approximation of a bilinear
  two-area interconnected power system.
\newblock Electr. Power Syst. Res. \textbf{26}(1), 11--19 (1993).
\newblock \doi{10.1016/0378-7796(93)90064-L}

\bibitem{AntBG20}
Antoulas, A.C., Beattie, C.A., Gugercin, S.: Interpolatory Methods for Model
  Reduction.
\newblock Computational Science \& Engineering. SIAM, Philadelphia, PA (2020).
\newblock \doi{10.1137/1.9781611976083}

\bibitem{AntGI16}
Antoulas, A.C., Gosea, I.V., Ionita, A.C.: Model reduction of bilinear systems
  in the {L}oewner framework.
\newblock {SIAM} J. Sci. Comput. \textbf{38}(5), B889--B916 (2016).
\newblock \doi{10.1137/15M1041432}

\bibitem{BaiS06}
Bai, Z., Skoogh, D.: A projection method for model reduction of bilinear
  dynamical systems.
\newblock Linear Algebra Appl. \textbf{415}(2--3), 406--425 (2006).
\newblock \doi{10.1016/j.laa.2005.04.032}

\bibitem{BalGR90}
Ball, J.A., Gohberg, I., Rodman, L.: Interpolation of Rational Matrix
  Functions, \emph{Operator Theory: Advances and Applications}, vol.~45.
\newblock Birkh{\"a}user, Basel (1990).
\newblock \doi{10.1007/978-3-0348-7709-1}

\bibitem{BeaG09}
Beattie, C.A., Gugercin, S.: Interpolatory projection methods for
  structure-preserving model reduction.
\newblock Syst. Control Lett. \textbf{58}(3), 225--232 (2009).
\newblock \doi{10.1016/j.sysconle.2008.10.016}

\bibitem{BenB12}
Benner, P., Breiten, T.: Interpolation-based $\mathcal{H}_{2}$-model reduction
  of bilinear control systems.
\newblock {SIAM} J. Matrix Anal. Appl. \textbf{33}(3), 859--885 (2012).
\newblock \doi{10.1137/110836742}

\bibitem{BenBD11}
Benner, P., Breiten, T., Damm, T.: Generalized tangential interpolation for
  model reduction of discrete-time {MIMO} bilinear systems.
\newblock Int. J. Control \textbf{84}(8), 1398--1407 (2011).
\newblock \doi{10.1080/00207179.2011.601761}

\bibitem{BenD11}
Benner, P., Damm, T.: {L}yapunov equations, energy functionals, and model order
  reduction of bilinear and stochastic systems.
\newblock {SIAM} J. Control Optim. \textbf{49}(2), 686--711 (2011).
\newblock \doi{10.1137/09075041X}

\bibitem{BenGW21a}
Benner, P., Gugercin, S., Werner, S.W.R.: Structure-preserving interpolation of
  bilinear control systems.
\newblock Adv. Comput. Math. \textbf{47}(3), 43 (2021).
\newblock \doi{10.1007/s10444-021-09863-w}

\bibitem{BreD10}
Breiten, T., Damm, T.: {K}rylov subspace methods for model order reduction of
  bilinear control systems.
\newblock Syst. Control Lett. \textbf{59}(8), 443--450 (2010).
\newblock \doi{10.1016/j.sysconle.2010.06.003}

\bibitem{ConI07}
Condon, M., Ivanov, R.: Krylov subspaces from bilinear representations of
  nonlinear systems.
\newblock Compel-Int. J. Comp. Math. Electr. Electron. Eng. \textbf{26}(2),
  399--406 (2007).
\newblock \doi{10.1108/03321640710727755}

\bibitem{FenB07}
Feng, L., Benner, P.: A note on projection techniques for model order reduction
  of bilinear systems.
\newblock {AIP} Conf. Proc. \textbf{936}(1), 208--211 (2007).
\newblock \doi{10.1063/1.2790110}

\bibitem{FlaG15}
Flagg, G.M., Gugercin, S.: Multipoint {V}olterra series interpolation and
  $\mathcal{H}_{2}$ optimal model reduction of bilinear systems.
\newblock {SIAM} J. Matrix Anal. Appl. \textbf{36}(2), 549--579 (2015).
\newblock \doi{10.1137/130947830}

\bibitem{GalVV04}
Gallivan, K., Vandendorpe, A., Van~Dooren, P.: Model reduction of {MIMO}
  systems via tangential interpolation.
\newblock {SIAM} J. Matrix Anal. Appl. \textbf{26}(2), 328--349 (2004).
\newblock \doi{10.1137/S0895479803423925}

\bibitem{GosPBetal19}
Gosea, I.V., Pontes~Duff, I., Benner, P., Antoulas, A.C.: Model order reduction
  of bilinear time-delay systems.
\newblock In: Proc. of 18th European Control Conference (ECC), pp. 2289--2294
  (2019).
\newblock \doi{10.23919/ECC.2019.8796085}

\bibitem{HsuDC83}
Hsu, C.S., Desai, U.B., Crawley, C.A.: Realization algorithms and approximation
  methods of bilinear systems.
\newblock In: The 22nd {IEEE} Conference on Decision and Control, pp. 783--788
  (1983).
\newblock \doi{10.1109/CDC.1983.269628}

\bibitem{MehS05}
Mehrmann, V., Stykel, T.: Balanced truncation model reduction for large-scale
  systems in descriptor form.
\newblock In: P.~Benner, V.~Mehrmann, D.C. Sorensen (eds.) Dimension Reduction
  of Large-Scale Systems, \emph{Lect. Notes Comput. Sci. Eng.}, vol.~45, pp.
  83--115. Springer, Berlin, Heidelberg (2005).
\newblock \doi{10.1007/3-540-27909-1\_3}

\bibitem{Moh70}
Mohler, R.R.: Natural bilinear control processes.
\newblock {IEEE} Transactions on Systems Science and Cybernetics \textbf{6}(3),
  192--197 (1970).
\newblock \doi{10.1109/TSSC.1970.300341}

\bibitem{Moh73}
Mohler, R.R.: Bilinear Control Processes: With Applications to Engineering,
  Ecology and Medicine, \emph{Mathematics in Science and Engineering}, vol.
  106.
\newblock Academic Press, New York, London (1973)

\bibitem{morwiki_steel}
{Oberwolfach Benchmark Collection}: Steel profile.
\newblock Hosted at {MORwiki} -- Model Order Reduction Wiki (2005).
\newblock \urlprefix\url{http://modelreduction.org/index.php/Steel_Profile}

\bibitem{QiaZ14}
Qian, K., Zhang, Y.: Bilinear model predictive control of plasma keyhole pipe
  welding process.
\newblock J. Manuf. Sci. Eng. \textbf{136}(3), 031002 (2014).
\newblock \doi{10.1115/1.4025337}

\bibitem{RodGB18}
Rodriguez, A.C., Gugercin, S., Boggaard, J.: Interpolatory model reduction of
  parameterized bilinear dynamical systems.
\newblock Adv. Comput. Math. \textbf{44}(6), 1887--1916 (2018).
\newblock \doi{10.1007/s10444-018-9611-y}

\bibitem{Rug81}
Rugh, W.J.: Nonlinear System Theory: {T}he {V}olterra/{W}iener Approach.
\newblock The Johns Hopkins University Press, Baltimore (1981)

\bibitem{Saa03}
Saak, J.: Effiziente numerische {L}{\"o}sung eines {O}ptimalsteuerungsproblems
  f{\"u}r die {A}bk{\"u}hlung von {S}tahlprofilen.
\newblock {D}iploma thesis, Universit{\"a}t Bremen, Germany (2003).
\newblock \doi{10.5281/zenodo.1187041}

\bibitem{SaaKB21}
Saak, J., K{\"o}hler, M., Benner, P.: {M-M.E.S.S.} -- {T}he {M}atrix
  {E}quations {S}parse {S}olvers library (version 2.1) (2021).
\newblock \doi{10.5281/zenodo.4719688}.
\newblock See also:~\url{https://www.mpi-magdeburg.mpg.de/projects/mess}

\bibitem{SapSH19}
Saputra, J., Saragih, R., Handayani, D.: Robust ${H}_{\infty}$ controller for
  bilinear system to minimize {HIV} concentration in blood plasma.
\newblock J. Phys.: Conf. Ser. \textbf{1245}, 012055 (2019).
\newblock \doi{10.1088/1742-6596/1245/1/012055}

\bibitem{Wer21}
Werner, S.W.R.: Structure-preserving model reduction for mechanical systems.
\newblock {D}issertation, Otto-von-Guericke-Universit{\"a}t, Magdeburg, Germany
  (2021).
\newblock \doi{10.25673/38617}

\bibitem{supWer22c}
Werner, S.W.R.: Code, data and results for numerical experiments in ``{A}
  unifying framework for tangential interpolation of structured bilinear
  control systems'' (version 1.0) (2022).
\newblock \doi{10.5281/zenodo.5793356}

\bibitem{ZhaL02}
Zhang, L., Lam, J.: On {$H_{2}$} model reduction of bilinear systems.
\newblock Automatica J. IFAC \textbf{38}(2), 205--216 (2002).
\newblock \doi{10.1016/S0005-1098(01)00204-7}

\end{thebibliography}

\end{document}